\RequirePackage{fix-cm}
\documentclass{article}
\usepackage{iclr2027_conference,times}
\usepackage[T1]{fontenc}
\usepackage[utf8]{inputenc}
\usepackage{amsmath,amssymb,amsthm,mathtools}
\usepackage{microtype,booktabs,array,tabularx,needspace,ragged2e,enumitem}
\usepackage{graphicx,xcolor,tikz}
\usepackage{placeins,float,wrapfig}
\usetikzlibrary{arrows.meta,positioning,calc,fit}
\usepackage{etoc}
\usepackage{hyperref,url}
\hypersetup{hidelinks,pdftitle={Stochastic Subgradient Descent at Sharply Repulsive Points: A Bounded Noise Counterexample and Gaussian Noise Avoidance},pdfauthor={}}
\newtheoremstyle{named}{6pt}{6pt}{\itshape}{}{\bfseries}{.}{.5em}{\thmname{#1}\thmnumber{ #2}\thmnote{ {\normalfont(#3)}}}
\theoremstyle{named}
\newtheorem{theorem}{Theorem}[section]
\newtheorem{lemma}[theorem]{Lemma}
\newtheorem{proposition}[theorem]{Proposition}
\newtheorem{corollary}[theorem]{Corollary}
\newtheoremstyle{namedremark}{6pt}{6pt}{}{}{\bfseries}{.}{.5em}{\thmname{#1}\thmnumber{ #2}\thmnote{ {\normalfont(#3)}}}
\theoremstyle{namedremark}
\newtheorem{definition}[theorem]{Definition}
\newtheorem{remark}[theorem]{Remark}

\newcommand{\R}{\mathbb R}
\newcommand{\E}{\mathbb E}
\newcommand{\PP}{\mathbb P}
\newcommand{\one}{\mathbf 1}
\newcommand{\pc}{\partial_C}
\newcommand{\pf}{\partial_F}
\newcommand{\pl}{\partial_L}
\newcommand{\inner}[2]{\langle #1,#2\rangle}
\newcommand{\ip}[2]{\langle #1,#2\rangle}
\newcommand{\norm}[1]{\lVert #1\rVert}
\newcommand{\op}{\mathrm{op}}
\DeclareMathOperator{\conv}{conv}
\DeclareMathOperator{\Cov}{Cov}
\DeclareMathOperator{\Unif}{Unif}

\DeclareMathOperator{\tr}{tr}
\allowdisplaybreaks[1]

\iclrfinalcopy
\hypersetup{pdfauthor={Shu Li and Jiang Hu}}
\title{Stochastic Subgradient Descent\\at Sharply Repulsive Points:\\A Bounded Noise Counterexample\\and Gaussian Noise Avoidance}
\author{Shu Li\\
Qiuzhen College\\
Tsinghua University\\
Beijing, China\\
\texttt{lishu26@mails.tsinghua.edu.cn}
\And
Jiang Hu\thanks{Corresponding author.}\\
Yau Mathematical Sciences Center\\
Tsinghua University\\
Beijing, China\\
\texttt{jianghu@tsinghua.edu.cn}}
\begin{document}
\etocdepthtag.toc{main}
\maketitle
\lhead{}
\begin{abstract}
Motivated by the open question of Bianchi, Hachem, and Schechtman \citep[Remark 4]{BHS}, we show that in the non-weakly convex definable setting, omnidirectional noise with conditional fourth-moment control does not suffice for universal almost-sure avoidance of sharply repulsive critical points. We demonstrate this failure by constructing a globally Lipschitz, coercive, definable objective that is not weakly convex. For this objective and its sufficiently small linear perturbations, stochastic subgradient descent with independent noise uniformly distributed on a ball in $\R^2$ converges \emph{with probability $1$} to a nonminimal sharply repulsive critical point from every initial point in a specified ball. A two-cycle argument establishes uniform bounds on the rescaled iterates, yielding convergence to the sharply repulsive critical point.
In contrast, for locally Lipschitz definable objectives, Gaussian noise ensures simultaneous almost-sure avoidance of any finite collection of sharply repulsive points. Through a construction of difference quotient functions, we show that convergence to any such point with positive probability would produce a stationary distribution with expected descent $0$, whereas a global positive lower bound on the Fr\'echet subgradient norms of the limiting functions forces strictly positive expected descent under the same distribution. This contradiction proves avoidance, with the everywhere positive Gaussian density playing a key role.

\end{abstract}
\section{Introduction}\label{sec:introduction}
Convergence to a critical point does not ensure local optimality in nonconvex optimization. This distinction matters for stochastic methods used to train neural networks, whose parameterized losses are generally nonconvex \citep{LBH,BCN}. A central question is whether their iterates avoid nonminimal critical points, and which assumptions on the objective and the noise suffice for such avoidance.

For smooth objectives, saddle points have received considerable attention in the study of neural network optimization \citep{Dauphin}. There are now several ways to obtain avoidance or escape guarantees under suitable assumptions. Stable-manifold arguments show that gradient descent with random initialization almost surely avoids strict saddles when the update map has the required regularity \citep{Lee}. Stochastic perturbations provide another mechanism: under strict-saddle assumptions, noisy gradient methods can escape saddles and reach approximate local minima \citep{Ge}; carefully timed perturbations also yield efficient convergence to approximate second-order stationary points \citep{Jin}. These results connect local geometry with algorithmic behavior. They do not, however, make every critical point benign, and their smooth geometric arguments do not directly cover the nonsmooth operations used in many learning models.

ReLU activations, max pooling, and nonsmooth regularizers lead naturally to objectives that are both nonconvex and nonsmooth. At such points, gradients are replaced by subdifferentials, and even the meaning of criticality becomes important. Definable geometry offers a broad setting in which these objectives can be studied without allowing arbitrary oscillatory behavior. It includes semialgebraic models and, in suitable o-minimal structures, many other functions used in learning. The stochastic subgradient convergence theory of \citet{DDKL} exploits this structure to obtain guarantees of criticality. The remaining question is whether the critical points permitted by such guarantees are local minima, or whether the stochastic iteration can converge to points with lower objective values arbitrarily nearby.

The nonsmooth setting introduces an obstruction beyond the familiar smooth saddle. The Clarke subdifferential convexifies limiting slopes, so it can contain zero even when the limiting subdifferential excludes zero. A point can consequently be Clarke critical without possessing a limiting subgradient equal to zero. Sharply repulsive critical points combine this discrepancy with an active manifold on which the point is a local minimizer. For a nonminimal point of this kind, favorable behavior along the manifold does not imply local optimality in the ambient space. Nor does the word repulsive, which describes subgradient geometry, itself prove that a discrete stochastic algorithm avoids the point. The algorithm moves in finite steps, and the interaction between those steps and the noise must be analyzed explicitly.

There are positive results under additional structure. Generic convergence to local minima is known for definable weakly convex objectives \citep{BHS} and, more broadly, under subdifferential regularity \citep{DDJ}. \citet[Remark~4]{BHS} ask whether this conclusion extends beyond weak convexity to locally Lipschitz definable functions; the required avoidance concerns both active strict saddles and sharply repulsive critical points.\footnote{The question is problem~75 in the Open Problems in Operations Research collection: \url{https://ai-in-math-or.github.io/open-problems-in-or/}. The repository's \href{https://github.com/AI-in-Math-OR/open-problems-in-or/blob/main/data/llm_math_export/solution_progress/W4309864804_p0/partial_progress/75.pdf}{regular-case progress note} treats the subdifferentially regular subclass.} Related work identifies pathological Lipschitz and path-differentiable subgradient dynamics \citep{DDpath,RZ}, studies constant-step instability \citep{JL} and deterministic convergence rates \citep{CS}, and develops pathwise saddle avoidance under locally invertible sampled updates \citep{QMMZ}. These results motivate a closer examination of the nonregular case, where Clarke criticality and limiting criticality can differ.

We study the sharply repulsive part of this question by comparing bounded and Gaussian noise, focusing on what the noise distribution contributes beyond centering, moment bounds, and excitation in every direction. The results concern nonconvergence to specified targets.

\subsection{Contributions}
\textbf{A counterexample under bounded noise.}
We resolve the open question of \citet[Remark~4]{BHS} in the negative: omnidirectional noise with conditional fourth-moment control does not ensure avoidance of nonminimal sharply repulsive points. Our counterexample is an explicit globally Lipschitz, coercive, semialgebraic objective on $\R^2$: with independent noise uniform on a ball, the iterates converge to such a point with probability $1$, for every initial point in a specified ball and all sufficiently small linear perturbations of the objective (Theorem~\ref{p:thm:main} and Corollary~\ref{p:cor:coercive}). A two-cycle argument establishes uniform boundedness of the rescaled iterates under bounded noise, yielding $\norm{x_n}=O(\gamma_n)$.

\textbf{Gaussian noise avoidance through difference quotient functions.}
For locally Lipschitz definable objectives, Gaussian noise prevents convergence to any prescribed point with $0\notin\pl f(x_*)$. This holds for every adapted Clarke selection and yields avoidance of finite target sets and, for almost every linear perturbation, of all sharply repulsive points (Corollary~\ref{g:cor:finite-generic}). Our difference quotient construction transfers a local lower bound on Fr\'echet subgradient norms to a global bound for limiting functions and converts convergence with positive probability into a stationary distribution with mean descent $0$. The inherited bound and Gaussian escape force strictly positive mean descent under the same distribution, giving a contradiction.

\subsection{Setting and subdifferentials}\label{sec:setup}
For a locally Lipschitz definable $f:\R^d\to\R$, consider
\begin{equation}\label{g:eq:algorithm}
X_{n+1}=X_n-\gamma_nV_n+\gamma_n\eta_{n+1},
\qquad V_n\in\pc f(X_n).
\end{equation}
The iterates and noises are defined on a probability space $(\Omega,\mathcal F,\PP)$. Appendix~\ref{g:sec:probability} specifies the spaces and measures for all expectations below.
A function is definable if its graph belongs to a fixed o-minimal expansion of the ordered real field, with real parameters allowed \citep[Section~2]{vddm}. Such a structure is closed under Boolean operations, products, and projections, and its subsets of $\R$ are finite unions of points and intervals. Semialgebraic functions, including the construction below, are definable; Appendix~\ref{g:sec:prelim} gives the complete definition.

\begin{definition}[Subdifferentials; {\citealp[Definition~8.3 and Theorem~9.61]{RW}}]
For locally Lipschitz $f$, the Fr\'echet subdifferential consists of the vectors $v$ satisfying
$f(x+h)\ge f(x)+\ip{v}{h}+o(\norm{h})$.
The limiting subdifferential $\pl f(x)$ consists of limits of Fr\'echet subgradients at points tending to $x$. The Clarke subdifferential is $\pc f(x)=\conv\pl f(x)$.
\end{definition}
Thus $\pf f(x)\subseteq\pl f(x)\subseteq\pc f(x)$, but $0\in\pc f(x)$ need not imply $0\in\pl f(x)$.

\begin{definition}[Sharply repulsive critical point; {\citealp[Definition~6.2.4]{SchechtmanThesis}}]\label{def:sharp-main}
An active manifold $M$ at a Clarke critical point $x$ is a smooth manifold through $x$ such that $f|_M$ is smooth and, near $x$ off $M$, all Clarke subgradients have norm bounded below by a common positive constant. The point is sharply repulsive if it admits such an $M$, is a local minimizer of $f|_M$, and satisfies $0\in\pc f(x)\setminus\pl f(x)$.
\end{definition}

\section{A bounded-noise counterexample}\label{sec:counterexample}
We answer the open question of \citet[Remark~4]{BHS} in the negative: bounded noise can lead to almost-sure convergence to a nonminimal sharply repulsive point. As summarized in Figure~\ref{p:fig:proof-roadmap}, the proof combines subdifferential calculations at $0$ with uniform bounds on the rescaled iterates from the two-cycle
argument; the vanishing step size then forces convergence. A coercive modification preserves both the local geometry and these trajectories.

\subsection{The counterexample and its geometry}\label{p:sec:main-geometry}
We define the following function based on the Euclidean norm and some piecewise linear functions:  
\begin{equation}\label{p:eq:function}
f(x,y)=\sqrt{x^2+y^2}-\min\{x,y\}-2\max\{0,-x-y\}.
\end{equation}
It is easy to check that $0$ is Clarke critical without being limiting critical. For $u\in\R^2$, set
\begin{equation}\label{p:eq:tilt}
f_u(z)=f(z)-\inner{u}{z},\qquad z=(x,y).
\end{equation}
The function $f_u$ is a linear tilt of $f$, obtained by subtracting the linear function $z\mapsto\inner{u}{z}$ \citep[Section~12, p.~104]{Rockafellar1970}. We apply~\eqref{g:eq:algorithm} to $f_u$, write its iterates and subgradients as $x_n$ and $v_n$, and set $a=(1,0)$ and $b=(0,1)$.
\begin{theorem}[A bounded-noise counterexample]\label{p:thm:main}
Fix $u\in\R^2$ with $\norm{u}<1/200$. The function $f_u$ in~\eqref{p:eq:tilt} is globally Lipschitz and definable, but is not weakly convex. In~\eqref{g:eq:algorithm} applied to $f_u$, use $\gamma_n=1/(n+100)$ and independent noise $\eta_n\sim\Unif(B(0,1/200))$. Let $x_0$ be independent of the noise sequence and satisfy $x_0\in\gamma_0\overline B(a,1/8)$ almost surely. Then every Clarke subgradient selection adapted to $\mathcal F_n=\sigma(x_0,\eta_1,\ldots,\eta_n)$ gives
\[
\PP\!\left(\lim_{n\to\infty}x_n=0\right)=1,
\]
where $0$ is a sharply repulsive Clarke critical point with active manifold $M=\{0\}$ and is \mbox{not a local minimizer.}
\end{theorem}

Theorem~\ref{p:thm:main} resolves the open question of \citet[Remark~4]{BHS} \textbf{in the negative}.
The proof of Theorem~\ref{p:thm:main} is completed in Appendix~\ref{p:sec:quantifiers}, using the geometry and two-cycle argument in Appendices~\ref{p:sec:properties}--\ref{p:sec:cycle}. The probability is taken for each fixed admissible tilt. The initial region contains the nonempty open set $\gamma_0B(a,1/8)$, and the tilt region is open. Thus the failure is not confined to a null set of tilts or a single initial point.

Lemma~\ref{p:lem:function-class} establishes the required function class. Lemma~\ref{p:lem:off-origin} bounds all Clarke subgradients away from zero off $M=\{0\}$, making $M$ an active manifold. Since $f_u|_M$ is constant, $0$ minimizes this restriction. Lemmas~\ref{p:lem:clarke-criticality} and~\ref{p:lem:sharp-repulsiveness} establish Clarke criticality, exclude limiting criticality, and give a descending ray:
\[
\pf f_u(0)=\varnothing,\qquad
0\in\pc f_u(0)\setminus\pl f_u(0),\qquad
f_u(-t,-t)<f_u(0)\quad(t>0).
\]
These conclusions identify $0$ as a nonminimal sharply repulsive critical point. The remaining task is to prove convergence to $0$.

\subsection{Convergence via rescaled iterates}\label{p:sec:main-trapping}
Set $z_n=x_n/\gamma_n$. We use a two-cycle argument to keep these rescaled iterates alternately in two bounded regions for every admissible noise sequence. Returning to the original scale gives $\norm{x_n}=O(\gamma_n)\to0$.
\begin{lemma}[Two-cycle argument for the rescaled iterates]\label{p:lem:normalized-trapping}
Fix $\norm{u}<1/200$, $x_0\in\gamma_0\overline B(a,1/8)$, and any noise sequence with $\norm{\eta_n}\le1/200$ for all $n\ge1$. For $\gamma_n=1/(n+100)$, \mbox{the rescaled iterates satisfy}
\begin{equation}
z_{2j}\in\overline B(a,1/8),\qquad
z_{2j+1}\in B(b,3/20),\qquad j\ge0.
\label{p:eq:invariant}
\end{equation}
In particular, $\norm{z_n}\le23/20$, and every physical iterate lies in a smooth cone, so the Clarke subgradient along the trajectory is unique.
\end{lemma}

The proof of the two-cycle argument Lemma~\ref{p:lem:normalized-trapping} is in Appendix~\ref{p:sec:cycle}.
Together these results give $\norm{x_n}\le 23/[20(n+100)]\to0$.
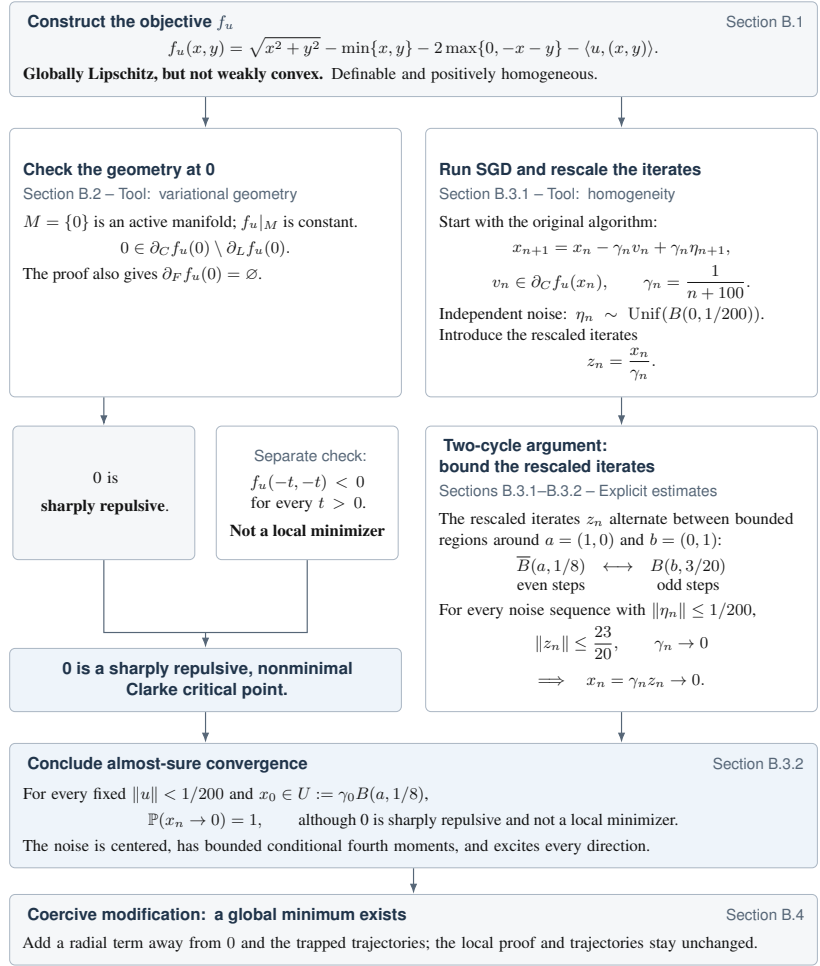
\begin{figure}[!t]
\centering
\resizebox{.80\linewidth}{!}{
\definecolor{RMInk}{HTML}{20344A}
\definecolor{RMMeta}{HTML}{566677}
\definecolor{RMLine}{HTML}{AAB6C3}
\definecolor{RMPale}{HTML}{F5F7F9}
\definecolor{RMBlue}{HTML}{EEF4FA}
\providecommand{\RMtitle}[1]{{\sffamily\bfseries\color{RMInk}#1}}
\providecommand{\RMmeta}[1]{{\sffamily\fontsize{8.5}{10.5}\selectfont\color{RMMeta}#1}}
\begin{tikzpicture}[
  box/.style={draw=RMLine,rounded corners=3pt,line width=.55pt,
    fill=white,text=black!92,align=left,text width=7.1cm,inner sep=7.5pt,
    font={\fontsize{9.2}{11.4}\selectfont
      \setlength{\parindent}{0pt}\RaggedRight
      \hyphenpenalty=10000\exhyphenpenalty=10000
      \setlength{\abovedisplayskip}{3pt}\setlength{\belowdisplayskip}{3pt}
      \setlength{\abovedisplayshortskip}{3pt}\setlength{\belowdisplayshortskip}{3pt}}},
  wide/.style={box,text width=15.2cm},
  smallbox/.style={box,text width=3.05cm,inner sep=7pt,align=center,
    minimum height=2.6cm},
  arrow/.style={-{Latex[length=2mm,width=1.45mm]},draw=RMMeta,line width=.65pt}
]
\node[wide,fill=RMPale] (construct) {
  \RMtitle{Construct the objective $f_u$}\hfill\RMmeta{Section~\ref{p:sec:properties}}
  \[
    f_u(x,y)=\sqrt{x^2+y^2}-\min\{x,y\}-2\max\{0,-x-y\}
    -\langle u,(x,y)\rangle.
  \]
  \textbf{Globally Lipschitz, but not weakly convex.}
  Definable and positively homogeneous.
};
\node[box,anchor=north] (geometry)
  at ([xshift=-4.05cm,yshift=-6mm]construct.south) {
  \begin{minipage}[t][4.3cm][t]{\linewidth}\RaggedRight
  \RMtitle{Check the geometry at 0}\\[2pt]
  \RMmeta{Section~\ref{p:sec:geometry} -- Tool: variational geometry}\\[4pt]
  $M=\{0\}$ is an active manifold; $f_u|_M$ is constant.
  \[
    0\in\partial_C f_u(0)\setminus\partial_L f_u(0).
  \]
  The proof also gives $\partial_F f_u(0)=\varnothing$.
  \end{minipage}
};
\node[smallbox,anchor=north,fill=RMPale] (sharp)
  at ([xshift=-1.98cm,yshift=-5.5mm]geometry.south) {
  0 is\\[4pt]\textbf{sharply repulsive}.
};
\node[smallbox,anchor=north] (nonminimum)
  at ([xshift=1.98cm,yshift=-5.5mm]geometry.south) {
  \RMmeta{Separate check:}\\[3pt]
  $f_u(-t,-t)<0$\\
  for every $t>0$.\\[4pt]
  \textbf{Not a local minimizer}
};

\node[box,anchor=north] (rescale)
  at ([xshift=4.05cm,yshift=-6mm]construct.south) {
  \begin{minipage}[t][4.3cm][t]{\linewidth}\RaggedRight
  \RMtitle{Run SGD and rescale the iterates}\\[2pt]
  \RMmeta{Section~\ref{p:sec:cycle} -- Tool: homogeneity}\\[4pt]
  Start with the original algorithm:
  \[
  \begin{gathered}
  x_{n+1}=x_n-\gamma_n v_n+\gamma_n\eta_{n+1},\\
  v_n\in\partial_C f_u(x_n),\qquad\gamma_n=\frac1{n+100}.
  \end{gathered}
  \]
  Independent noise: $\eta_n\sim\operatorname{Unif}(B(0,1/200))$.
  Introduce the rescaled iterates
  \[z_n=\frac{x_n}{\gamma_n}.\]
  \end{minipage}
};
\node[box,below=5.5mm of rescale,minimum height=4.2cm] (trap) {
  \RMtitle{Two-cycle argument:\\bound the rescaled iterates}\\[2pt]
  \RMmeta{Sections~\ref{p:sec:cycle}--\ref{p:sec:quantifiers} -- Explicit estimates}\\[4pt]
  The rescaled iterates $z_n$ alternate between bounded regions around $a=(1,0)$ and $b=(0,1)$:
  \[
  \begin{array}{c@{\quad}c@{\quad}c}
    \overline B(a,1/8) & \longleftrightarrow & B(b,3/20)\\[-1pt]
    \text{even steps} & & \text{odd steps}
  \end{array}
  \]
  For every noise sequence with $\|\eta_n\|\le1/200$,
  \[
  \begin{gathered}
    \|z_n\|\le\frac{23}{20},\qquad\gamma_n\to0\\[2pt]
    \Longrightarrow\quad x_n=\gamma_nz_n\to0.
  \end{gathered}
  \]
};
\node[box,align=center,anchor=south,minimum height=1.15cm,fill=RMBlue] (badpoint)
  at (geometry.center |- trap.south) {
  \RMtitle{0 is a sharply repulsive, nonminimal\\Clarke critical point.}
};

\node[wide,fill=RMBlue,below=6mm of trap,xshift=-4.05cm] (probability) {
  \RMtitle{Conclude almost-sure convergence}\hfill\RMmeta{Section~\ref{p:sec:quantifiers}}\\[5pt]
  For every fixed $\|u\|<1/200$ and $x_0\in U:=\gamma_0B(a,1/8)$,
  \[
    \mathbb P(x_n\to0)=1,
    \qquad\text{although 0 is sharply repulsive and not a local minimizer.}
  \]
  The noise is centered, has bounded conditional fourth moments, and excites every direction.
};
\node[wide,fill=RMPale,below=5mm of probability] (coercive) {
  \RMtitle{Coercive modification: a global minimum exists}\hfill\RMmeta{Section~\ref{p:sec:coercive}}\\[4pt]
  Add a radial term away from $0$ and the trapped trajectories;
  the local proof and trajectories stay unchanged.
};

\draw[arrow] ([xshift=-4.05cm]construct.south) -- (geometry.north);
\draw[arrow] ([xshift=4.05cm]construct.south) -- (rescale.north);
\draw[arrow] ([xshift=-1.98cm]geometry.south) -- (sharp.north);
\coordinate (leftmerge) at ($(badpoint.north)+(0,3mm)$);
\draw[draw=RMMeta,line width=.65pt]
  (sharp.south) -- (sharp.south |- leftmerge) -- (leftmerge)
  -- (nonminimum.south |- leftmerge) -- (nonminimum.south);
\draw[arrow] (leftmerge) -- (badpoint.north);
\draw[arrow] (rescale.south) -- (trap.north);
\draw[arrow] (badpoint.south) -- ([xshift=-4.05cm]probability.north);
\draw[arrow] (trap.south) -- ([xshift=4.05cm]probability.north);
\draw[arrow] (probability.south) -- (coercive.north);
\end{tikzpicture}}
\caption{Proof of the bounded-noise counterexample. Left: $0$ is sharply repulsive but not a local minimizer. Right: a two-cycle argument keeps the rescaled iterates bounded, so the original iterates converge to $0$ almost surely. A coercive modification preserves these trajectories.}
\label{p:fig:proof-roadmap}
\end{figure}

Lemma~\ref{p:lem:normalized-trapping} thus gives convergence to the nonminimal point $0$ for $f_u$. However, $f_u$ is not coercive. To obtain a coercive counterexample with the same trajectories, we add a radial penalty that vanishes near $0$ and all trapped iterates. Specifically, set
\[
\widehat f_u(z)=f_u(z)+10\max\{0,\norm{z}-1/10\}.
\]
\begin{corollary}[The counterexample can be coercive]\label{p:cor:coercive}
The function $\widehat f_u$ is globally Lipschitz, definable, coercive, and attains a global minimum. It is not weakly convex. Under the step-size and noise assumptions of Theorem~\ref{p:thm:main}, every initial point in $\gamma_0\overline B(a,1/8)$ produces the same trapped trajectory for $\widehat f_u$ as for $f_u$. The point $0$ remains sharply repulsive and nonminimizing.
\end{corollary}

The complete proof is in Appendix~\ref{p:sec:coercive}. The radial term vanishes on a neighborhood of every trapped iterate, so the local geometry and trajectories coincide. The full arguments also record the extensions to power steps in $(1/2,1]$.

\Needspace{14\baselineskip}
\section{Gaussian noise avoidance}\label{sec:gaussian}\label{g:sec:mainargument}
In this section, we show that simply requiring the noise to have an everywhere positive density, e.g., Gaussian noise, ensures avoidance of the sharply repulsive points considered in the open question of \citet[Remark~4]{BHS}. For locally Lipschitz definable objectives, the following theorem rules out convergence to any prescribed $x_*$ with $0\notin\pl f(x_*)$ for every adapted Clarke subgradient selection.
\begin{theorem}[Gaussian avoidance of a specified point]\label{g:thm:main}
Let $f:\R^d\to\R$ be locally Lipschitz and definable, and fix
$x_*\in\R^d$ with $0\notin\pl f(x_*)$. In \eqref{g:eq:algorithm}, let
$X_0$ be finite almost surely and let $X_n,V_n$ be adapted to a
filtration $\mathcal F_n$. Suppose $\eta_{n+1}$ is independent of
$\mathcal F_n$ and the noise vectors are identically distributed as
$N(0,\Sigma)$ for a fixed positive-definite matrix $\Sigma$. If
\begin{equation}\label{g:eq:steps}
 \gamma_n=\kappa(n+n_0)^{-\alpha},
 \qquad \kappa,n_0>0,\qquad \tfrac12<\alpha\le1,
\end{equation}
then $\PP(X_n\to x_*)=0$ for every adapted choice
$V_n\in\pc f(X_n)$.
\end{theorem}

Under hypothetical convergence, summing decreases of a Lipschitz extension $h$ over late stopped windows yields vanishing mean descent for the distributions of actual difference quotient states. Compactness and continuity give a limit with mean descent $0$, and a separate state-test identity proves its stationarity. The inherited subgradient bound and Gaussian escape force positive mean descent under this same distribution, a contradiction (Figure~\ref{g:fig:route}).
\subsection{Difference quotient functions}\label{g:sec:rescaled}
We analyze objective decreases at the scale of each update. Translate the target to $0$ and subtract $f(x_*)$, so $f(0)=0$. Since Gaussian updates can leave the local neighborhood, first extend $f$ globally: for sufficiently small $R>0$ and a Lipschitz constant $L$ on $\overline B(0,R)$, set
\begin{equation}\label{g:eq:extension}
 h(x)=\inf_{y\in\overline B(0,R)}\{f(y)+L\norm{x-y}\}.
\end{equation}
Then $h=f$ on $\overline B(0,R)$ and $h$ is globally $L$-Lipschitz (Lemma~\ref{g:fact:extension}). At center $x$ and scale $s>0$, define
\begin{equation}\label{g:eq:profile}
 H_{x,s}(u)=\frac{h(x+su)-h(x)}s,\qquad u\in\R^d.
\end{equation}
Here $u$ measures displacement in units of $s$. An update $x^+=x+s(-v+\eta)$ \mbox{therefore has scaled decrease}
\[
 \frac{h(x)-h(x^+)}s=H_{x,s}(0)-H_{x,s}(-v+\eta).
\]
The uniform Lipschitz bound and $H_{x,s}(0)=0$ give locally uniform subsequential limits. Let $\mathcal C$ consist of all limits obtained as $x_j\to0$ and $s_j\downarrow0$:
\begin{equation}\label{g:eq:profilelimit}
 \sup_{\norm{u}\le M}|H_{x_j,s_j}(u)-H(u)|\longrightarrow0
 \qquad\text{for every }M>0.
\end{equation}
The extension $h$ remains fixed; the limit $H$ may depend on the centers and scales. Lemma~\ref{g:lem:profilecompactness} proves compactness of $\mathcal C$.
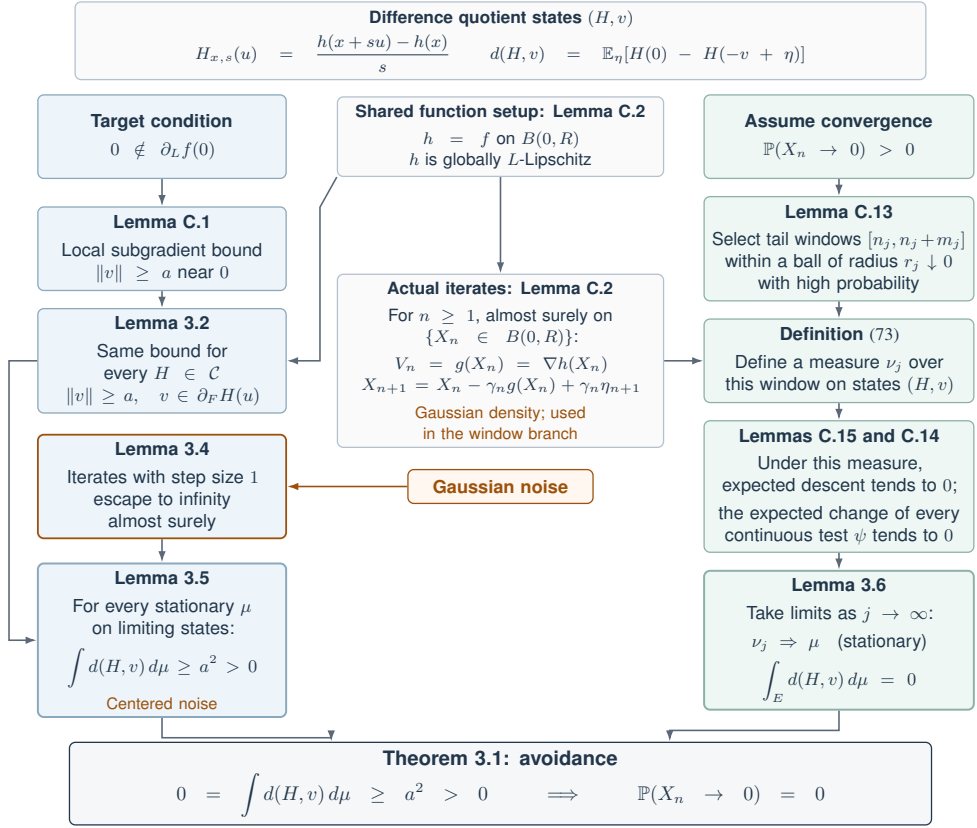
\begin{figure}[!t]
\centering
\begingroup
\definecolor{Ink}{HTML}{21364B}
\definecolor{Blue}{HTML}{2C6085}
\definecolor{BlueFill}{HTML}{EDF4FA}
\definecolor{Green}{HTML}{2F7263}
\definecolor{GreenFill}{HTML}{EEF7F3}
\definecolor{Amber}{HTML}{9B510B}
\definecolor{Muted}{HTML}{5B6877}
\definecolor{Line}{HTML}{B5C1CE}
\newcommand{\RMhead}[1]{{\bfseries\color{Ink}#1}\\[3pt]}
\newcommand{\RMnoise}[1]{\\[4pt]{\color{Amber}\fontsize{8}{9.5}\selectfont #1}}
\resizebox{.92\textwidth}{!}{%
\begin{tikzpicture}[
 font=\sffamily\fontsize{9}{11}\selectfont,
 box/.style={draw=Line,rounded corners=3pt,line width=.65pt,
   fill=white,text=Ink,align=center,text width=4.12cm,
   minimum height=1.48cm,inner sep=4pt},
 geometry/.style={box,draw=Blue!55,fill=BlueFill},
 convergence/.style={box,draw=Green!55,fill=GreenFill,text width=4.5cm},
 centerbox/.style={box,fill=black!2,text width=5.5cm,
   font=\sffamily\fontsize{8.5}{10.5}\selectfont},
 seed/.style={box,minimum height=.95cm},
 result/.style={minimum height=1.75cm,line width=1pt},
 arrow/.style={-{Latex[length=2mm,width=1.35mm]},draw=Muted,line width=.8pt}
]
\node[centerbox,text width=14.8cm,minimum height=1.15cm] (notation) at (6,1.7)
 {\RMhead{Difference quotient states $(H,v)$}
 $\displaystyle H_{x,s}(u)=\frac{h(x+su)-h(x)}s$
 \qquad
 $d(H,v)=\E_\eta[H(0)-H(-v+\eta)]$};

\node[seed,geometry] (target) at (0,0)
 {\RMhead{Target condition}$0\notin\pl f(0)$};
\node[seed,convergence] (assume) at (12,0)
 {\RMhead{Assume convergence}$\PP(X_n\to0)>0$};
\node[centerbox,minimum height=1.15cm] (shared) at (6,0)
 {\RMhead{Shared function setup: Lemma~\ref{g:lem:extension}}
 $h=f$ on $B(0,R)$\\
 $h$ is globally $L$-Lipschitz};

\node[convergence,below=.30cm of assume] (window)
 {\RMhead{Lemma~\ref{g:lem:window}}
 Select tail windows $[n_j,n_j+m_j]$\\
 within a ball of radius $r_j\downarrow0$\\
 with high probability};
\node[geometry] (gap) at (target |- window)
 {\RMhead{Lemma~\ref{g:lem:localgap}}Local subgradient bound\\$\norm{v}\ge a$ near $0$};

\node[convergence,below=.30cm of window] (measure)
 {\RMhead{Definition~\eqref{g:eq:occupation}}
 Define a measure $\nu_j$ over\\
 this window on states $(H,v)$};

\node[centerbox,minimum height=1.8cm] (local) at (shared |- measure)
 {\RMhead{Actual iterates: Lemma~\ref{g:lem:extension}}
 For $n\ge1$, almost surely on\\
 $\{X_n\in B(0,R)\}$:\\[3pt]
 $V_n=g(X_n)=\nabla h(X_n)$\\
 $X_{n+1}=X_n-\gamma_ng(X_n)+\gamma_n\eta_{n+1}$
 \RMnoise{Gaussian density; used in the window branch}};

\node[geometry] (inherit) at (target |- measure)
 {\RMhead{Lemma~\ref{g:lem:profilegap}}Same bound for every $H\in\mathcal C$\\$\norm{v}\ge a,\quad v\in\pf H(u)$};
\node[convergence,below=.30cm of measure] (averages)
 {\RMhead{Lemmas~\ref{g:lem:occupation-descent} and~\ref{g:lem:occupation}}
 Under this measure,\\
 expected descent tends to $0$;\\[3pt]
 the expected change of every\\
 continuous test $\psi$ tends to $0$};
\node[geometry,draw=Amber,line width=.9pt] (escape) at (target |- averages)
 {\RMhead{Lemma~\ref{g:lem:fixed}}Iterates with step size $1$\\escape to infinity\\almost surely};
\node[box,draw=Amber,fill=orange!6,text width=3cm,minimum height=.6cm,
 text=Amber,font=\sffamily\bfseries\fontsize{9}{11}\selectfont]
 (gaussian) at (shared |- averages) {Gaussian noise};

\node[convergence,result,below=.30cm of averages] (limit)
 {\RMhead{Lemma~\ref{g:prop:zero}}
 Take limits as $j\to\infty$:\\[3pt]
 $\displaystyle\nu_{j}\Rightarrow\mu\quad\text{(stationary)}$\\[4pt]
 $\displaystyle\int_E d(H,v)\,d \mu=0$};
\node[geometry,result] (positive) at (target |- limit)
 {\RMhead{Lemma~\ref{g:prop:negative}}For every stationary $\mu$\\on limiting states:\\[4pt]
  $\displaystyle\int d(H,v)\,d\mu\ge a^2>0$\RMnoise{Centered noise}};

\node[box,fill=Ink!4,draw=Ink,text width=15.0cm,minimum height=1.3cm,
 font=\sffamily\fontsize{10}{12}\selectfont,
 below=.55cm of limit,xshift=-6cm] (conclusion)
 {\RMhead{Theorem~\ref{g:thm:main}: avoidance}
  $\displaystyle 0=\int d(H,v)\,d\mu\ge a^2>0
   \qquad\Longrightarrow\qquad \PP(X_n\to0)=0$};

\draw[arrow] (shared.south)--(local.north);
\draw[arrow] (shared.south west)--(2.8,-1.25)--(2.8,0 |- inherit)--(inherit.east);
\draw[arrow] (local.east)--(measure.west);
\draw[arrow] (target)--(gap);
\draw[arrow] (gap)--(inherit);
\draw[arrow] (assume)--(window);
\draw[arrow] (window)--(measure);
\draw[arrow] (measure)--(averages);
\draw[arrow] (averages)--(limit);
\draw[arrow] (inherit.west)--++(-.5,0)|-(positive.west);
\draw[arrow,draw=Amber] (gaussian.west)--(escape.east);
\draw[arrow] (escape)--(positive);
\draw[arrow] (positive.south)--++(0,-.35)-|([xshift=-3cm]conclusion.north);
\draw[arrow] (limit.south)--++(0,-.35)-|([xshift=3cm]conclusion.north);
\end{tikzpicture}
}
\endgroup
\caption{Proof of Gaussian noise avoidance via difference quotient functions,
which capture objective changes at a different scale.
Left: $0\notin\pl f(0)$ yields a subgradient bound inherited by all limiting
functions; Gaussian escape then forces positive mean descent under every
stationary distribution.
Right: hypothetical convergence yields such a distribution with zero mean
descent, a contradiction.}
\label{g:fig:route}
\end{figure}

To pass to limits of the recorded pairs $(H_{X_n,s},V_n)$, we use states $(H,v)\in\mathcal C\times\overline B(0,L)$. We record $v$ alongside $H$ because locally uniform convergence need not preserve gradients at $0$. To express the function from a new center, define
\begin{equation}\label{g:eq:translation}
 (T_zH)(u)=H(u+z)-H(z).
\end{equation}
For actual quotients, $T_zH_{x,s}=H_{x+sz,s}$: this changes the center while keeping the scale fixed. Given $(H,v)$, draw $\eta\sim N(0,\Sigma)$ and put $\Delta=-v+\eta$. The updated state is $(T_\Delta H,\nabla H(\Delta))$. Since $H$ is Lipschitz and $\Delta$ has a density, Rademacher's theorem \cite[Theorem~9.60]{RW} implies that the gradient exists with probability $1$; set it to $0$ at exceptional points. After changing the center to $z$, $T_zH$ still belongs to $\mathcal C$. The updated distribution defines the transition kernel $P((H,v),\cdot)$. The expected descent is
\begin{equation}\label{g:eq:descentdef}
 d(H,v)=\E_{\eta\sim N(0,\Sigma)}[H(0)-H(-v+\eta)].
\end{equation}

For a probability measure $\mu$ on this state space, sample $(H,V)\sim\mu$ independently of $\eta\sim N(0,\Sigma)$, and set $\Delta=-V+\eta$. The updated distribution is $\mu P$, where for every Borel set $A$,
\[
 (\mu P)(A)=\int P((H,v),A)\,d\mu(H,v).
\]
Stationarity means $\mu P=\mu$: one update preserves the joint distribution of the function and vector. The mean descent is
\begin{equation}\label{g:eq:averaged-descent}
 \int d(H,v)\,d\mu(H,v)
 =\E_{\substack{(H,V)\sim\mu\\\eta\sim N(0,\Sigma)}}[H(0)-H(\Delta)].
\end{equation}

\subsection{Positive expected descent}
The condition $0\notin\pl f(0)$ gives a local Fr\'echet subgradient bound that becomes global for limiting difference quotient functions. Combined with Gaussian escape, it yields strictly positive stationary mean descent. For some $a>0$ and sufficiently small $R>0$,
\begin{equation}\label{g:eq:intro-gap}
 \norm{w}\ge a\qquad(x\in\overline B(0,R),\ w\in\pf f(x)).
\end{equation}
\begin{lemma}[Difference quotients preserve the same lower bound]\label{g:lem:profilegap}
For every difference quotient $H_{x,s}$,
\begin{equation}\label{g:eq:exact-subgradient-scaling}
 \pf H_{x,s}(u)=\pf h(x+su).
\end{equation}
In particular, whenever $x+su\in B(0,R)$, $\norm{v}\ge a$ for every $v\in\pf H_{x,s}(u)$, with the same $a$ as in~\eqref{g:eq:intro-gap}.
Every limiting difference quotient function $H\in\mathcal C$ satisfies
\begin{equation}\label{g:eq:profilegap}
 \norm{v}\ge a\qquad(u\in\R^d,\ v\in\pf H(u)).
\end{equation}
\end{lemma}
Appendix~\ref{g:sec:profiles} gives the complete proof.

The translation class $\{T_zH:z\in\R^d\}$ groups functions differing only by a change of center. The following decomposition separates invariant directions from the remaining motion.
\begin{lemma}[Linear and normal parts of a difference quotient function]\label{g:lem:split}
For $H\in\mathcal C$ set
\begin{equation}\label{g:eq:stabilizer}
 S=\{z:T_zH=H\},\qquad N=S^\perp.
\end{equation}
Then $S$ is a linear subspace. There are $c\in S$ and a definable
$L$-Lipschitz function \mbox{$K:N\to\R$ such that}
\begin{equation}\label{g:eq:split}
 H(p+q)=\ip{c}{p}+K(q),\qquad p\in S,\ q\in N.
\end{equation}
Moreover,
\begin{equation}\label{g:eq:subsplit}
 \pf H(p+q)=\{c+w:w\in\pf^{N}K(q)\},
\end{equation}
Here $P_S$ and $P_N$ denote the orthogonal projections onto $S$ and
$N$, and $\pf^N$ uses increments within $N$. The map
$z\mapsto T_zH$ is a continuous bijection from $N$ onto the orbit of
$H$, with a Borel inverse.
\end{lemma}

Within a translation class, the normal position follows $Z^+=Z-\nabla_NK(Z)+P_N\eta$ on $N$. Appendix~\ref{g:sec:orbits} proves the decomposition and transfers stationarity to these normal positions. The next lemma excludes a stationary distribution under a global positive subgradient bound.
\begin{lemma}[Fixed-step escape from a global subgradient bound]\label{g:lem:fixed}
Let $H:\R^d\to\R$ be definable and globally $L$-Lipschitz, where
$d\ge1$.  Suppose that, for some $a>0$,
\begin{equation}\label{g:fixed:gap}
 \norm{v}\ge a
 \qquad\text{for every }x\in\R^d\text{ and }v\in\pf H(x).
\end{equation}
Let $Y_0$ be finite almost surely and let
\begin{equation}\label{g:fixed:chain}
 Y_{n+1}=Y_n-V_n+\eta_{n+1},\qquad V_n\in\pc H(Y_n),
\end{equation}
where the states and selections are adapted to $(\mathcal F_n)$.
Assume that the noise vectors are independent and identically distributed
with law $N(0,\Sigma)$, that $\Sigma$ is positive definite, and that
$\eta_{n+1}$ is independent of $\mathcal F_n$.  Then
\begin{equation}\label{g:fixed:escape}
 \PP\bigl(\norm{Y_n}\longrightarrow\infty\bigr)=1.
\end{equation}
In particular, for any fixed measurable rule choosing a Clarke subgradient
from the current state, one update of \eqref{g:fixed:chain} changes every
probability distribution.
\end{lemma}

The Gaussian density makes an open region with positive escape probability accessible from every bounded set; repeated access yields almost-sure escape (Appendix~\ref{g:sec:fixed-tool}). If $N\ne\{0\}$ and $\norm{c}<a$, the normal potential $K$ inherits the positive subgradient bound $\sqrt{a^2-\norm{c}^2}$, contradicting stationarity by Lemma~\ref{g:lem:fixed}. Thus $\norm{c}\ge a$; for $N=\{0\}$, the inherited bound gives this directly. Stationarity cancels the expected increment of $K$, and centering of the noise leaves mean descent $\norm{c}^2$.
\begin{lemma}[Positive expected descent]\label{g:prop:negative}
Every stationary probability $\mu$ for $P$ on $\mathcal C\times\overline B(0,L)$ satisfies
\begin{equation}\label{g:eq:negative}
 \int d(H,v)\,d\mu(H,v)\ge a^2>0.
\end{equation}
\end{lemma}
Appendix~\ref{g:sec:energy} gives the complete argument, including the averaging over translation classes.

\subsection{Expected descent equal to \texorpdfstring{$0$}{0} under convergence}
Assuming convergence to the target with positive probability, we construct a stationary distribution with mean descent $0$. The construction uses stopped averages of actual difference quotient states and separate telescoping identities for the values of $h$ and for state tests.
\begin{lemma}[Expected descent equal to $0$ under convergence]\label{g:prop:zero}
If $\PP(X_n\to0)>0$, there is a probability measure $\mu$ on $\mathcal C\times\overline B(0,L)$ such that
\begin{equation}\label{g:eq:zero}
 \int d(H,v)\,d\mu(H,v)=0,\qquad \mu P=\mu.
\end{equation}
\end{lemma}
The full proof is in Appendices~\ref{g:sec:windows}--\ref{g:sec:convergence}. Here we give a proof sketch following the window branch of Figure~\ref{g:fig:route}. First, assuming $\PP(X_n\to0)>0$, we apply Lemma~\ref{g:lem:window} to select late windows $[n_j,n_j+m_j]$ whose iterates remain in the ball of radius $j^{-2}$ with conditional probability tending to one. There are increasingly many recorded steps in conditional expectation, and the fixed scale $s_j=\gamma_{n_j}$ makes $\gamma_n/s_j\to1$ uniformly over each window.

Next, we define a probability $\nu_j$ by averaging the distributions of the recorded pairs $(H_{X_n,s_j},g(X_n))$. We normalize the conditional expected visit counts by the conditional expected number of records. These measures live on the compact space $E=\mathcal D\times\overline B(0,L)$, where $\mathcal D$ is the locally uniform closure of all actual difference quotients. Appendix~\ref{g:sec:profiletools} extends $P$ and $d$ to this space.

Under these measures, we can show that the mean descent and the mean change of each continuous state test tend to zero. Summing the decreases of $h$ leaves only the difference between its initial and final values. The chosen radius and window length make the conditional expectation of this difference negligible after normalization, even if the last update leaves the ball. Similarly, successive changes of any continuous state test $\psi$ telescope to two bounded endpoint values. Their average vanishes as the expected number of records grows. Replacing the nearly unit step ratios by one, Lemmas~\ref{g:lem:occupation-descent} and~\ref{g:lem:occupation} give
\[
 \int_E d\,d\nu_j\longrightarrow0,
 \qquad
 \int_E(P\psi-\psi)\,d\nu_j\longrightarrow0
 \quad(\psi\in C(E)).
\]
Finally, by compactness we take a common subsequence $\nu_{j_\ell}\Rightarrow\mu$. The shrinking centers and scales put its support in $\mathcal C\times\overline B(0,L)$. Bounded continuity of $d$ and $P\psi$ passes both identities to this same limit: its mean descent is zero and it is stationary. The measure $\mu$ may involve different limiting functions $H$; no single common function limit is required.

Applying Lemma~\ref{g:prop:negative} to this $\mu$ gives $0=\int d\,d\mu\ge a^2>0$, the required contradiction.

\begin{corollary}[Finite sets and generic linear tilts]\label{g:cor:finite-generic}\label{g:rem:generic}
Under the function, initialization, noise, and step-size assumptions of
Theorem~\ref{g:thm:main}, every adapted Clarke selection satisfies
\[
 \PP(\exists x\in A:X_n\to x)=0
\]
for each finite $A\subset\R^d$ with $0\notin\pl f(x)$ for all $x\in A$.
Moreover, outside a definable exceptional set of dimension less than
$d$, each fixed tilt $u$ gives a finite set $R_u$ of sharply repulsive
points of $f_u(x)=f(x)-\ip{u}{x}$. Under the same initialization,
noise, and step-size assumptions, every adapted Clarke selection for
$f_u$ satisfies $\PP(\exists x\in R_u:X_n^u\to x)=0$.
\end{corollary}

The corollary follows from a finite union and generic finiteness of the Clarke critical set; Appendix~\ref{g:sec:scope} gives the proof. It establishes sharply repulsive-point avoidance in the generic question of \cite[Remark~4]{BHS}. \begin{remark}[Noise assumptions]\label{g:rem:noise}
The conclusion also holds for iid noise independent of the past with $\E[\eta_1]=0$ and $\E[\norm{\eta_1}^2]<\infty$, whose joint density is bounded, continuous, and everywhere positive. This includes independent centered Laplace or logistic coordinates and fixed nonsingular linear transforms of these vectors. Centering and the second moment control expected descent and the window boundary terms; the density supplies gradient identification, continuity of the transition, and access to escape regions. The projected noise retains absolute continuity, full support, centering, and finite second moments, which suffice for the normal-chain argument; hence the same proof applies.
\end{remark}

\begin{figure}[!tbp]
\centering
\includegraphics[width=.78\linewidth]{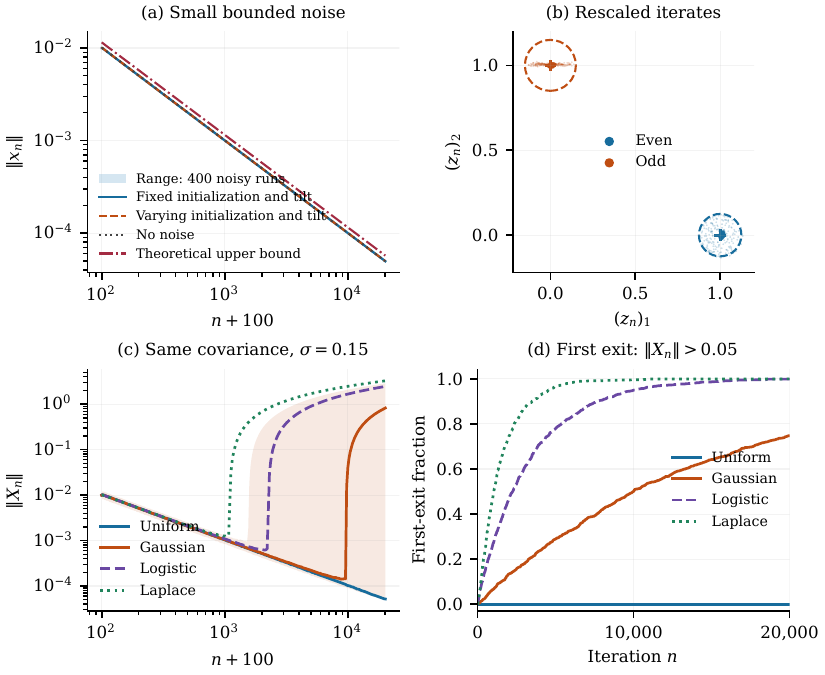}
\caption{Experiments on the original counterexample. (a,b) Uniform noise of radius $0.005$: physical distances and rescaled iterates. (c,d) Four noise laws with covariance $0.15^2I_2$: median distance and cumulative first exit from radius $0.05$, using all $1{,}024$ trajectories per law. Gaussian shading in (c) marks the pointwise $10$th--$90$th percentiles.}
\label{fig:experiments}
\end{figure}
\section{Numerical experiments}\label{sec:experiments}
On the original objective~\eqref{p:eq:function}, we examine the bounds on rescaled iterates and compare finite-time escapes under different noise distributions. For uniform noise of radius $1/200$, we run two groups of $200$ trajectories for $20{,}000$ updates: one fixes $u=0$ and $x_0=a/100$; the other varies the admissible initial point and tilt, with each tilt fixed along its trajectory. Figure~\ref{fig:experiments}(a,b) shows distances shrinking with the step size and rescaled iterates alternating between two bounded regions, as predicted by the pathwise bounds.

For the noise comparison, we fix $u=0$, $X_0=(0.01,0)$, and $\gamma_n=1/(n+100)$. Gaussian, disk-uniform, logistic, and Laplace noises all have covariance $0.15^2I_2$. For each distribution, we run four independent groups of $256$ trajectories for $20{,}000$ updates. Figure~\ref{fig:experiments}(c,d) shows escapes under the full-support noises, while uniform-noise distances remain small. Logistic and Laplace escapes occur earlier, and the Gaussian band includes trajectories that have not yet escaped. Equal covariance controls the noise scale, but support and tails both vary, so neither effect is isolated.

The first exit $\norm{X_n}>0.05$ is a finite-horizon statistic; it neither requires permanent escape nor is equivalent to $\PP(X_n\to0)=0$. Appendix~\ref{app:experiments} gives individual trajectories and escape counts.

\section{Conclusion}\label{sec:conclusion}
Our bounded-noise counterexample resolves the open question of \citet{BHS} in the negative, showing that stochastic subgradient descent can converge almost surely to a nonminimal sharply repulsive point. Under Gaussian noise, however, we prove avoidance of prescribed finite sets with $0\notin\pl f(x)$ and, under generic linear perturbations, all sharply repulsive points. A two-cycle argument and a stationary descent contradiction establish the respective results. Building on the Gaussian result, a future direction is to establish avoidance of active strict saddles beyond weak convexity.

\label{sec:main-end}
\clearpage
\section*{AI use statement}
Generative AI tools assisted with developing, checking, and writing mathematical proofs. The authors reviewed and revised the arguments and take responsibility for the mathematical claims and the final manuscript.

\begingroup
\raggedright
\bibliography{references}
\bibliographystyle{iclr2027_conference}
\endgroup
\clearpage
\appendix
\etocdepthtag.toc{appendix}
\begingroup
\etocsettagdepth{main}{none}
\etocsettagdepth{appendix}{subsection}
\etocsetnexttocdepth{subsection}
\etocsettocstyle{\section*{Contents of the Appendix}}{}
\tableofcontents
\endgroup

\section{Preliminaries}\label{p:sec:preliminaries}\label{sec:appendix-preliminaries}
We collect the notation, nonsmooth and definable geometry, and
probabilistic tools used in both Sections~\ref{sec:counterexample}
and~\ref{sec:gaussian}. Results specific to the counterexample or the
Gaussian argument are developed in Appendices~\ref{p:sec:proof}
and~\ref{g:sec:proofs}, respectively.
\subsection{Notation and nonsmooth geometry}\label{p:sec:subdifferentials}\label{p:sec:regularity}\label{p:sec:active}
We use the Euclidean inner product $\inner{x}{y}=x^\top y$ and norm $\norm{x}$. The open and closed balls are $B(c,r)$ and $\overline B(c,r)$; $\norm{A}_{\op}=\sup_{\norm{h}=1}\norm{Ah}$ is the operator norm. The vectors $e_i$ are coordinate vectors, and $e=(1,1)$ in the counterexample. We write $D\Phi$ for a derivative matrix, $\nabla f$ for a gradient, $\nabla^2 f$ for a Hessian, and $\conv S$ for a convex hull.

\begin{definition}[Fr\'echet, limiting, and Clarke subdifferentials; {\citealp[Definition~8.3 and Theorem~9.61]{RW}}]
Let $f:\R^d\to\R$ be locally Lipschitz. The Fr\'echet subdifferential consists of the vectors $p$ such that
\begin{equation}\label{g:eq:frechet}
 \liminf_{\substack{h\to0\\h\ne0}}\frac{f(x+h)-f(x)-\inner{p}{h}}{\norm{h}}\ge0.
\end{equation}
The limiting subdifferential $\pl f(x)$ consists of all limits $p_j\to p$ with $x_j\to x$ and $p_j\in\pf f(x_j)$. The Clarke subdifferential is
\[
 \pc f(x)=\conv\left\{\lim_j\nabla f(x_j):x_j\to x,\ f\text{ is differentiable at }x_j,\ \text{the limit exists}\right\}.
\]
A point is Clarke critical if $0\in\pc f(x)$.
\end{definition}
Continuity makes function-value convergence automatic in the limiting definition. In this setting,
\[
 \pf f(x)\subseteq\pl f(x)\subseteq\pc f(x),\qquad \pc f(x)=\conv\pl f(x).
\]
The limiting subdifferential has a closed graph \citep[Proposition~8.7]{RW}. All three subdifferentials equal $\{\nabla f(x)\}$ when $f$ is $C^1$ near $x$, and linear perturbations satisfy \citep[Exercise~8.8]{RW}
\[
 \partial_\bullet(f-\inner{u}{\cdot})(x)=\partial_\bullet f(x)-u,
 \qquad\bullet\in\{F,L,C\}.
\]
Subdifferential regularity is equivalent here to equality of the three subdifferentials \citep[Definition~7.25 and Corollary~8.11]{RW}. Weak convexity means that $f+(\lambda/2)\norm{\cdot}^2$ is convex for some $\lambda\ge0$ \citep[Section~2]{BHS}. It implies regularity because convex functions are regular and smooth perturbations preserve regularity \citep[Proposition~8.12 and Exercise~8.8(c)]{RW}.

\begin{definition}[Active manifold and sharply repulsive critical point]\label{g:def:sharp}
An \emph{active manifold} $M$ at a Clarke critical point $x^\star$ is a smooth manifold through $x^\star$ such that $f|_M$ is smooth near $x^\star$ and there are $\varepsilon,\delta_0>0$ for which \citep[Definition~3]{BHS}
\[
\norm{v}\ge\delta_0
\quad\text{whenever }x\in B(x^\star,\varepsilon)\setminus M,
\quad v\in\partial_C f(x).
\]
Following \citet[Definition 6.2.4]{SchechtmanThesis}, a Clarke critical point $x^\star$ is called \emph{sharply repulsive} if it admits an active manifold $M$ such that
\[
x^\star\text{ is a local minimizer of }f|_M,
\qquad 0\in\partial_C f(x^\star)\setminus\partial_L f(x^\star).
\]
\end{definition}
The counterexample in Section~\ref{sec:counterexample} uses $M=\{0\}$. A singleton is a $0$-dimensional smooth manifold; the restriction of any function to it is constant and hence has a local minimum. 

\subsection{Definability and variational tools}\label{g:sec:prelim}
The arguments use definability to control limits and variational tools to transfer subgradient bounds. We collect these facts together with the Lipschitz extension used to handle updates outside a local neighborhood.
The notation $\mathcal L^d$ denotes Lebesgue measure on $\R^d$.
\begin{definition}[Definable sets and functions; {\citealp[Section~2]{vddm}}]\label{g:def:definable}
An o-minimal expansion of the ordered real field is a collection
$\mathcal S=(\mathcal S_m)_{m\ge1}$ of families of subsets of $\R^m$
that contains all semialgebraic sets, is closed under finite unions,
intersections, complements, Cartesian products, and coordinate
projections, and whose sets in $\mathcal S_1$ are finite unions of
points and intervals. A semialgebraic set is a finite Boolean
combination of polynomial equalities and inequalities.
A set is definable if it belongs to the fixed collection $\mathcal S$;
a function is definable if its graph does. Real parameters are allowed.
\end{definition}
We fix one such structure throughout; see \cite[Chapter~1]{vdd}
and \cite[Section~2]{vddm}. Semialgebraic functions, including
polynomials, norms, and finite maxima of polynomials, are definable.
An expansion may contain additional functions while retaining the
closure properties and the $1$-dimensional finiteness condition.

A function $F:\R^d\to\R$ that is
definable and locally Lipschitz admits a finite partition of its
domain into smooth pieces on each of which it is $C^1$. In particular,
there is a definable open set $O$ with
\[
 F|_O\in C^1(O),\qquad \mathcal L^d(\R^d\setminus O)=0,
 \qquad \pc F(x)=\{\nabla F(x)\}\quad(x\in O).
\]
These statements follow from smooth cell decomposition
\cite[Chapters~3 and~7]{vdd}, \cite[Section~4.2]{vddm}, and the
Clarke gradient formula \cite[Theorem~9.61]{RW}. Rademacher's theorem gives almost-everywhere differentiability
for any locally Lipschitz function \cite[Theorem~9.60]{RW}; definability
provides the open $C^1$ set above.

We also use a uniform version of $1$-dimensional monotonicity.
For a fixed definable family of functions of one real variable, each
member is monotone on finitely many intervals, with a bound on the
number of intervals independent of the parameters
\cite[Sections~4.1, 4.2 and~4.4]{vddm}.

The generalized Fermat rule says that a local minimizer has $0$ as a
Fr\'echet subgradient \cite[Theorem~10.1]{RW}. Together with
\eqref{g:eq:frechet}, it gives the smooth perturbation and change-of-scale
rules used below.

We will use the following two variational facts.
\begin{lemma}[Exact Lipschitz extension]\label{g:fact:extension}
Let $A\subset\R^d$ be nonempty and let $F:A\to\R$ \mbox{be $L$-Lipschitz. Then}
\[
 \widetilde F(x)=\inf_{y\in A}\{F(y)+L\norm{x-y}\}
\]
is a finite, globally $L$-Lipschitz function equal to $F$ on $A$.
If $A$ and $F$ are definable, so is $\widetilde F$.
\end{lemma}
This is the McShane--Whitney theorem \cite[Theorem~2.3]{heinonen};
definability follows by expressing the infimum with real quantifiers.

\begin{lemma}[A smooth lower support for a subgradient]\label{g:fact:support}
Let $F:\R^d\to\R$ be locally Lipschitz. For $v\in\pf F(x)$ there is a $C^1$ function $\phi$ on a
neighborhood of $x$ such that
\[
 \phi(x)=F(x),\qquad \nabla\phi(x)=v,\qquad
 \phi(y)<F(y)\quad\text{for }y\ne x\text{ in that neighborhood}.
\]
\end{lemma}
The lower support may be chosen $C^1$ and strict away from the
contact point. We use this form of \cite[Proposition~8.5, p.~302]{RW} to transfer
subgradient bounds through convergence of local minimizers.

\begin{definition}[Conservative field; {\citealp[Definition~2.13]{schechtman}}]\label{g:def:conservative}
Let $\mathcal V:\R^d\rightrightarrows\R^d$ be a set-valued map,
so that $\mathcal V(x)\subset\R^d$, and assume it has closed graph,
nonempty compact values, and is locally bounded. It is conservative
for a locally Lipschitz function $F$ if, along every
absolutely continuous curve $x$, for almost every $t$,
\[
 (F\circ x)'(t)=\ip{v}{\dot x(t)}
 \quad\text{for every }v\in\mathcal V(x(t)).
\]
\end{definition}
This is the chain-rule form of a conservative field; see
\cite[Section~2.4, Definition~2.13]{schechtman}. A variational stratification of $(F,\mathcal V)$ is a partition
into smooth manifolds satisfying the Whitney compatibility condition,
on each of which $F$ is smooth and every vector in $\mathcal V(x)$
projects onto the gradient of the restriction to that manifold.
The Whitney condition requires limiting tangent spaces of an
approaching stratum to contain the tangent space of the stratum it
meets; see \cite[Definitions~2.3 and~3.1]{schechtman}.
For definable locally Lipschitz $F$, the pair $(F,\pc F)$ admits
such a stratification and $\pc F$ is conservative in the sense of
Definition~\ref{g:def:conservative}
\cite[Section~2.4, Proposition~2.15]{schechtman}.

\begin{lemma}[A conservative limit for one parameter]\label{g:fact:conservative}
Let $F_s:\R^d\to\R$, $0<s\le1$, be jointly definable, locally
Lipschitz, and locally uniformly convergent to a locally Lipschitz
function $F_0$. Let $\mathcal V_s$ be jointly definable closed fields with
nonempty values, and suppose each pair $(F_s,\mathcal V_s)$ admits a definable
variational stratification. If the field values are uniformly bounded
on each compact set for small $s$, then the outer graphical limit of
$\mathcal V_s$, and its convexification, are conservative for $F_0$.
At $x$, this limit consists of all $v$ obtained from sequences
$s_j\downarrow0$, $x_j\to x$, and $v_j\in\mathcal V_{s_j}(x_j)$
with $v_j\to v$.
\end{lemma}
This bounded form follows from \cite[Theorem~4.3 and Corollary~4.5]{schechtman}:
local uniform convergence gives the required potential limits and the
bounded field has horizon field $\{0\}$: limits of vanishing positive
multiples of its values can only be $0$. Only a single scalar parameter
is used. The limit field need not equal $\pc F_0$.

\subsection{Probability and stationary distributions}\label{g:sec:probability}
The original stochastic iteration is defined on a probability space
$(\Omega,\mathcal F,\PP)$, with the standard probability-space and random-variable conventions of \citet[Sections~1.1--1.3]{durrett}. Here $\omega\in\Omega$ specifies a realization
of the initialization, the entire noise sequence, and any randomness
in the adapted subgradient selections; $\mathcal F$ is the sigma-field
of events, and $\PP$ is their probability measure.
A filtration $(\mathcal F_n)_{n\ge0}$ is an increasing sequence of
sub-sigma-fields of $\mathcal F$; $\mathcal F_n$ represents the
information available through time $n$. A process is adapted if its
value at time $n$ is $\mathcal F_n$-measurable \citep[Section~4.2]{durrett}. In particular, $X_n$,
$V_n$, and the noises already observed are measurable with respect
to the current history. The sigma-field $\mathcal F$ also contains
future events, whereas $\mathcal F_n$ records only the current
information. For an integer stopping time $\tau$, $\mathcal F_\tau$
consists of events $A\in\mathcal F$ such that
$A\cap\{\tau\le n\}\in\mathcal F_n$ for every $n$ \citep[Section~4.8]{durrett}.

\paragraph{Expectations and conditioning on the original space.}
Unsubscripted expectations of path variables use $\PP$:
$\E[Y]=\int_\Omega Y(\omega)\,\PP(d\omega)$.
For a probability $Q$ on $(\Omega,\mathcal F)$ and an integrable $Y$,
$\E_Q[Y\mid\mathcal G]$, where $\mathcal G\subseteq\mathcal F$,
is a $\mathcal G$-measurable random variable satisfying \citep[Section~4.1]{durrett}
\[
 \int_A\E_Q[Y\mid\mathcal G]\,dQ=\int_A Y\,dQ
 \qquad(A\in\mathcal G).
\]
This specifies it up to equality $Q$-almost surely. It is generally
a random variable, while $\E_Q[Y]$ is a number. For nonnegative
$Y$, the same notation is also used for extended expectations,
without assuming integrability. Conditional
probabilities mean $\PP(A\mid\mathcal G)=\E[\mathbf1_A\mid\mathcal G]$.
For an event $B$ with $\PP(B)>0$, conditioning instead defines the
probability space $(\Omega,\mathcal F,\PP_B)$, where
\[
 \PP_B(A)=\frac{\PP(A\cap B)}{\PP(B)},\qquad
 \E_{\PP_B}[Y]=\frac{1}{\PP(B)}\int_B Y\,d\PP=\E[Y\mid B].
\]
Section~\ref{g:sec:windows} displays this conditioning explicitly as
$\E[Y\mid B_j]$. Its notation $\E[Y\mid\mathcal F_t,B_j]$
means conditional expectation given $\mathcal F_t$ on
$(\Omega,\mathcal F,\PP(\,\cdot\mid B_j))$; these identities are
understood almost surely on $B_j$. We use the tower property and the rule for taking known factors outside a conditional expectation as stated in \citet[Theorems~4.1.13--4.1.14]{durrett}, applied on this conditional probability space.

\paragraph{A single noise vector and a concrete trajectory space.}
Write $\mathcal B(S)$ for the Borel sigma-field of a metric space $S$
and $\Gamma_\Sigma=N(0,\Sigma)$ for the Gaussian probability measure \citep[Section~3.10]{durrett}.
In the Gaussian argument, $\E_\eta$ and $\PP_\eta$ use
$(\R^d,\mathcal B(\R^d),\Gamma_\Sigma)$ with coordinate random
variable $\eta(z)=z$. Thus, with all other arguments fixed,
\[
 \E_\eta[\phi(\eta)]
 =\int_{\R^d}\phi(z)\,\Gamma_\Sigma(dz)
 =\int_{\R^d}\phi(z)\varphi_\Sigma(z)\,dz.
\]
A block of $k$ independent noise vectors uses
$((\R^d)^k,\mathcal B(\R^d)^{\otimes k},
\Gamma_\Sigma^{\otimes k})$; the product-measure construction is recalled in \citet[Section~1.7]{durrett}.
For initialization with distribution $\rho_0$, independent of the
noise, and subgradient selections determined measurably by the
observed history, one concrete realization of the whole iteration is
\[
 \begin{aligned}
 \Omega&=\R^d\times(\R^d)^{\mathbb N},&
 \mathcal F&=\mathcal B(\R^d)\otimes
                   \mathcal B(\R^d)^{\otimes\mathbb N},\\
 \PP&=\rho_0\otimes\Gamma_\Sigma^{\otimes\mathbb N},&
 \mathcal F_n&=\sigma(X_0,\eta_1,\ldots,\eta_n).
 \end{aligned}
\]
Here $X_0$ and $\eta_i$ are coordinate projections, and the updates
define $X_n$ as functions of these coordinates. Deterministic
initialization corresponds to a point mass $\rho_0$. The countable product is justified by the extension theorem \citep[Theorem~2.1.21]{durrett}.
The theorems allow any filtered probability space satisfying their
adaptation and independence assumptions, including additional
randomness in the selections. For the bounded-noise example, replace
$\Gamma_\Sigma$ by the uniform ball distribution; for
Remark~\ref{g:rem:noise}, replace it by the specified noise distribution.

\paragraph{Expectations over states and independent noise.}
For a probability $\rho$ on a measurable space $(S,\mathcal A)$,
an integral $\int_S\phi\,d\rho$ is an expectation on
$(S,\mathcal A,\rho)$, with coordinate random variable $Z(z)=z$.
In particular, for the state space $E$ of
Section~\ref{g:sec:profiletools}, the integrals against $\nu_j$,
$\mu$, or $\mu_\theta$ use
$(E,\mathcal B(E),\nu_j)$, $(E,\mathcal B(E),\mu)$, or
$(E,\mathcal B(E),\mu_\theta)$, respectively.
The underlying sample is a single function and vector pair.
The integral over orbit labels uses $(\Theta,\mathcal B(\Theta),\nu)$,
where $\mathcal B(\Theta)$ denotes its standard Borel sigma-field;
this $\nu$ is distinct from the stopped averages $\nu_j$.
When a state $(H,V)$ with distribution $\rho$ is combined with
independent Gaussian noise, the probability space is
\[
 \bigl(E\times\R^d,\,
       \mathcal B(E)\otimes\mathcal B(\R^d),\,
       \rho\otimes\Gamma_\Sigma\bigr),
\]
and its coordinate random variables are $(H,V)$ and $\eta$. Fubini's theorem \citep[Theorem~1.7.2]{durrett} gives, for integrable $\Phi$,
\[
 \E_{\rho\otimes\Gamma_\Sigma}[\Phi(H,V,\eta)]
 =\int_E\left(\int_{\R^d}\Phi(H,v,z)\,
                   \Gamma_\Sigma(dz)\right)\rho(d(H,v)).
\]
This is the joint expectation in \eqref{g:eq:averaged-descent} when
$\rho=\mu$. The normal-position expectations in
Section~\ref{g:sec:energy} use the corresponding product space on
$N\times\R^d$, specified there.
We write $\mathbf1_A$ for an indicator and $\mathcal L(Z)$ for the
distribution of $Z$ under the probability measure specified locally.

For a metric space $K$, $C(K)$ denotes the real-valued continuous
functions on $K$, $\norm{\psi}_{\infty,K}=\sup_{x\in K}|\psi(x)|$,
and $\operatorname{Lip}(\psi)$ denotes the Lipschitz constant.
We omit $K$ when the domain is clear. The notation
$\mu_j\Rightarrow\mu$ means weak convergence of probability measures,
that is, convergence of integrals against every bounded continuous
test function; the same notation for random variables refers to
their distributions \citep[Section~3.10]{durrett}.
\begin{definition}[Stationary probability; {\citealp[Section~5.2 and Example~6.1.2]{durrett}}]\label{g:def:stationary}
On a measurable state space $K$, a Markov kernel $P$ assigns a
probability measure $P(x,\cdot)$ to each $x\in K$, with
$x\mapsto P(x,A)$ measurable for every measurable $A\subset K$.
Thus $P(x,A)$ is the conditional probability that the next state
belongs to $A$. Write
\[
 P\psi(x)=\int\psi(y)P(x,dy),\qquad
 (\mu P)(A)=\int P(x,A)\,d\mu(x).
\]
A probability $\mu$ is stationary if $\mu P=\mu$, equivalently
\[
 \int P\psi\,d\mu=\int\psi\,d\mu
 \quad\text{for every bounded measurable }\psi.
\]
\end{definition}
The transition kernel $P$ is distinct from the path probability
measure $\PP$. For a fixed input $x$, $P\psi(x)$ can also be read as
an expectation on $(K,\mathcal B(K),P(x,\cdot))$ when $K$ is a
Borel state space.
Thus a state with law $\mu$ still has law $\mu$ after one transition;
see \cite[Example~6.1.2]{durrett}. The general-state-space Markov
kernel convention is used throughout.

\begin{lemma}[Upward conditional-expectation convergence]\label{g:lem:upward}
Let $(\mathcal F_n)_{n\ge0}$ be a filtration on
$(\Omega,\mathcal F,\PP)$ and put
$\mathcal F_\infty=\sigma(\bigcup_n\mathcal F_n)$. For an integrable random variable $Z$,
\[
 \E[Z\mid\mathcal F_n]\longrightarrow\E[Z\mid\mathcal F_\infty]
 \quad\text{almost surely and in }L^1.
\]
In particular, for $A\in\mathcal F_\infty$,
\[
 \PP(A\mid\mathcal F_n)\longrightarrow\mathbf1_A
 \quad\text{almost surely and in }L^1.
\]
\end{lemma}
These are \cite[Theorems~4.6.8--4.6.9]{durrett}. The event form
will turn conditional probabilities based on longer histories into
indicators of the corresponding eventual behavior.
\begin{lemma}[Doob's $L^2$ maximal inequality]\label{g:lem:doob}
Let $(\mathcal F_k)_{k=0}^m$ be a filtration and let $(M_k)_{k=0}^m$
be a square-integrable Euclidean-valued martingale: $M_k$ is
$\mathcal F_k$-measurable and \mbox{$\E[M_{k+1}\mid\mathcal F_k]=M_k$. Then}
\[
 \E\!\left[\max_{0\le k\le m}\norm{M_k}^2\mid\mathcal F_0\right]
 \le4\E[\norm{M_m}^2\mid\mathcal F_0].
\]
Consequently, for every positive $\mathcal F_0$-measurable $r$,
\[
 \PP\!\left(\max_{0\le k\le m}\norm{M_k}>r\mid\mathcal F_0\right)
 \le\frac{4\E[\norm{M_m}^2\mid\mathcal F_0]}{r^2}.
\]
\end{lemma}
Apply \cite[Theorem~4.4.4]{durrett}, with $p=2$, to the nonnegative
submartingale $\norm{M_k}$. Testing against events in $\mathcal F_0$
gives the conditional form; conditional Markov inequality gives the
second display. An integer horizon known at time $0$ is handled
on each of its level events.

\section{Proofs for the bounded-noise counterexample}\label{p:sec:proof}
\label{p:sec:main}
The proof of the counterexample in Section~\ref{sec:counterexample} combines two arguments: subgradient calculations establish a nonminimal sharply repulsive point $0$, and a pathwise two-cycle argument proves convergence to it. A radial modification then makes the objective coercive while preserving its local geometry and the convergent trajectories.
\subsection{The objective is globally Lipschitz but not weakly convex}\label{p:sec:properties}

The representation as a sum of a norm and piecewise linear terms gives global Lipschitz continuity, definability, and positive homogeneity. A downward jump in a directional derivative excludes weak convexity, while homogeneity supplies the rescaling used in the convergence proof.

\begin{lemma}[Function class and positive homogeneity]\label{p:lem:function-class}
For every $u\in\R^2$, the function $f_u$ is globally Lipschitz with constant at most $2+2\sqrt2+\norm{u}$, definable, and positively homogeneous of degree $1$. It is not weakly convex.
\end{lemma}
\begin{proof}
The Euclidean norm is $1$-Lipschitz. Also,
\[
|\min\{x,y\}-\min\{\widetilde x,\widetilde y\}|
\le\max\{|x-\widetilde x|,|y-\widetilde y|\}
\le\norm{(x,y)-(\widetilde x,\widetilde y)}.
\]
The map $(x,y)\mapsto\max\{0,-x-y\}$ is $\sqrt2$-Lipschitz. Thus $f$ is globally Lipschitz with constant at most $2+2\sqrt2$, and $f_u$ has constant at most $2+2\sqrt2+\norm{u}$.

The graph of the norm is described by $r\ge0$ and $r^2=x^2+y^2$. Minimum and maximum of finitely many linear functions have graphs described by finite combinations of polynomial equalities and inequalities. Such descriptions are definable in the real field. Closure of definable functions under finite sums and linear tilts then shows that $f$ and $f_u$ are definable, with $u$ as a real parameter, in the real field and hence in any o-minimal expansion of it. Each term is positively homogeneous of degree $1$, so
\[
f(tz)=t f(z),\qquad f_u(tz)=t f_u(z)
\qquad(t\ge0).
\]
Only nonnegative scaling is asserted; the identities need not hold for negative $t$.

Restrict $f$ to the line $z(t)=(1+t,-1+t)$. Here $\min\{1+t,-1+t\}=-1+t$ and $x+y=2t$. Therefore
\[
f(z(t))=
\begin{cases}
\sqrt{2+2t^2}+1+3t,&t<0,\\
\sqrt{2+2t^2}+1-t,&t\ge0.
\end{cases}
\]
The derivative of $\sqrt{2+2t^2}$ at $t=0$ is $0$. Hence the one-sided derivatives of this restriction are
\[
\left.\frac{d}{dt^-}f(z(t))\right|_{t=0}=3,
\qquad
\left.\frac{d}{dt^+}f(z(t))\right|_{t=0}=-1.
\]
For a convex function of one variable, the left derivative cannot exceed the right derivative. The downward jump $3>-1$ violates this condition. Adding $\lambda\norm{z(t)}^2/2$ adds a differentiable function whose derivative is continuous at $t=0$; it cannot remove the jump. The linear tilt likewise adds the same derivative to both sides. Since the restriction of a convex function to a line is convex, no quadratic correction can make $f_u$ convex.
\end{proof}

The objective therefore has the required Lipschitz and definability properties without weak convexity. We next examine the geometry of its critical point.

\subsection{The point \texorpdfstring{$0$}{0} is a nonminimal sharply repulsive Clarke critical point}\label{p:sec:geometry}\label{p:sec:bound}

A uniform subgradient bound away from $0$ supplies the active-manifold condition for $M=\{0\}$. We combine this bound with explicit subdifferential calculations at $0$ to establish sharp repulsiveness, and use a descent ray to verify nonminimality.

Set
\begin{equation}
\delta=1-\frac1{\sqrt2}>0.
\label{p:eq:delta}
\end{equation}
\begin{lemma}[A uniform Clarke subgradient bound away from $0$]\label{p:lem:off-origin}
For the function $f$ in \eqref{p:eq:function},
\begin{equation}
\norm{v}\ge\delta
\qquad(z\ne0,\ v\in\partial_C f(z)).
\label{p:eq:sharp}
\end{equation}
Consequently, for $\norm{u}<1/200$, every $z\ne0$ and $v_u\in\partial_C f_u(z)$ satisfy $\norm{v_u}\ge\delta-\norm{u}>0$.
\end{lemma}
\begin{proof}
Write $z=(x,y)$, $\rho=\norm{z}$, $w=z/\rho$, and $e=(1,1)$. Thus $\norm{w}=1$. For the base function $f=f_0$ in \eqref{p:eq:function}, at a point off the two lines $x=y$ and $x+y=0$, differentiation of its three terms gives
\begin{equation}
v=w-p+2\one_{\{x+y<0\}}e,
\qquad
p=\begin{cases}e_1,&x<y,\\e_2,&y<x.\end{cases}
\label{p:eq:gradient}
\end{equation}
Indeed, the norm contributes $w$, the negative minimum contributes $-p$, and $-2\max\{0,-x-y\}$ contributes $2e$ exactly when $x+y<0$.

At a nonzero point of $x=y$, only the minimum term changes branch, so its two neighboring gradient limits are convexified: $p$ ranges over $\conv\{e_1,e_2\}$. At a nonzero point of $x+y=0$, the minimum branch is fixed, while $0$ and $2e$ are convexified. Thus the last term becomes $te$, $0\le t\le2$. This $t$ is a convexification parameter, not an iteration index or the line parameter of Section~\ref{p:sec:properties}. The lines intersect only at $0$, which is excluded here.

\emph{Case 1: $x+y>0$.} For every admissible $p$, including the tie $x=y$, we have $\inner{p}{w}=\min\{w_1,w_2\}$. It follows that
\[
\inner{v}{w}=1-\min\{w_1,w_2\}
\ge1-\frac1{\sqrt2}=\delta.
\]
The inequality follows because both coordinates of a unit vector cannot exceed $1/\sqrt2$.~By~\mbox{Cauchy--Schwarz},
\[
\norm{v}=\norm{v}\norm{w}\ge|\inner{v}{w}|
\ge\inner{v}{w}\ge\delta.
\]

\emph{Case 2: $x+y<0$.} Every admissible $p$ has coordinate sum $1$. Since $\inner{e}{e}=2$,
\[
\inner{v}{e}=\inner{w}{e}-1+4=3+\inner{w}{e}.
\]
Cauchy--Schwarz gives $\inner{w}{e}\ge-\norm{w}\norm{e}=-\sqrt2$. Applying it once more,
\[
\norm{v}\ge\frac{\inner{v}{e}}{\norm{e}}
\ge\frac{3-\sqrt2}{\sqrt2}>\delta.
\]

\emph{Case 3: $x+y=0$ and $z\ne0$.} Since $w_1+w_2=0$ and $w_1^2+w_2^2=1$, we have
\[
w=(c,-c)\quad\text{or}\quad w=(-c,c),\qquad c=\frac1{\sqrt2}.
\]
In the first case $y<x$, so $p=e_2$ and
\[
v=(c,-c-1)+t(1,1),\qquad 0\le t\le2.
\]
Its first coordinate is $c+t\ge c$, and hence $\norm{v}\ge c>\delta$. In the second case $p=e_1$ and
\[
v=(-c-1,c)+t(1,1),\qquad 0\le t\le2;
\]
its second coordinate is at least $c$. This proves \eqref{p:eq:sharp} everywhere off $0$.

For the tilted function, every $v_u\in\partial_C f_u(z)$ has the form $v-u$ with $v\in\partial_C f(z)$. The \mbox{reverse triangle inequality yields}
\[
\norm{v_u}\ge\norm{v}-\norm{u}\ge\delta-\norm{u}>0.
\]
In particular, no nonzero point is Clarke critical for the stated tilts.
\end{proof}

\phantomsection\label{p:sec:repulsive}

The lower bound excludes Clarke critical points away from $0$. We now show that $0$ itself remains Clarke critical under all the allowed tilts.
\begin{lemma}[Clarke criticality persists under small linear tilts]\label{p:lem:clarke-criticality}
Let $c=1/\sqrt2$ and \mbox{$k=c-1/2>1/5$. Then}
\begin{equation}
(1,-1),\qquad(-1,1),\qquad k(1,1),\qquad-\tfrac12(1,1)
\quad\text{belong to }\partial_C f(0).
\label{p:eq:four}
\end{equation}
Moreover, $D=\conv\{\pm(1,-1),\ \pm k(1,1)\}\subset\partial_C f(0)$, and $0\in\partial_C f_u(0)$ for every $\norm{u}<1/200$.
\end{lemma}
\begin{proof}
At every differentiability point, \eqref{p:eq:gradient} is unchanged by positive scaling. Thus, for a fixed smooth direction $d$, the sequence $d/j\to0$ has constant gradient $\nabla f(d)$.

For $d=a$, \eqref{p:eq:gradient} gives $\nabla f(a)=(1,-1)$, while $d=b$ gives $\nabla f(b)=(-1,1)$. Hence both are gradient limits at $0$. They are also limiting subgradients because these points have smooth neighborhoods and their gradients are Fr\'echet subgradients.

For the positive diagonal, choose $\varepsilon_j=1/j$ and the smooth points
\[
z_j^+=\frac1j(1+\varepsilon_j,1),
\qquad z_j^-=\frac1j(1,1+\varepsilon_j).
\]
Both sequences approach $0$ with normalized direction $(c,c)$. The first has $y<x$ and the second has $x<y$, giving
\[
\nabla f(z_j^+)\to(c,c-1),\qquad
\nabla f(z_j^-)\to(c-1,c).
\]
Their average is $(c-1/2,c-1/2)=k(1,1)$. Convexity of $\partial_C f(0)$ therefore proves membership of the third vector. We do not claim that this average is itself a limiting gradient.

For the fourth vector, take $j\ge2$ and
\[
\widetilde z_j^+=\frac1j(1,-1+\varepsilon_j),
\qquad
\widetilde z_j^-=\frac1j(-1+\varepsilon_j,1).
\]
Both have strictly positive coordinate sum, so the $2e$ term in \eqref{p:eq:gradient} is absent. Their normalized directions approach $(c,-c)$ and $(-c,c)$, respectively. Consequently,
\[
\nabla f(\widetilde z_j^+)\to(c,-c-1),\qquad
\nabla f(\widetilde z_j^-)\to(-c-1,c).
\]
The average is $(-1/2,-1/2)$, proving the fourth membership by convexification. This specifies both the approach side and the points tending to $0$.

The first two vectors in \eqref{p:eq:four} have midpoint $0$. Also, $0<k<1/2$, and
\[
-k(1,1)=2k\bigl[-\tfrac12(1,1)\bigr]+(1-2k)0.
\]
Hence the diamond
\[
D=\conv\{\pm(1,-1),\ \pm k(1,1)\}
\]
is contained in $\partial_C f(0)$. To test membership, write
\[
u=t(1,-1)+s(1,1),\qquad
t=\frac{u_1-u_2}{2},\quad s=\frac{u_1+u_2}{2}.
\]
In this basis, the diamond is the weighted unit $\ell_1$ ball: $u\in D$ if and only if $|t|+|s|/k\le1$. Indeed, the four vertices have basis coordinates $(\pm1,0)$ and $(0,\pm k)$, and their convex hull is exactly this inequality. By Cauchy--Schwarz,
\[
|t|\le\frac{\norm{u}}{\sqrt2},\qquad
|s|\le\frac{\norm{u}}{\sqrt2}.
\]
Since $1/k<5$, we obtain
\[
|t|+\frac{|s|}{k}
\le\frac{1+1/k}{\sqrt2}\norm{u}
\leq\frac6{\sqrt2}\norm{u}<1.
\]
Thus $u\in D\subset\partial_C f(0)$, and the tilt rule implies $0\in\partial_C f_u(0)$.
\end{proof}

It remains to exclude limiting criticality at $0$ and to find nearby points with lower objective values.

\begin{lemma}[Sharp repulsiveness and nonminimality]\label{p:lem:sharp-repulsiveness}
For every $\norm{u}<1/200$, we have $\partial_F f_u(0)=\varnothing$ and $0\in\partial_C f_u(0)\setminus\partial_L f_u(0)$. The point $0$ is a sharply repulsive Clarke critical point with active manifold $M=\{0\}$, and is not a local minimizer.
\end{lemma}
\begin{proof}
Clarke criticality follows from Lemma~\ref{p:lem:clarke-criticality}.
Suppose, for contradiction, that $p\in\partial_F f_u(0)$. Fix any nonzero $h\in\R^2$ and use increments $th$, $t\downarrow0$, in the definition. Since $f_u(0)=0$ and $f_u(th)=t f_u(h)$,
\[
\frac{f_u(th)-f_u(0)-\inner{p}{th}}{\norm{th}}
=\frac{t f_u(h)-t\inner{p}{h}}{t\norm{h}}
=\frac{f_u(h)-\inner{p}{h}}{\norm{h}}.
\]
The right-hand side is independent of $t$. Its limit must be nonnegative, so
\[
\inner{p}{h}\le f_u(h)\qquad\text{for every }h\in\R^2.
\]
The case $h=0$ is immediate. The passage from a local condition to all fixed directions is justified by positive homogeneity, not by convexity.

Take $h=(1,1)$ and $-h=(-1,-1)$. Direct evaluation gives
\[
f_u(h)=\sqrt2-1-u_1-u_2,
\qquad f_u(-h)=\sqrt2-3+u_1+u_2.
\]
Adding the two supposed lower-support inequalities yields
\begin{equation}
0=\inner{p}{h}+\inner{p}{-h}
\le f_u(h)+f_u(-h)=2\sqrt2-4<0,
\label{p:eq:frechet-contradiction}
\end{equation}
a contradiction. Therefore $\partial_F f_u(0)=\varnothing$.

If $0\in\partial_L f_u(0)$, there would be $x_j\to0$ and $p_j\in\partial_F f_u(x_j)$ with $p_j\to0$. No $x_j$ can equal $0$, because $\partial_F f_u(0)$ is empty. At every remaining point, the inclusion $\partial_F\subseteq\partial_C$ and Lemma~\ref{p:lem:off-origin} give
\[
\norm{p_j}\ge\delta-\norm{u}>0,
\]
contradicting $p_j\to0$. We have proved
\begin{equation}
0\in\partial_C f_u(0)\setminus\partial_L f_u(0).
\label{p:eq:repulsive}
\end{equation}
Thus $f_u$ is not subdifferentially regular at $0$. The singleton $M=\{0\}$ is a smooth active manifold by the uniform bound away from $0$, and $f_u|_M$ is constant. This verifies the stated \mbox{definition of sharp repulsiveness.}

Finally,
\[
f_u(-1,-1)=\sqrt2-3+u_1+u_2
\le\sqrt2-3+\sqrt2\norm{u}<0=f_u(0).
\]
Positive homogeneity gives $f_u(-t,-t)<0$ for every $t>0$. There are therefore points arbitrarily close to $0$ with strictly lower function values, so $0$ is not a local minimizer.
\end{proof}

This completes the geometric part: $0$ is a nonminimal sharply repulsive Clarke critical point. We next show that the stochastic iteration actually converges to it.

\subsection{Almost-sure convergence of stochastic subgradient descent to \texorpdfstring{$0$}{0}}\label{p:sec:convergence}

We first rescale the iterates by the step size. Positive homogeneity gives a recursion on two smooth cones, and an explicit two-cycle argument keeps the rescaled iterates in bounded neighborhoods of $a$ and $b$ for every allowed noise sequence. Returning to the original scale then forces convergence to $0$. Finally, the noise support condition turns this pathwise conclusion into almost-sure convergence.

\subsubsection{The rescaled iterates remain in two bounded regions}\label{p:sec:rescaling}\label{p:sec:cycle}

For the iteration \eqref{g:eq:algorithm} applied to $f_u$, define the rescaled iterates and successive step-size ratios by
\[
z_n=\frac{x_n}{\gamma_n},\qquad q_n=\frac{\gamma_n}{\gamma_{n+1}}.
\]
We introduce the open cones
\[
C_a=\{(x,y):x>y,\ x+y>0\},\qquad
C_b=\{(x,y):y>x,\ x+y>0\}.
\]
They are cones because multiplication by any positive scalar preserves their defining inequalities. On $C_a$, the minimum is $y$ and the negative-part term vanishes, so $f(x,y)=\rho-y$, where $\rho=\norm{z}$. Hence $\nabla f(z)=z/\rho-b$ and
\begin{equation}
T(z):=z-\nabla f(z)=b+(\rho-1)\frac z\rho
\qquad(z\in C_a).
\label{p:eq:Ta}
\end{equation}
Similarly, on $C_b$,
\begin{equation}
f(x,y)=\rho-x,\qquad
T(z)=a+(\rho-1)\frac z\rho.
\label{p:eq:Tb}
\end{equation}
Here $T$ is the unit-step gradient update map for the unperturbed objective on these regions; it is an auxiliary map, not the physical iteration with step $\gamma_n$. It is used below to give the two-cycle argument proving Lemma~\ref{p:lem:normalized-trapping}. Since $\norm{a}=\norm{b}=1$, we have $T(a)=b$ and $T(b)=a$. In fact, the unit-circle arc in $C_a$ maps to $b$, and the arc in $C_b$ maps to $a$. This explains the choice of the two points.

As long as $z_n$ lies in either cone, $x_n=\gamma_nz_n$ lies in that cone too. Formula~\eqref{p:eq:gradient} gives $\nabla f(\gamma_nz_n)=\nabla f(z_n)$, and $v_n=\nabla f(z_n)-u$. Dividing \eqref{g:eq:algorithm} by $\gamma_{n+1}$ therefore gives
\begin{equation}
z_{n+1}=q_n\{T(z_n)+u+\eta_{n+1}\},
\qquad 1\le q_n=1+\frac1{n+100}\le\frac{101}{100},
\qquad \norm{u+\eta_{n+1}}\le\frac1{100}.
\label{p:eq:normalized}
\end{equation}
The noise in this rescaled recursion has fixed bounded size. The small physical noise $\gamma_n\eta_{n+1}$ must not be treated as a vanishing perturbation after rescaling.

The statement is given in Lemma~\ref{p:lem:normalized-trapping}.
\begin{proof}[Proof of Lemma~\ref{p:lem:normalized-trapping}]
We prove \eqref{p:eq:invariant} by induction. The induction also justifies the smooth formulas at each successive physical iterate; smoothness of the whole future path is not assumed in advance.

Set $r=1/8$ and $R=3/20$. If $z=a+h$ with $\norm{h}\le r$, then
\[
z_1-z_2\ge1-\sqrt2r>0,\qquad
z_1+z_2\ge1-\sqrt2r>0.
\]
Thus $\overline B(a,r)\subset C_a$. The same calculation with $b$ and $R$ shows that $\overline B(b,R)\subset C_b$. Both balls avoid $0$, since their radii are less than $1$. Positive rescaling keeps them inside the same smooth cones.

Suppose $z\in\overline B(a,r)$, $r=1/8$. Set $\rho=\norm{z}$, $s=\rho-1$, and $w=z/\rho$. The reverse triangle inequality gives $|s|\le r$ and $\rho\ge1-r$. Since the second coordinate of $a$ is $0$,
\[
|w_2|=\frac{|z_2|}{\rho}\le\frac r{1-r}=\frac17,
\qquad |s w_2|\le\frac1{56}.
\]
Consider any update $z'=q(T(z)+\xi)$ with $1\le q\le101/100$ and $\norm{\xi}\le1/100$. Here $\xi$ represents the combined perturbation $u+\eta$, and the prime labels the next point; it is not a derivative. Put $e'=z'-b$. From $T(z)=b+sw$,
\[
e'=(q-1)b+qsw+q\xi.
\]
Taking norms gives
\begin{equation}
\norm{e'}\le\frac1{100}
+\frac{101}{100}\left(\frac18+\frac1{100}\right)
=\frac{2927}{20000}<\frac3{20}=:R.
\label{p:eq:first-distance}
\end{equation}
Taking the second coordinate instead gives the stronger componentwise estimate
\begin{equation}
|e'_2|\le\frac1{100}
+\frac{101}{100}\left(\frac1{56}+\frac1{100}\right)
<\frac1{25}.
\label{p:eq:first-coordinate}
\end{equation}
This additional information is important. An arbitrary point in $B(b,R)$ need not satisfy \eqref{p:eq:first-coordinate}; an odd iterate produced by the previous phase does.

The next update map is controlled by the radial error $|\norm{z'}-1|$, the distance of $z'$ from the unit circle. This is different from $\norm{z'-b}$, the distance from the single point $b$. The second coordinate estimate improves the radial error because $b=(0,1)$ is the radial direction at that point.

To estimate the radial error, let $\psi(z)=\norm{z}$. For $z\ne0$,
\[
\nabla\psi(z)=\frac z{\norm{z}},\qquad
\nabla^2\psi(z)=\frac1{\norm{z}}I
-\frac1{\norm{z}^3}zz^\top
=\frac1{\norm{z}}(I-ww^\top),\quad w=\frac z{\norm{z}}.
\]
Here $zz^\top$ is the outer product, not the scalar $z^\top z$. The matrix $I-ww^\top$ is the orthogonal projection onto the line perpendicular to $w$. Its eigenvalues in $\R^2$ are $0$ and $1$, so
\[
\norm{\nabla^2\psi(z)}_{\op}=\frac1{\norm{z}}.
\]
This identity is a property of the Euclidean norm in dimension two, not the definition of \mbox{a Hessian operator norm.}

Every point $b+t e'$, $0\le t\le1$, satisfies
\[
\norm{b+t e'}\ge\norm{b}-t\norm{e'}\ge1-R>0.
\]
Thus $\psi$ is twice continuously differentiable along the segment, and its Hessian operator norm there is at most $1/(1-R)$.

To apply one-variable calculus, set $g(t)=\psi(b+t e')$. The vector $e'$ is fixed while $t$ varies. The multivariable chain rule yields
\[
g'(t)=\sum_{i=1}^2\frac{\partial\psi}{\partial z_i}(b+t e')e'_i
=\inner{\nabla\psi(b+t e')}{e'},
\]
and differentiating once more gives
\[
g''(t)=\sum_{i,j=1}^2 e'_i
\frac{\partial^2\psi}{\partial z_i\partial z_j}(b+t e')e'_j
=(e')^\top\nabla^2\psi(b+t e')e'.
\]
The fundamental theorem of calculus, applied twice, gives the integral remainder:
\[
\begin{aligned}
g(1)-g(0)-g'(0)
&=\int_0^1\bigl(g'(t)-g'(0)\bigr)\,dt\\
&=\int_0^1\int_0^t g''(s)\,ds\,dt
=\int_0^1(1-s)g''(s)\,ds.
\end{aligned}
\]
Since $\psi(b)=1$ and $\nabla\psi(b)=b=(0,1)$, the first-order term is $g'(0)=\inner{b}{e'}=e'_2$. For any matrix $A$, Cauchy--Schwarz and the definition of the operator norm give
\[
|(e')^\top A e'|\le\norm{e'}\norm{Ae'}
\le\norm{A}_{\op}\norm{e'}^2.
\]
Consequently,
\[
\begin{aligned}
\left|\norm{b+e'}-1-e'_2\right|
&\le\int_0^1(1-t)
\left|(e')^\top\nabla^2\psi(b+t e')e'\right|\,dt\\
&\le\frac{\norm{e'}^2}{1-R}\int_0^1(1-t)\,dt
=\frac{\norm{e'}^2}{2(1-R)}.
\end{aligned}
\]
The factor $1/2$ comes from the integral, not from an additional matrix estimate. Combining the last bound with \eqref{p:eq:first-distance}--\eqref{p:eq:first-coordinate} proves
\begin{equation}
\bigl|\norm{z'}-1\bigr|
\le |e'_2|+\frac{\norm{e'}^2}{2(1-R)}
<\frac1{25}+\frac{(3/20)^2}{2(1-3/20)}
=\frac1{25}+\frac9{680}<\frac3{50}.
\label{p:eq:radial}
\end{equation}

By \eqref{p:eq:first-distance}, $z'\in C_b$. Formula~\eqref{p:eq:Tb} gives
\[
\norm{T(z')-a}
=\left\|(\norm{z'}-1)\frac{z'}{\norm{z'}}\right\|
=\bigl|\norm{z'}-1\bigr|<\frac3{50}.
\]
Allow a new factor $\widetilde q\in[1,101/100]$ and a new perturbation $\widetilde\xi$ with $\norm{\widetilde\xi}\le1/100$. They need not equal those in the preceding step. For $z''=\widetilde q(T(z')+\widetilde\xi)$,
\begin{equation}
\norm{z''-a}
\le\frac{101}{100}\left(\frac3{50}+\frac1{100}\right)+\frac1{100}
=\frac{807}{10000}<\frac18.
\label{p:eq:return}
\end{equation}
The base case $z_0=a$ lies in $\overline B(a,r)$. The estimates show that each even iterate generates an odd iterate satisfying both \eqref{p:eq:first-distance} and \eqref{p:eq:first-coordinate}, followed by an even iterate in $B(a,r)$. This proves \eqref{p:eq:invariant}. The same induction starts from any $z_0\in\overline B(a,r)$.

We have not asserted that the two entire balls map into each other in one step. The return estimate uses the extra coordinate information carried by the odd iterate. Finally,
\[
\norm{z_{2j}}\le1+\frac18<\frac{23}{20},\qquad
\norm{z_{2j+1}}<1+\frac3{20}=\frac{23}{20}.
\]
The cone inclusions and positive scaling give the claimed smoothness and uniqueness.
\end{proof}

\begin{remark}\label{p:rem:tangent-maps}

The tangent maps explain why a two-cycle argument works. Using the Hessian formula for $\psi(z)=\norm{z}$ derived in the preceding proof, we have $DT(z)=I-\nabla^2\psi(z)$ on either smooth cone. In particular, on the unit circle, $DT(z)=zz^\top$. Thus
\begin{equation}
DT(a)=\begin{pmatrix}1&0\\0&0\end{pmatrix},\qquad
DT(b)=\begin{pmatrix}0&0\\0&1\end{pmatrix},\qquad
D(T\circ T)(a)=DT(b)DT(a)=0.
\label{p:eq:derivatives}
\end{equation}
The tangent maps $DT(a)$ and $DT(b)$ describe the first-order transport of a small displacement, not its exact evolution. At $a$, only the horizontal component survives; at $b$, only the vertical component survives. Hence a displacement passes through the linearized updates as
\[
(h_1,h_2)\xmapsto{\ DT(a)\ }(h_1,0)
\xmapsto{\ DT(b)\ }(0,0).
\]
Since $T(a)=b$, $T(b)=a$, and $T\circ T$ is $C^2$ near $a$, Taylor's formula and \eqref{p:eq:derivatives} give
\[
T(T(a+h))-a=O(\norm{h}^2)\qquad(h\to0).
\]
Thus the first-order error cancels, but a second-order remainder may remain. This statement concerns the auxiliary unperturbed unit-step map, not the actual rescaled recursion with $q_n$, $u$, and $\eta_{n+1}$. It explains the two-cycle argument; the explicit estimates above establish trapping for \mbox{that noisy, time-varying recursion.}

\end{remark}


Thus the rescaled iterates remain bounded for every noise path in the prescribed ball. We now transfer this deterministic control to the original stochastic iteration.

\subsubsection{Returning to the original scale and almost-sure convergence}\label{p:sec:quantifiers}

The vanishing step size turns bounded rescaled iterates into convergence to $0$. The next lemma also verifies the moment and excitation properties of the uniform noise.

\begin{lemma}[Uniform ball noise and almost-sure convergence]\label{p:lem:almost-sure-convergence}
Under the setting of Theorem~\ref{p:thm:main}, the noise is conditionally centered with $\Cov(\eta_{n+1}\mid\mathcal F_n)=I_2/160000$, and satisfies
\begin{equation}
\begin{aligned}
\sup_{n\ge0}\E[\norm{\eta_{n+1}}^4\mid\mathcal F_n]&\le(1/200)^4,\\
\inf_{\norm{h}=1}\E[|\inner{h}{\eta_{n+1}}|\mid\mathcal F_n]&=\frac4{600\pi}>0
\qquad(n\ge0).
\end{aligned}
\label{p:eq:noise-conditions}
\end{equation}
Moreover, \eqref{p:eq:decay} holds with probability $1$. In particular, $\PP(\lim_{n\to\infty}x_n=0)=1$ for every fixed admissible tilt and initial point, and for an independent random initialization supported \mbox{in the stated region.}
\end{lemma}
\begin{proof}
Since $x_n=\gamma_n z_n$, Lemma~\ref{p:lem:normalized-trapping} gives \eqref{p:eq:decay}. The cone inclusions show simultaneously that every physical iterate is nonzero and is a differentiability point. All estimates hold for every perturbation path in the specified support; no averaging or probabilistic cancellation was used.

Let $r_\eta=1/200$ and let $U_{r_\eta}$ denote the uniform probability measure on $\overline B(0,r_\eta)\subset\R^2$. For deterministic initialization, a concrete probability space is
\[
 (\Omega,\mathcal F,\PP)
 =\bigl((\R^2)^{\mathbb N},\,
        \mathcal B(\R^2)^{\otimes\mathbb N},\,
        U_{r_\eta}^{\otimes\mathbb N}\bigr).
\]
The variable $\eta_i$ is the $i$th coordinate; take $\mathcal F_n=\sigma(\eta_1,\ldots,\eta_n)$. By induction, $x_n$ and its uniquely determined subgradient are measurable with respect to this past. Independence and \mbox{symmetry of the noise give}
\[
\E[\eta_{n+1}\mid\mathcal F_n]=0.
\]
Thus the noise is a martingale difference sequence. Its density is $1/(\pi r_\eta^2)$ on the ball of radius $r_\eta=1/200$ centered at $0$, and $0$ outside. It is absolutely continuous, although its density is not positive everywhere in $\R^2$.

The single-vector moment calculations use
$(\R^2,\mathcal B(\R^2),U_{r_\eta})$ with coordinate
$\eta(z)=z$; write $\E_{U_{r_\eta}}$ for this expectation.
For $\eta=(\eta^{(1)},\eta^{(2)})$, symmetry gives mean $0$ and off-diagonal \mbox{covariance $0$. In polar coordinates,}
\[
\E_{U_{r_\eta}}[(\eta^{(1)})^2]
=\frac1{\pi r_\eta^2}\int_0^{r_\eta}\int_0^{2\pi}
\rho^2\cos^2\theta\,\rho\,d\theta\,d\rho
=\frac{r_\eta^2}{4}.
\]
The same holds for the second coordinate. Hence $\Cov(\eta)=r_\eta^2 I_2/4=I_2/160000$, giving the stated covariance. Also $\E_{U_{r_\eta}}\norm{\eta}^m\le r_\eta^m$ for every $m>0$, so all positive moments are bounded.

Every unit direction receives nontrivial noise. By rotational invariance, for any unit $h$,
\[
\E_{U_{r_\eta}}|\inner{h}{\eta}|
=\frac1{\pi r_\eta^2}\int_0^{r_\eta}\int_0^{2\pi}
\rho|\cos\theta|\,\rho\,d\theta\,d\rho
=\frac{4r_\eta}{3\pi}=\frac4{600\pi}>0.
\]
By symmetry, the expected positive part of this projection is half that number. Thus the noise is omnidirectional, with a common positive lower bound for the expected absolute projection over all unit directions. Independence makes these same identities and bounds conditional on $\mathcal F_n$, proving \eqref{p:eq:noise-conditions}.

The step sizes of Theorem~\ref{p:thm:main} are positive, decrease to $0$, and satisfy
\[
\sum_{n=0}^{\infty}\gamma_n=\infty,
\qquad \sum_{n=0}^{\infty}\gamma_n^2<\infty.
\]
With probability $1$, every noise vector lies in the closed ball $\overline B(0,1/200)$, because the intersection of these countably many probability-$1$ events still has probability $1$. On that event, the deterministic estimates of Section~\ref{p:sec:cycle} apply at all times. Therefore $x_n\to0$ almost surely. The stronger pathwise statement is that any prescribed sequence inside this ball gives the same trapping conclusion, irrespective of independence.

\[
\begin{gathered}
\Omega_0=\bigcap_{n=1}^{\infty}\{\norm{\eta_n}\le1/200\},
\qquad \PP(\Omega_0)=1,\\[0.3em]
\Omega_0\subseteq\{\lim_{n\to\infty}x_n=0\}
\quad\Longrightarrow\quad
\PP\!\left(\lim_{n\to\infty}x_n=0\right)=1.
\end{gathered}
\]

The bounds use $u$ only through $\norm{u}\le1/200$. Consequently, the same noise law and initial point work for every $u\in V:=B(0,1/200)$. Since $V$ is a nonempty open set, it has positive Lebesgue measure. The failure is not confined to a null set of tilt parameters.

Likewise, any initial point in the nonempty open set $U:=\gamma_0B(a,1/8)$ gives $z_0\in B(a,1/8)$ and starts the same induction; Lemma~\ref{p:lem:normalized-trapping} also includes its closure. If $x_0$ has distribution $\rho_0$, is independent of the entire noise sequence, and is supported in $U$ or its closure, use
$\Omega=\R^2\times(\R^2)^{\mathbb N}$ with the product Borel sigma-field and
$\PP=\rho_0\otimes U_{r_\eta}^{\otimes\mathbb N}$.
Take $\mathcal F_n=\sigma(x_0,\eta_1,\ldots,\eta_n)$. All conclusions then hold almost surely. More generally, an independent initialization assigning positive probability to $U$ has a failure probability at least $\PP(x_0\in U)$.
\end{proof}

\begin{proof}[Proof of Theorem~\ref{p:thm:main}]
Lemmas~\ref{p:lem:function-class} and~\ref{p:lem:sharp-repulsiveness} give the stated properties of $f_u$ and $0$. On the probability-$1$ event $\bigcap_{n\ge1}\{\norm{\eta_n}\le1/200\}$, Lemma~\ref{p:lem:normalized-trapping} gives smoothness along the trajectory and
\begin{equation}\label{p:eq:decay}
\norm{x_n}=\gamma_n\norm{z_n}\le\frac{23}{20(n+100)}\longrightarrow0.
\end{equation}
The estimate holds for every allowed initial point and noise path, so it also applies to an independent random initialization supported in the stated set. Lemma~\ref{p:lem:almost-sure-convergence} verifies the conditional centering, fourth-moment bound, and excitation of every direction.
\end{proof}

The same argument already covers
\[
\gamma_n=(n+100)^{-\alpha},\qquad
x_0=\gamma_0a,\qquad \alpha\in(1/2,1].
\]
Indeed,
\[
q_n=\left(1+\frac1{n+100}\right)^\alpha\in[1,101/100],
\]
which is the only step-ratio property used in the two-cycle argument. Now $\norm{x_n}\le(23/20)(n+100)^{-\alpha}$. The two series conditions also hold. In particular, the counterexample does not rely on the endpoint $\alpha=1$ and also applies to the usual stochastic-approximation step range $(1/2,1)$.

\subsection{A coercive counterexample preserving convergence to \mbox{the nonminimal point \texorpdfstring{$0$}{0}}}\label{p:sec:coercive}

The following modification makes the objective coercive without changing the local geometry at $0$ or any of the trapped trajectories. Define
\begin{equation}
\widehat f_u(z)=f_u(z)+10\max\{0,\norm{z}-1/10\}.
\label{p:eq:coercive}
\end{equation}

The statement is given in Corollary~\ref{p:cor:coercive}.

\begin{proof}
The radial term vanishes throughout the open ball $B(0,1/10)$. From \eqref{p:eq:decay},
\[
\norm{x_n}\le\frac{23}{2000}<\frac1{10}.
\]
Thus it vanishes on a neighborhood of each trapped iterate and of $0$. The two objectives have identical local values and subdifferentials there. Starting from the same initial point and using the same noise, their recursions coincide inductively along the whole trapped trajectory. The active-manifold property, \eqref{p:eq:repulsive}, and nonminimality at $0$ are all unchanged.

To prove coercivity, note that $\min\{z_1,z_2\}\le\norm{z}$, so $\norm{z}-\min\{z_1,z_2\}\ge0$. Also,
\[
\max\{0,-z_1-z_2\}\le|z_1+z_2|\le\sqrt2\norm{z},
\qquad -\inner{u}{z}\ge-\norm{u}\norm{z}.
\]
Consequently $f_u(z)\ge-(2\sqrt2+\norm{u})\norm{z}$. Since $\max\{0,\norm{z}-1/10\}\ge\norm{z}-1/10$,
\[
\widehat f_u(z)\ge(10-2\sqrt2-\norm{u})\norm{z}-1
\longrightarrow+\infty\quad\text{as }\norm{z}\to\infty.
\]
The coefficient is positive for the stated tilts. Continuity and coercivity imply that a sufficiently low nonempty sublevel set is compact; minimizing over it gives an attained global minimum.

The radial term is globally $10$-Lipschitz and definable, since the norm and the maximum of two definable functions are definable. Thus the modified function retains both properties. To see that weak convexity still fails, fix a sufficiently small $r_0>0$ and consider $(r_0+t,-r_0+t)$ for small $|t|$. This entire short line segment lies in $B(0,1/10)$, where the radial modification is $0$. Its norm term has derivative $0$ at $t=0$, and the same downward derivative jump from $3$ to $-1$ persists, up to the common derivative of the linear tilt. The argument of Section~\ref{p:sec:properties} therefore \mbox{still excludes weak convexity.}
\end{proof}

For $\alpha\in(1/2,1)$, use $\gamma_n=(n+1000)^{-\alpha}$ and $x_0\in\gamma_0\overline B(a,1/8)$. Then $q_n\le101/100$ and
\[
\norm{x_n}\le\frac{23}{20}\gamma_n
\le\frac{23}{20}\,1000^{-\alpha}
<\frac{23}{20\sqrt{1000}}<\frac1{10}.
\]
The same coercive argument applies. Thus the attained-global-minimum version is compatible with the full open interval of exponents used in the related stochastic subgradient setting.

\section{Proofs for Gaussian noise avoidance}\label{g:sec:proofs}\label{g:sec:descent-proofs}
This appendix supplies the proofs for Section~\ref{sec:gaussian}.
We first establish the difference quotient construction and the
continuity of its state transition, including the measurability and
stability arguments. Translation orbits and fixed-step escape then
lead to positive expected descent. The convergence windows and their
stopped identities lead to expected descent $0$. We finish with the
finite-set and generic-tilt corollary.
\subsection{Properties of the difference quotient functions}\label{g:sec:profiles}\label{g:sec:transition}
We prove compactness of the difference quotient family and transfer the local subgradient bound to every limiting function on all of $\R^d$. Local gradient identification connects these functions to the original stochastic iteration; the continuity of their noisy transition is established in Section~\ref{g:sec:profiletools}.
\begin{lemma}[The local Fr\'echet subgradient bound]\label{g:lem:localgap}
There are a definable open neighborhood $U$ of $0$ and $a>0$ such that
\begin{equation}\label{g:eq:localgap}
 \norm{v}\ge a\qquad(x\in U,\ v\in\pf f(x)).
\end{equation}
\end{lemma}
\begin{proof}
If the assertion failed, for every $j\ge1$ we could choose
\[
 x_j\in B(0,1/j),\qquad v_j\in\pf f(x_j),\qquad \norm{v_j}<1/j.
\]
Then $(x_j,f(x_j),v_j)\to(0,f(0),0)$, so the definition of the
limiting subdifferential gives $0\in\pl f(0)$, a contradiction.
A sufficiently small open ball may be used for $U$.
\end{proof}
Choose $R>0$ so that $\overline B(0,R)\subset U$ and $f$ is
$L$-Lipschitz on this ball. These are the constants used in
Section~\ref{g:sec:rescaled}, and \eqref{g:eq:localgap} gives
\eqref{g:eq:intro-gap}. Fix $0<r<R$, set $K_0=\overline B(0,r)$,
and define
\[
 g(x)=\begin{cases}\nabla h(x),&h\text{ is differentiable at }x,\\
                    0,&\text{otherwise}.
       \end{cases}
\]
\begin{wrapfigure}{r}{0.34\textwidth}
\centering
\begin{tikzpicture}[scale=.74,font=\small]
 \fill[blue!6] (-2.05,-1.9) rectangle (2.05,1.9);
 \filldraw[fill=blue!13,draw=blue!65!black] (0,0) circle (1.38);
 \filldraw[fill=blue!23,draw=blue!65!black,dashed] (0,0) circle (.75);
 \fill (0,0) circle (.035);
 \node[above right] at (0,0) {$0$};
 \node at (0,-.42) {$K_0$};
 \node at (0,1.05) {$h=f$};
 \node at (0,1.65) {$h$ globally $L$-Lipschitz};
 \node at (0,-1.68) {$\overline B(0,R)$};
\end{tikzpicture}
\caption{The extension agrees with $f$ around $K_0$.}\label{g:fig:extension}
\end{wrapfigure}
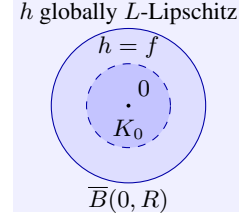
The extension supplies a global Lipschitz bound while preserving $f$ locally, as Figure~\ref{g:fig:extension} illustrates. We next identify the actual selections with its gradient near the target.
\begin{lemma}[Local extension and gradient identification]\label{g:lem:extension}
The extension $h$ in \eqref{g:eq:extension} is definable and globally
$L$-Lipschitz, and $h=f$ on an open neighborhood of $K_0$.
Its gradient representative $g$, defined above,
is bounded, definable, and Borel. Almost surely, for every $n\ge1$,
\begin{equation}\label{g:eq:localiteration}
 X_n\in K_0\quad\Longrightarrow\quad V_n=g(X_n).
\end{equation}
The same assertion holds after conditioning on any positive-probability
event in a past sigma-field.
\end{lemma}

\begin{proof}
The extension formula is Lemma~\ref{g:fact:extension}. For
$x\in\overline B(0,R)$, the Lipschitz inequality gives $h(x)\ge f(x)$,
while choosing $y=x$ gives the reverse inequality. Hence $h=f$ on
$B(0,R)$, an open neighborhood of $K_0$, and $\norm{g}\le L$.
The infimum and the derivative limits are definable, so $h$ and $g$
are definable and $g$ is Borel.

Choose a definable open set $O$ on which $h$ is $C^1$, with
$\mathcal L^d(O^c)=0$, using the regularity recalled in
Section~\ref{g:sec:prelim}. For $n\ge1$,
\[
 \mathcal L(X_n\mid\mathcal F_{n-1})
 =N(X_{n-1}-\gamma_{n-1}V_{n-1},\gamma_{n-1}^2\Sigma).
\]
Its nondegenerate density implies $X_n\in O$ for every $n\ge1$
almost surely. Consequently, on $\{X_n\in K_0\}$,
\[
 V_n\in\pc f(X_n)=\pc h(X_n)=\{\nabla h(X_n)\}=\{g(X_n)\}.
\]
Conditioning on a positive-probability event $B\in\mathcal F_m$
preserves every null event. Independence also preserves the joint law
of the noise vectors after time $m$ under this conditioning.
\end{proof}

For later use, \eqref{g:eq:profile} and \eqref{g:eq:translation} give
\begin{equation}\label{g:eq:profileidentity}
 T_zH_{x,s}=H_{x+sz,s},\qquad
 H_{x,s}(z)=\frac{h(x+sz)-h(x)}s.
\end{equation}
Introduce the space containing both actual difference quotient functions and their limits,
\[
 \mathcal D=\overline{\{H_{x,s}:x\in\R^d,\ s>0\}},
\]
where the closure is taken for uniform convergence on compact sets.
Fix a compatible metric $d_0$ for this topology.
\begin{lemma}[Compactness and recentering]\label{g:lem:profilecompactness}
The spaces $\mathcal D$ and $\mathcal C\subset\mathcal D$ are nonempty
and compact. Their members vanish at $0$ and are $L$-Lipschitz.
The map $(H,z)\mapsto T_zH$ is continuous, and for every $z\in\R^d$,
\[
 H\in\mathcal D\ \Longrightarrow\ T_zH\in\mathcal D,
 \qquad H\in\mathcal C\ \Longrightarrow\ T_zH\in\mathcal C.
\]
Every sequence $H_{x_j,s_j}$ with $x_j\to0$ and $s_j\to0$, $s_j>0$,
has a subsequence converging locally uniformly to a member of
$\mathcal C$.
\end{lemma}
\begin{proof}
The normalization and Lipschitz bound follow directly from
\eqref{g:eq:profile}. Arzel\`a--Ascoli, followed by a
diagonal argument over balls, makes their local-uniform closure
$\mathcal D$ a compact metric space. The sequence definition
\eqref{g:eq:profilelimit} is equivalently
\[
 \mathcal C=\bigcap_{j\ge1}
 \overline{\{H_{x,s}:\norm{x}\le1/j,\ 0<s\le1/j\}},
\]
with closures taken in $\mathcal D$. Each set in this decreasing
intersection is nonempty and compact, so $\mathcal C$ is nonempty
and compact. The equivalence follows by choosing an approximating
difference quotient function from the $j$th set with distance less than $1/j$ in a metric
for local-uniform convergence.

The map $(H,z)\mapsto T_zH$ is continuous: local-uniform convergence
of $H_j$ and convergence of $z_j$ imply local-uniform convergence of
$T_{z_j}H_j$. Also $T_zT_w=T_{z+w}$. For actual difference quotient functions,
\eqref{g:eq:profileidentity} gives invariance of $\mathcal D$ under
recentering. If $H_{x_j,s_j}\to H\in\mathcal C$ with vanishing
centers and scales, then
\[
 T_zH_{x_j,s_j}=H_{x_j+s_jz,s_j}\longrightarrow T_zH,
 \qquad x_j+s_jz\to0.
\]
Thus $T_zH\in\mathcal C$ as well. Finally, for any $x_j\to0$ and
$s_j\to0$ with $s_j>0$, compactness of $\mathcal D$ gives a
convergent subsequence of $H_{x_j,s_j}$. Its limit belongs to
$\mathcal C$ by definition; the scales may be made decreasing by a
further subsequence.

\end{proof}

We now prove Lemma~\ref{g:lem:profilegap}, which gives the exact
subgradient scaling and the inherited bound used in the main text.
\begin{proof}[Proof of Lemma~\ref{g:lem:profilegap}]
Substituting $y=x+sw$ in \eqref{g:eq:frechet} cancels the scale
in the function value against the scale in the displacement. This
proves \eqref{g:eq:exact-subgradient-scaling}. Local equality $h=f$
and \eqref{g:eq:localgap} give the stated bound for actual difference quotients.

To pass the same constant to the limit, take actual difference quotient functions $H_j=H_{x_j,s_j}\to H$, with $x_j\to0$ and
$s_j\to0$, and fix $v\in\pf H(u)$. Lemma~\ref{g:fact:support} supplies a $C^1$ function $\phi$
and a small closed ball $B$ about $u$ such that
\[
 \phi(u)=H(u),\qquad \nabla\phi(u)=v,\qquad
 \phi(w)<H(w)\quad(w\in B\setminus\{u\}).
\]
Choose $u_j\in\arg\min_B(H_j-\phi)$ and put
$\varepsilon_j=\norm{H_j-H}_{\infty,B}\to0$. Then
\[
 0\le H(u_j)-\phi(u_j)
 \le H_j(u_j)-\phi(u_j)+\varepsilon_j
 \le H_j(u)-\phi(u)+\varepsilon_j\le2\varepsilon_j.
\]
Compactness and the unique zero of $H-\phi$ on $B$ give $u_j\to u$.
The minimizers are eventually interior, so the smooth perturbation rule
and the scaling in \eqref{g:eq:profile} give
\[
 \nabla\phi(u_j)\in\pf H_j(u_j)
 =\pf h(x_j+s_ju_j).
\]
Indeed, $y=x_j+s_jw$ in \eqref{g:eq:frechet} cancels the two scale
factors. Since $x_j+s_ju_j\to0$, local equality and
\eqref{g:eq:localgap} imply, for all large $j$,
\[
 \nabla\phi(u_j)\in\pf f(x_j+s_ju_j),\qquad
 \norm{\nabla\phi(u_j)}\ge a.
\]
Continuity of $\nabla\phi$ gives $\norm{v}\ge a$.
\end{proof}

The conclusion is global for each limit $H$, although the original
bound was local for $f$. There is no loss in the constant $a$.
The global definability of every $H\in\mathcal C$, needed for the
escape argument, is established in Lemma~\ref{g:fact:codes} by the
definable compactification described in Section~\ref{g:sec:orbits}.

We next relate these difference quotient functions to the original iteration. For $n\ge1$,
$X_n\in K_0$, and a fixed scale $s>0$, write
$H=H_{X_n,s}$, $v=V_n=g(X_n)=\nabla H(0)$, and
$\lambda=\gamma_n/s$. Lemma~\ref{g:lem:extension} justifies these
gradient identities almost surely. The scaled displacement is
\begin{equation}\label{g:eq:physical-displacement}
 \Delta=\frac{X_{n+1}-X_n}{s}=\lambda(-v+\eta_{n+1}).
\end{equation}
Thus $s\Delta$ is the original displacement, and
$H(0)-H(\Delta)=[h(X_n)-h(X_{n+1})]/s$ is its value decrease
in units of $s$.

The same displacement identity holds at every pre-exit update of a window, using its fixed spatial scale $s$. Choosing $s=\gamma_n$ makes the first rescaled step a unit step; later ratios are controlled in Lemma~\ref{g:lem:window}.

\subsection{The state transition and its stability}\label{g:sec:profiletools}
We define the noisy transition on difference quotient states and prove its continuity after averaging over the noise, which is needed to pass approximate stationarity to a weak limit. Set
$E=\mathcal D\times\overline B(0,L)$ with its product metric and
Borel sigma-field, generated by the open sets. An actual state has $v=\nabla H(0)$, but we retain
the two components separately because uniform convergence of functions
need not preserve the gradient at the center. For $(H,v)\in E$,
$\lambda>0$, use the noise space
$(\R^d,\mathcal B(\R^d),\Gamma_\Sigma)$, where
$\Gamma_\Sigma=N(0,\Sigma)$ and $\eta(z)=z$. With $(H,v)$ and
$\lambda$ fixed, define
\begin{equation}\label{g:eq:kernel}
 \Delta=\lambda(-v+\eta),\qquad
 H^+=T_\Delta H,\qquad v^+=\nabla H(\Delta).
\end{equation}
The components in \eqref{g:eq:kernel} are determined by the current
state and the noise. In particular,
\[
 H^+(u)=H(u+\Delta)-H(\Delta),\qquad
 v^+=\nabla H(\Delta)=\nabla H^+(0)
\]
whenever the derivative exists. Thus the update records the same
function viewed from the new point, together with its gradient there.
The gradient is assigned $0$ at nondifferentiability points;
Lemma~\ref{g:lem:measurable_gradient} gives a jointly Borel choice.
The Gaussian displacement has a density and therefore avoids these
exceptional points almost surely. By Lemma~\ref{g:lem:profilecompactness},
recentering preserves $\mathcal D$ and $\mathcal C$.

For a fixed input $(H,v)$ and a Borel set $A\subset E$, define
\begin{equation}\label{g:eq:kernel-law}
 P_\lambda((H,v),A)
 =\PP_\eta\bigl((T_{\lambda(-v+\eta)}H,
                 \nabla H(\lambda(-v+\eta)))\in A\bigr).
\end{equation}
This is the probability that one update from $(H,v)$ lands in $A$;
the notation $\PP_\eta$ means probability under $\Gamma_\Sigma$.
Likewise, $\E_\eta$ below integrates only this noise coordinate.
Joint measurability makes $P_\lambda$ a Markov kernel in the sense of
Definition~\ref{g:def:stationary}. Its action on a bounded Borel test
function $\psi:E\to\R$ is
\begin{equation}\label{g:eq:energydef}
 \begin{aligned}
 P_\lambda\psi(H,v)
 &=\int_E\psi(J,w)\,P_\lambda((H,v),d(J,w))\\
 &=\E_\eta\psi(H^+,v^+).
 \end{aligned}
\end{equation}
In particular, if the current state has law $\mu$ and the noise is
independent, its next-state law is given by
\[
 (\mu P_\lambda)(A)=\int_E P_\lambda((H,v),A)\,d\mu(H,v).
\]
Stationarity for the unit-step kernel $P=P_1$ means that this
next-state law equals $\mu$. It does not require a realized state to
remain unchanged after the update.

The expected descent at ratio $\lambda$ is defined by
\begin{equation}\label{g:eq:scaled-descent}
 d_\lambda(H,v)=\E_\eta[H(0)-H(\lambda(-v+\eta))].
\end{equation}
These definitions extend those of Section~\ref{g:sec:rescaled}:
$P=P_1$ and $d=d_1$. Probabilities on
$\mathcal C\times\overline B(0,L)$ are viewed as probabilities on
$E$ assigning mass $1$ to that subset. The Lipschitz bound gives
\[
 |d_\lambda(H,v)|\le\lambda L(L+\E_\eta\norm{\eta}).
\]

The next result supplies the continuity needed to take limits in the
difference quotient function transition \eqref{g:eq:kernel}. We then verify that stationarity
identifies the recorded vector with the current gradient.

Uniform monotonicity for the definable family bounds the number of
monotone pieces on every coordinate line. This gives local $L^1$
control of gradient differences by function-value differences.
Integration against the Gaussian density then gives continuity of
the update. We give the measurability and continuity proofs next.
The following arguments supply the two regularity facts used to define the
kernel in \eqref{g:eq:kernel} and to pass stopped occupation measures to a
stationary limit. The gradient representative is Borel by countable
difference-quotient tests. Stability then follows from uniform
definable monotonicity and Gaussian averaging.
\subsubsection{A measurable gradient representative}
\begin{lemma}[A jointly measurable gradient]\label{g:lem:measurable_gradient}
On the compact family $\mathcal D$ of normalized $L$-Lipschitz
difference quotient functions, the map $(H,u)\mapsto\nabla H(u)$, with value $0$ whenever
$H$ is not differentiable at $u$, is jointly Borel and takes values
in $\overline B(0,L)$. Moreover,
\[
 G=\{(H,v)\in E:H\text{ is differentiable at }0,
                    \ v=\nabla H(0)\}
\]
is a Borel subset of $E$.
\end{lemma}
\begin{proof}
Let $e_i$ denote the $i$th standard basis vector. For each coordinate
form the limits
\[
 b_i(H,u)=\lim_{m\to\infty}m[H(u+e_i/m)-H(u)],
\]
assigning $0$ when a limit does not exist. Evaluation is continuous,
so the limits and their existence sets are Borel. Where all coordinate
limits exist, differentiability with derivative $b$ is equivalent to
\[
 \lim_{k\to\infty}
 \sup_{\substack{q\in\mathbb Q^d\\0<\norm{q}\le1/k}}
 \frac{|H(u+q)-H(u)-\ip{b}{q}|}{\norm{q}}=0.
\]
For fixed $H,u,b$, the quotient is continuous in each nonzero
increment. Density of rational increments therefore extends the test
to all real increments. The test is Borel as a countable supremum
followed by a countable limit. Using $b$ on this Borel set and $0$
elsewhere gives the required joint representative; the Lipschitz bound
implies $\norm{\nabla H(u)}\le L$.

All difference quotient functions satisfy $H(0)=0$. Thus a pair $(H,v)\in E$ \mbox{belongs to $G$ exactly when}
\[
 \lim_{k\to\infty}
 \sup_{\substack{u\in\mathbb Q^d\\0<\norm{u}\le1/k}}
 \frac{|H(u)-\ip{v}{u}|}{\norm{u}}=0.
\]
Continuity in the increment again identifies this rational test with
differentiability at $0$ with derivative $v$. Each expression is
Borel in $(H,v)$, so $G$ is Borel.
\end{proof}
\subsubsection{Continuity after Gaussian averaging}
\begin{proposition}[Stability of one noisy update]\label{g:prop:kernel}
Let $H_j,H\in\mathcal D$, $v_j,v\in\overline B(0,L)$, and
$\lambda_j,\lambda>0$. Suppose $H_j\to H$ uniformly on compact
sets, $v_j\to v$, and $\lambda_j\to\lambda$.
Apply \eqref{g:eq:kernel} to these inputs. Then
\[
 (H_j^+,v_j^+)\ \Rightarrow\ (H^+,v^+),\qquad
 d_{\lambda_j}(H_j,v_j)\longrightarrow d_\lambda(H,v).
\]
In particular, $P_\lambda\psi(H,v)$ is jointly continuous in
$(H,v,\lambda)$ for every continuous $\psi:E\to\R$.
\end{proposition}

\begin{proof}[Proof of Proposition~\ref{g:prop:kernel}]
We first estimate gradient errors directly from function errors.
Fix $R>0$, set $Q_R=(-R,R)^d$, and take two actual difference quotient functions $H,J$ from
\eqref{g:eq:profile}. Their differences along coordinate lines form
one definable family; both sets of difference quotient function parameters and all
transverse coordinates are included as parameters. Uniform
monotonicity gives a common upper bound $M$ on the number of
monotone intervals of each such difference
\cite[\S\S4.1, 4.2 and 4.4]{vddm}.

Write $q(t)=H(y+te_i)-J(y+te_i)$ on a coordinate segment in $Q_R$,
with $y\in e_i^\perp$ and $|y_k|<R$ for $k\ne i$.
This function is Lipschitz and hence absolutely continuous. For
its monotone partition $-R=t_0<\cdots<t_m=R$, $m\le M$,
\[
 \int_{-R}^R|q'(t)|\,dt
 =\sum_{\ell=1}^m|q(t_\ell)-q(t_{\ell-1})|
 \le2M\sup_{z\in Q_R}|H(z)-J(z)|.
\]
For each fixed pair $H,J$, the slice derivatives agree with the
corresponding gradient components for almost every $t$ on almost
every line.
Integrating over $y$ and summing over coordinate directions gives
\begin{equation}\label{g:eq:gradientbound}
 \int_{Q_R}\norm{\nabla H-\nabla J}
 \le C_R\sup_{z\in Q_R}|H(z)-J(z)|,\qquad
 C_R=2Md(2R)^{d-1}.
\end{equation}
The constant is independent of the two difference quotient functions.

To extend this estimate to $\mathcal D$, let actual difference quotient functions
$F_k\to F$ locally uniformly. Equation~\eqref{g:eq:gradientbound}
makes $\nabla F_k$ Cauchy in $L^1(Q_R)$ for every $R$; let its
local limit be $G$. For a smooth compactly supported test function
$\phi\in C_c^\infty(\R^d)$,
\[
 \int G_i\phi
 =\lim_k\int\partial_iF_k\phi
 =-\lim_k\int F_k\partial_i\phi
 =-\int F\partial_i\phi.
\]
The limit $F$ is $L$-Lipschitz, so $G=\nabla F$ almost everywhere.
Approximating each of $H,J\in\mathcal D$ by actual difference quotient functions now
extends \eqref{g:eq:gradientbound} to all of $\mathcal D$.
In particular, $H_j\to H$ locally uniformly implies
$\nabla H_j\to\nabla H$ in local $L^1$.
This argument does not require the limits to be definable.

For the random update, use a metric $d_0$ on $\mathcal D$ for local
uniform convergence and the sum product metric on $E$.
Let $\varphi_\Sigma$ denote the density of $N(0,\Sigma)$.
In the common displacement variable $z$, the Gaussian densities are
\[
 p_j(z)=\lambda_j^{-d}\varphi_\Sigma(z/\lambda_j+v_j),\qquad
 p(z)=\lambda^{-d}\varphi_\Sigma(z/\lambda+v).
\]
They converge pointwise and integrate to $1$. Thus
\[
 \norm{p_j-p}_{L^1}
 =2-2\int\min\{p_j,p\}\longrightarrow0,
\]
by dominated convergence with dominating function $p$.
For a bounded Lipschitz test $\psi$ on $E$, both updates can be integrated over the same displacement $z$:
\[
 P_{\lambda_j}\psi(H_j,v_j)=\int\psi(T_zH_j,\nabla H_j(z))p_j(z)\,dz,
 \quad
 P_\lambda\psi(H,v)=\int\psi(T_zH,\nabla H(z))p(z)\,dz.
\]
Add and subtract the first integrand with density $p$, then separate $Q_R$ from its complement. This gives
\begin{align*}
 &|P_{\lambda_j}\psi(H_j,v_j)-P_\lambda\psi(H,v)|\\
 &=\Big|\int\Big\{\psi(T_zH_j,\nabla H_j(z))[p_j(z)-p(z)]\\
 &\hspace{35mm}+ [\psi(T_zH_j,\nabla H_j(z))-\psi(T_zH,\nabla H(z))]p(z)\Big\}\,dz\Big|\\
 &\le\norm{\psi}_\infty\norm{p_j-p}_{L^1}\\
 &\quad+\int_{Q_R}|\psi(T_zH_j,\nabla H_j(z))-\psi(T_zH,\nabla H(z))|p(z)\,dz
       +2\norm{\psi}_\infty\int_{Q_R^c}p(z)\,dz\\
 &\le\norm{\psi}_\infty\norm{p_j-p}_{L^1}
       +\operatorname{Lip}(\psi)\int_{Q_R}d_0(T_zH_j,T_zH)p(z)\,dz\\
 &\quad+\operatorname{Lip}(\psi)\int_{Q_R}\norm{\nabla H_j(z)-\nabla H(z)}p(z)\,dz
       +2\norm{\psi}_\infty\int_{Q_R^c}p(z)\,dz\\
 &\le\norm{\psi}_\infty\norm{p_j-p}_{L^1}
       +\operatorname{Lip}(\psi)\sup_{z\in Q_R}d_0(T_zH_j,T_zH)\\
 &\quad+\operatorname{Lip}(\psi)\norm{p}_\infty C_R\sup_{z\in Q_R}|H_j(z)-H(z)|
       +2\norm{\psi}_\infty\int_{Q_R^c}p(z)\,dz.
\end{align*}
The four terms respectively control the change of density, the translated functions, the gradients through~\eqref{g:eq:gradientbound}, and the tail outside $Q_R$.
For fixed $R$, joint continuity of translation on the compact
closure $\overline Q_R$ makes the supremum tend to $0$. First let
$j\to\infty$, then $R\to\infty$. Bounded Lipschitz functions
are uniformly dense in $C(E)$; for any such approximation $\psi_\epsilon$,
\begin{align*}
 &|P_{\lambda_j}\psi(H_j,v_j)-P_\lambda\psi(H,v)|\\
 &\quad\le2\norm{\psi-\psi_\epsilon}_\infty
   +|P_{\lambda_j}\psi_\epsilon(H_j,v_j)
                         -P_\lambda\psi_\epsilon(H,v)|.
\end{align*}
This proves continuity for every continuous test, and hence the
claimed convergence in distribution.

For each fixed $\eta$, local uniform convergence and the common
Lipschitz bound give
\[
 H_j(0)-H_j(\lambda_j(-v_j+\eta))
 \to H(0)-H(\lambda(-v+\eta)),
\]
and
\[
 |H_j(0)-H_j(\lambda_j(-v_j+\eta))|
 \le L\bigl(\sup_j\lambda_j\bigr)(L+\norm{\eta}).
\]
The bound is integrable, so dominated convergence proves the
assertion about the expected descent.
\end{proof}

The current gradient is correct after one step: for every input,
the displacement has a density and $H$ is differentiable there
almost surely. Hence
\[
 v^+=\nabla H(\Delta)=\nabla H^+(0)\quad\text{almost surely}.
\]
Let $G=\{(H,v):H\text{ is differentiable at }0,\ v=\nabla H(0)\}$,
which is Borel by Lemma~\ref{g:lem:measurable_gradient}. Since
$P((H,v),G)=1$ for every input, every stationary probability satisfies
\begin{equation}\label{g:eq:stationarygradient}
 \mu(G)=\mu P(G)=\int_E P(y,G)\,d\mu(y)=1.
\end{equation}
Thus stationarity restores the gradient relation that need not be
preserved by a deterministic limit at the center of a difference quotient function.

Joint continuity on compact sets also gives
\begin{equation}\label{g:eq:uniformkernel}
 \sup_{(H,v)\in E}|P_\lambda\psi(H,v)-P\psi(H,v)|\to0,
 \qquad \norm{d_\lambda-d}_\infty\to0
 \quad\text{as }\lambda\to1.
\end{equation}

\subsection{Translation orbits and stationary distributions}\label{g:sec:orbits}
To compute stationary mean descent, we decompose a stationary distribution into components supported on individual translation orbits. Definability provides measurable orbit representatives, and each component reduces to a position process in the normal directions.

\subsubsection{Definable limits and measurable orbit representatives}
A definable parametrization of the limiting functions permits measurable choices of representatives for their translation orbits.
For nonempty compact sets $A,B$, $d_H(A,B)$ denotes their Hausdorff
distance, the larger of $\sup_{a\in A}\inf_{b\in B}\norm{a-b}$ and
$\sup_{b\in B}\inf_{a\in A}\norm{a-b}$.
A standard Borel space is a measurable space isomorphic to a Borel
subset of a complete separable metric space. Such spaces admit the
conditional distributions used below.
\begin{lemma}[A jointly definable family of limiting functions]\label{g:fact:codes}
There are a compact definable parameter set $P_0$ and a jointly definable family $(H_p)_{p\in P_0}$ such that
\[
 \mathcal C=\{H_p:p\in P_0\},\qquad
 p_j\to p\ \Longrightarrow\ H_{p_j}\to H_p\text{ locally uniformly}.
\]
In particular, every $H\in\mathcal C$ is globally definable.
\end{lemma}
\begin{proof}
We use the Lion--Speissegger compactification \cite{lion}, in the form of \cite[Theorem~5]{zell}: a bounded definable family of nonempty compact sets $A_t$ is contained in a definable family $B_p$ parametrized by a compact definable set $P$, with $p\mapsto B_p$ Hausdorff continuous.
To apply the compactification theorem, we represent each difference quotient function by a compact
set that also records its center and scale. For $\norm{x}\le1$ and
$0<s\le1$, define
\begin{equation}\label{g:eq:graphcompact}
 A_{x,s}=\overline{\left\{
 \left(\frac{u}{\sqrt{1+\norm{u}^2}},
       \frac{H_{x,s}(u)}{\sqrt{1+\norm{u}^2}},x,s\right)
 :u\in\R^d\right\}}.
\end{equation}
The closure is taken in $\R^d\times\R\times\R^d\times\R$ with
$(x,s)$ fixed. Thus $A_{x,s}$ is the closure of the transformed graph
of $H_{x,s}$, with the same center $x$ and scale $s$ attached to every
point. Its first two components record the transformed argument and
function value; its last two components retain the parameters from
which the difference quotient function was obtained. Since $H_{x,s}(0)=0$ and $H_{x,s}$
is $L$-Lipschitz, these nonempty compact sets satisfy
\[
 A_{x,s}\subset
 \overline B(0,1)\times[-L,L]\times\{x\}\times\{s\},
\]
and form a bounded definable family over the stated parameter range.
The first component has norm strictly below $1$ at every finite
argument $u$; points with first-component norm $1$ are added by the
closure to record behavior as $\norm{u}\to\infty$.
This theorem gives a compact parameter set $P$ and a
Hausdorff-continuous definable family $B_p$ containing all these sets.

Let $P_1$ consist of the parameters representing an original set
$A_{x,s}$, and let $P_0$ consist of their limiting parameters for
which the recorded center and scale are both $0$. Explicitly, set
\[
 \begin{split}
 P_1&=\{p\in P:\exists x\in\overline B(0,1),\ \exists s\in(0,1],
                    \ B_p=A_{x,s}\},\\
 P_0&=\{p\in\overline{P_1}:
 B_p\subset\overline B(0,1)\times[-L,L]\times\{(0,0)\}\}.
 \end{split}
\]
These sets are definable, and $P_0$ is compact. If $p_j\in P_1$
converges to $p\in P_0$, write $B_{p_j}=A_{x_j,s_j}$. Hausdorff
continuity and the recorded parameters give
\[
 d_H(A_{x_j,s_j},B_p)\to0,\qquad (x_j,s_j)\to(0,0).
\]
Conversely, a sequence with $(x_j,s_j)\to(0,0)$ has representing
parameters in $P_1$ with a subsequence converging into $P_0$.

For $p\in P_0$, the part of $B_p$ with spatial norm strictly below
one determines a unique global $L$-Lipschitz function $H_p$ with
$H_p(0)=0$. More precisely, a point $(y,r,0,0)\in B_p$,
where $y\in\R^d$, $\norm{y}<1$, and $r\in\R$, recovers an argument
and its function value by
\[
 u=\frac{y}{\sqrt{1-\norm{y}^2}},\qquad
 H_p(u)=\frac{r}{\sqrt{1-\norm{y}^2}}.
\]
To justify this reconstruction, extract a locally uniformly
convergent subsequence of the approaching difference quotient functions. The displayed
inverse is continuous when $\norm{y}<1$, so Hausdorff convergence
identifies the recovered value with that of the difference quotient function limit.
The interior graph determines the function, and hence
\[
 \{H_p:p\in P_0\}=\mathcal C,\qquad
 p_j\to p\quad\Longrightarrow\quad H_{p_j}\to H_p
 \ \text{locally uniformly}.
\]
The convergence follows from compactness of the normalized Lipschitz
family: every subsequential limit has the same interior graph.
Evaluation $(p,u)\mapsto H_p(u)$ is definable by the inverse formula.
Different parameters may still encode different boundary data at
infinity for the same function.

\end{proof}
The joint family is needed below: it makes equality and translation equivalence definable relations on a single parameter set, rather than separate assertions about individual functions.

\begin{lemma}[A measurable quotient by translations]\label{g:lem:orbits}
There are a standard Borel space $Q$, a Borel map $\rho:\mathcal C\to Q$ whose fibers are exactly the translation orbits, and a Borel choice of a representative function for each $q\in Q$.
\end{lemma}
\begin{proof}
Use the family from Lemma~\ref{g:fact:codes}. Strong definable choice \cite[Chapter~6, Proposition~1.2]{vdd} states that a definable family of nonempty sets $D_t$ admits a definable choice $b(t)\in D_t$ with $D_t=D_{t'}\Rightarrow b(t)=b(t')$. We apply it first to remove duplicate function parameters and then to choose one function from each translation class.

Apply strong choice to the definable relation
$p\sim p'$ if $H_p(u)=H_{p'}(u)$ for all $u$. Let $P_*\subset P_0$
be the resulting definable set containing exactly one parameter for
each distinct function. The map
\[
 \chi:P_*\to\mathcal C,\qquad \chi(p)=H_p,
\]
is a continuous bijection. Its inverse is Borel by the Lusin--Souslin
theorem \cite[Corollary~15.2]{kechris}. On $P_*$ the orbit relation is
\[
 p\simeq p'\quad\Longleftrightarrow\quad
 \exists z\ \forall u\quad
 H_{p'}(u)=H_p(u+z)-H_p(z).
\]
A second application of strong choice gives a definable representative
map $b:P_*\to Q\subset P_*$ with
\[
b(p)=b(p')\quad\Longleftrightarrow\quad p\simeq p'.
\]
Here $Q=b(P_*)$ contains one parameter for each translation orbit.
Thus $\rho=b\circ\chi^{-1}$ is Borel and
\[
 \rho(H)=\rho(J)\quad\Longleftrightarrow\quad
 J=T_zH\ \text{for some }z\in\R^d.
\]
The representative of $q\in Q$ is $\chi(q)$. The definable Euclidean
set $Q$ is standard Borel, \mbox{completing the construction.}
\end{proof}

\subsubsection{The geometry of a translation orbit}
The translations satisfying $T_zH=H$ form a subspace $S$. Splitting $\R^d=S\oplus S^\perp$ separates a linear part of $H$ from the normal function that determines its translation orbit.
\mbox{For example, in two dimensions,}
\[
 H(u_1,u_2)=c u_1+k(u_2)
 \quad\Longrightarrow\quad
 (T_{(a,b)}H)(u_1,u_2)=c u_1+k(u_2+b)-k(b).
\]
The position $a$ does not enter the new difference quotient function. The construction below
also includes any additional translation symmetries of $k$.

\begin{proof}[Proof of Lemma~\ref{g:lem:split}]
Continuity and the translation group law give
\[
 z_i\in S,\ z_i\to z\ \Longrightarrow\ T_zH=H,
 \qquad T_{z+w}H=T_zT_wH,
 \qquad T_{-z}T_zH=H.
\]
Thus $S$ is a closed additive subgroup. It is definable by
Lemma~\ref{g:fact:codes}. For $z\in S$, the definable subgroup
$\{t\in\R:tz\in S\}$ contains $\mathbb Z$, hence contains an interval
$I$ of positive length. Consequently
\[
 \R=\bigcup_{m\ge1}m(I-I)\subset\{t:tz\in S\},
\]
so $S$ is a linear subspace.

The identity $T_pH=H$ gives
\[
 H(u+p)=H(u)+H(p),\qquad p\in S,
 \qquad H(tp)=tH(p),\qquad t\in\mathbb Q.
\]
For $t_i\in\mathbb Q$ with $t_i\to t\in\R$, continuity yields
$H(tp)=\lim_iH(t_ip)=tH(p)$. In an orthonormal basis
$e_1,\ldots,e_{\dim S}$, set $c=\sum_iH(e_i)e_i$. Then
\[
 H(p)=\sum_i\ip{p}{e_i}H(e_i)=\ip{c}{p},
 \qquad H(p+q)=\ip{c}{p}+H(q).
\]
Taking $K=H|_N$ proves \eqref{g:eq:split}; restriction preserves
definability and the Lipschitz bound.

Let $v\in\pf H(p+q)$. For $0\ne b\in S$, increments $tb$ give
\[
 0\le\liminf_{t\to0}
 \frac{t\ip{c-P_Sv}{b}}{|t|\norm{b}}
 =-\frac{|\ip{c-P_Sv}{b}|}{\norm{b}}.
\]
Hence $P_Sv=c$. With increments in $N$,
\[
 \liminf_{b\to0,\ b\in N\setminus\{0\}}
 \frac{K(q+b)-K(q)-\ip{P_Nv}{b}}{\norm{b}}\ge0,
\]
which gives $P_Nv\in\pf^NK(q)$ and the first inclusion in
\eqref{g:eq:subsplit}. Conversely, let $w\in\pf^NK(q)$. For every
$\epsilon>0$ and sufficiently small $b\in N$,
$K(q+b)-K(q)-\ip{w}{b}\ge-\epsilon\norm{b}$, also when $b=0$.
For a sufficiently small orthogonal increment $a+b\in S+N$,
\begin{align*}
 &H(p+q+a+b)-H(p+q)-\ip{c+w}{a+b}\\
 &\qquad=K(q+b)-K(q)-\ip{w}{b}
 \ge-\epsilon\norm{b}
 \ge-\epsilon\norm{a+b}.
\end{align*}
This proves $c+w\in\pf H(p+q)$ and the reverse inclusion.

Finally,
\[
 T_{p+q}H=T_qH,\qquad
 T_{q_1}H=T_{q_2}H
 \ \Longleftrightarrow\ q_1-q_2\in S\cap N=\{0\}.
\]
Thus $z\mapsto T_zH$ is a continuous bijection from $N$ onto its
orbit. Its inverse is Borel by the Lusin--Souslin consequence
\cite[Corollary~15.2]{kechris}, applied to this injection into the
\mbox{compact metric space $\mathcal C$.}
\end{proof}

Thus \eqref{g:eq:split} is an exact decomposition, not a Taylor
linearization. Movement along $S$ contributes a linear change in value
but does not need to be recorded as a new difference quotient function position. This
observation turns the difference quotient function kernel into an ordinary chain on $N$.

\begin{lemma}[The stationary chain in normal coordinates]\label{g:lem:normal}
Let a stationary probability for $P$ assign probability $1$ to pairs
whose difference quotient function is a translation of $H\in\mathcal C$.
Under the identification in
Lemma~\ref{g:lem:split}, its normal position has a stationary law for
\begin{equation}\label{g:eq:normalchain}
 Z^+=Z-\nabla_NK(Z)+P_N\eta.
\end{equation}
Its full displacement and its difference quotient function decrease are
\begin{align}
 \Delta&=-c-\nabla_NK(Z)+\eta,\label{g:eq:fulldisplacement}\\
 (T_ZH)(0)-(T_ZH)(\Delta)
 &=\norm{c}^2-\ip{c}{P_S\eta}+K(Z)-K(Z^+).
 \label{g:eq:orbitincrement}
\end{align}
If $N$ has positive dimension, the stationary normal law is absolutely
continuous on $N$.
\end{lemma}
\begin{proof}
Let $\pi$ be the distribution on $N$ obtained by taking the
normal position of a state with the given stationary distribution.
Realize the position and the next noise on
\[
 \bigl(N\times\R^d,\,
       \mathcal B(N)\otimes\mathcal B(\R^d),\,
       \pi\otimes\Gamma_\Sigma\bigr),
 \qquad Z(z,e)=z,\quad \eta(z,e)=e.
\]
Here $\Gamma_\Sigma=N(0,\Sigma)$, and $H$, $N$, $K$, and $c$ are
fixed. The product measure makes $Z$ and $\eta$ independent.
The stationary-gradient identity \eqref{g:eq:stationarygradient}
identifies the mark at the state $T_ZH$
almost surely. Differentiating \eqref{g:eq:split} gives
\[
 v=\nabla(T_ZH)(0)=\nabla H(Z)=c+\nabla_NK(Z),
 \qquad \Delta=-c-\nabla_NK(Z)+\eta.
\]
Differentiability of $H$ at $Z$ and of $K$ on $N$ at $Z$ are
equivalent by the same decomposition. For the output difference quotient function,
\begin{align*}
 H^+&=T_\Delta T_ZH=T_{Z+\Delta}H=T_{P_N(Z+\Delta)}H,\\
 Z^+&=P_N(Z+\Delta)
      =Z-\nabla_NK(Z)+P_N\eta.
\end{align*}
The Borel inverse in Lemma~\ref{g:lem:split} therefore takes a stationary
difference quotient function law to a stationary law for \eqref{g:eq:normalchain}. Also,
\begin{align*}
 (T_ZH)(0)-(T_ZH)(\Delta)
 &=H(Z)-H(Z+\Delta)\\
 &=-\ip{c}{P_S\Delta}+K(Z)-K(Z+P_N\Delta)\\
 &=\norm{c}^2-\ip{c}{P_S\eta}+K(Z)-K(Z^+),
\end{align*}
which proves \eqref{g:eq:orbitincrement}.

If $\dim N>0$, then $u^{\mathsf T}\Sigma u>0$ for every
$u\in N\setminus\{0\}$. Write $\varphi_N$ for the Gaussian density
of $P_N\eta$ on $N$. For every Lebesgue-null set $A\subset N$,
\[
 \begin{aligned}
 \pi(A)&=(\pi\otimes\Gamma_\Sigma)(Z^+\in A)\\
 &=\int_N\left(\int_A
       \varphi_N(z-y+\nabla_NK(y))\,dz\right)\pi(dy)=0.
 \end{aligned}
\]
This proves absolute continuity. If $N=\{0\}$, the normal chain has
one state and all formulas hold with $\nabla_NK=0$.
\end{proof}

Only the marginal distribution of $P_N\eta$ enters the normal chain;
the two orthogonal noise projections may be correlated.

\subsubsection{Conditioning a stationary distribution on its orbit}
A stationary distribution may involve several translation orbits, while the normal position process is defined within one orbit. Because recentering preserves each orbit, disintegration yields stationary orbit components to which the fixed-step escape lemma can be applied.

\begin{proposition}[Stationary analysis of a fixed shape]\label{g:prop:mixture}
Let $\mu$ be stationary for $P$ and satisfy
$\mu(\mathcal C\times\overline B(0,L))=1$. There exist a standard
Borel space $\Theta$, a probability $\nu$ on $\Theta$, a Borel map
$\theta\mapsto H_\theta\in\mathcal C$, and a measurable family of
probabilities $\mu_\theta$ on $E$ such that, for every $\theta\in\Theta$,
$H_\theta$ is globally definable and
\[
 \mu_\theta P=\mu_\theta,\qquad
 \mu_\theta\{(J,v):J=T_zH_\theta
                    \text{ for some }z\in\R^d\}=1.
\]
For every bounded Borel function $\psi:E\to\R$,
\[
 \int_E\psi\,d\mu
 =\int_\Theta\left(\int_E\psi\,d\mu_\theta\right)d\nu(\theta).
\]
In particular, the mean descent can be computed by this identity
with $\psi=d$.
\end{proposition}

Recentering preserves a translation class. Conditioning on that
class will therefore preserve stationarity as well.

\begin{proof}
Use Lemma~\ref{g:lem:orbits} and write $\rho(H,v)=\rho(H)$ on
$\mathcal C\times\overline B(0,L)$. All fibers $\rho^{-1}(q)$ below
are taken in this domain and viewed as Borel subsets of $E$.
The standard Borel conditional
distribution theorem \cite[Theorems~4.1.17--4.1.18]{durrett} gives
\[
 \bar\mu=\rho_\#\mu,\qquad
 \mu=\int_Q\mu_q\,d\bar\mu(q),\qquad
 \mu_q(\rho^{-1}(q))=1
 \quad\text{for }\bar\mu\text{-almost every }q.
\]
Here $\rho_\#\mu$ is the pushforward probability,
$(\rho_\#\mu)(A)=\mu(\rho^{-1}(A))$, and $(\mu_q)$ is a measurable
family of conditional probabilities on $E$. The
measure identity holds on every Borel set and against every
bounded Borel test.

For bounded Borel $b$ on $Q$, extend $b\circ\rho$ by $0$ outside
$\mathcal C\times\overline B(0,L)$. The kernel preserves each orbit,
so on this set, for $\psi\in C(E)$,
\[
 P[(b\circ\rho)\psi]=(b\circ\rho)P\psi.
\]
Stationarity of $\mu$ gives
\begin{align*}
 &\int_Q b(q)\!\left[\int_E(P\psi-\psi)\,d\mu_q\right]d\bar\mu(q)\\
 &\qquad=\int_E\bigl(P[(b\circ\rho)\psi]-(b\circ\rho)\psi\bigr)\,d\mu=0.
\end{align*}
Since $b$ is arbitrary, the inner integral vanishes almost everywhere.
Let $\mathcal A\subset C(E)$ be countable and uniformly dense, and set
\[
 \Theta=\{q:\mu_q(\rho^{-1}(q))=1\}
 \cap\bigcap_{\psi\in\mathcal A}
 \left\{q:\int_E(P\psi-\psi)\,d\mu_q=0\right\}.
\]
The set $\{(q,H,v):\rho(H)=q\}$ is Borel. Integration against the
measurable family $\mu_q$ therefore makes the first condition Borel;
the other conditions are Borel as well. Thus $\Theta\subset Q$ is
standard Borel and $\bar\mu(\Theta)=1$.

For $\theta\in\Theta$ and $\psi\in C(E)$, choose
$\psi_i\in\mathcal A$ with $\norm{\psi_i-\psi}_\infty\to0$. Since
$P$ is a contraction \mbox{in the supremum norm,}
\[
 \left|\int_E(P\psi-\psi)\,d\mu_\theta\right|
 \le2\norm{\psi_i-\psi}_\infty
       +\left|\int_E(P\psi_i-\psi_i)\,d\mu_\theta\right|
 \longrightarrow0.
\]
Continuous tests determine probability measures on the compact space
$E$, so $\mu_\theta P=\mu_\theta$ for every $\theta\in\Theta$.

Restrict $\bar\mu$ to $\Theta$ and call the resulting probability
$\nu$. Let $H_\theta$ be the Borel representative supplied by
Lemma~\ref{g:lem:orbits}. Each $H_\theta$ is globally definable, and
\[
 \rho^{-1}(\theta)
 =\{(J,v)\in E:J=T_zH_\theta\text{ for some }z\in\R^d\}.
\]
The defining property of $\Theta$ gives probability $1$ to this orbit.
Finally, for every bounded Borel $\psi$,
\[
 \int_E\psi\,d\mu
 =\int_Q\left(\int_E\psi\,d\mu_q\right)d\bar\mu(q)
 =\int_\Theta\left(\int_E\psi\,d\mu_\theta\right)d\nu(\theta).
\]
The function $d$ is bounded and Borel by Proposition~\ref{g:prop:kernel},
so it is an admissible test in this identity.
\end{proof}

\subsection{Fixed-step escape}\label{g:sec:fixed-tool}
We prove Lemma~\ref{g:lem:fixed}: a global positive lower bound on Fr\'echet subgradient norms forces almost-sure escape for Gaussian subgradient iterations with step size $1$. The bound gives linear decrease on spheres; a homogeneous limit at infinity identifies a region supporting outward motion, and the Gaussian density makes this region accessible from every bounded set. The resulting exclusion of stationary position distributions is the escape ingredient in the positive-descent proof of Section~\ref{g:sec:energy}.

\label{g:sec:fixed}
The hypothesis implies $L\ge a>0$, by evaluating a gradient at a
differentiability point.  It allows empty Fr\'echet subdifferentials and points at
which $0\in\pc H$.  The proof therefore cannot replace
\eqref{g:fixed:gap} by a lower bound on the particular Clarke selection
$V_n$.  Instead, it first extracts a direction of linear decrease from
the Fr\'echet gap.

\begin{lemma}[A lower value on every sphere]\label{g:lem:sphere}
Under \eqref{g:fixed:gap}, every $x\in\R^d$ and $r>0$ satisfy
\begin{equation}\label{g:fixed:sphere}
 \min_{\norm{z-x}=r}H(z)\le H(x)-\frac a2 r.
\end{equation}
\end{lemma}
\begin{proof}
For $\varepsilon>0$, minimize
\[
 H(z)+\frac a2\sqrt{\norm{z-x}^2+\varepsilon^2}
\]
over the closed ball $\overline B(x,r)$.  By the generalized Fermat rule \cite[Theorem~10.1]{RW}, an interior
minimizer would give
\[
 -\frac a2\frac{z-x}{\sqrt{\norm{z-x}^2+\varepsilon^2}}
 \in\pf H(z).
\]
The displayed vector has norm less than $a/2$, contrary to
\eqref{g:fixed:gap}.  A minimizer $z_\varepsilon$ thus lies on the
boundary.  Comparison with $x$ gives
\[
 H(z_\varepsilon)\le H(x)+\frac a2\varepsilon
             -\frac a2\sqrt{r^2+\varepsilon^2}.
\]
Take a convergent subsequence on the compact sphere and let
$\varepsilon$ decrease to $0$.
\end{proof}

We next look at $H$ from distances tending to infinity.  The resulting
homogeneous potential describes radial movement, while an outer limit
of the original Clarke fields retains the possible drift directions.
Lemma~\ref{g:lem:recession} makes this construction precise.
Retaining this field is necessary because convergence of potentials
does not identify their Clarke subdifferentials pointwise.

\begin{lemma}[The homogeneous potential and its conservative field]\label{g:lem:recession}
There is a definable, globally $L$-Lipschitz, positively homogeneous
function $J$, where positive homogeneity means $J(tu)=tJ(u)$ for
$t\ge0$, such that
\begin{equation}\label{g:fixed:recession}
 \frac{H(Ru)-H(0)}{R}\longrightarrow J(u)
 \quad\text{locally uniformly as }R\longrightarrow\infty,
 \qquad \min_{\norm{u}=1}J(u)\le-\frac a2.
\end{equation}
Define
\begin{equation}\label{g:fixed:field}
 \begin{split}
 A(u)=\{v:\ &\text{there exist }R_j\to\infty,\ u_j\to u,
                  \ v_j\to v\\[-1mm]
             &\text{with }v_j\in\pc H(R_ju_j)\},
 \qquad \mathcal V(u)=\conv A(u).
 \end{split}
\end{equation}
Then $\mathcal V$ is definable, has closed graph and nonempty compact convex
values, and is bounded by $L$.  It is conservative for $J$ and satisfies
\begin{equation}\label{g:fixed:radial}
 \mathcal V(tu)=\mathcal V(u),\qquad \ip{v}{u}=J(u)
 \quad\text{for }t>0\text{ and }v\in \mathcal V(u).
\end{equation}
\end{lemma}
\begin{proof}
Set $f_s(u)=s[H(u/s)-H(0)]$ for $s>0$.  For every $u,w$,
\[
 |f_s(u)|\le L\norm{u},\qquad
 |f_s(u)-f_s(w)|\le L\norm{u-w}.
\]
For each fixed $u$, the one-variable definable function $s\mapsto f_s(u)$ has a finite partition into intervals on each of which it is constant or strictly monotone \cite[Section~4.1]{vddm}. It is therefore constant or monotone on some interval $(0,\varepsilon_u)$. The bound $|f_s(u)|\le L\norm{u}$ gives a finite limit $J(u)=\lim_{s\downarrow0}f_s(u)$.  The second inequality passes to
the limit, so $J$ is $L$-Lipschitz.  If $u_1,\ldots,u_m$ form an
$\varepsilon$-net in a compact set $K$, then
\[
 \sup_{u\in K}|f_s(u)-J(u)|
 \le2L\varepsilon+\max_{1\le i\le m}|f_s(u_i)-J(u_i)|.
\]
First let $s\downarrow0$, then $\varepsilon\downarrow0$.  This
proves local uniform convergence.  Definability follows from
\[
 \begin{split}
 J(u)=z\quad\Longleftrightarrow\quad
 &\forall\varepsilon>0\ \exists s_0>0\ \forall s\in(0,s_0),\\[-1mm]
 &|f_s(u)-z|<\varepsilon.
 \end{split}
\]
For $t>0$, homogeneity and the sphere bound of
Lemma~\ref{g:lem:sphere} give
\begin{align*}
 J(tu)&=\lim_{s\downarrow0}f_s(tu)
       =t\lim_{s\downarrow0}f_{s/t}(u)=tJ(u),\\
 \min_{\norm{u}=1}f_s(u)
 &=s\!\left[\min_{\norm{z}=1/s}H(z)-H(0)\right]\le-\frac a2,\\
 \min_{\norm{u}=1}J(u)
 &=\lim_{s\downarrow0}\min_{\norm{u}=1}f_s(u)\le-\frac a2.
\end{align*}
The last limit uses uniform convergence on the sphere.

The Clarke values are nonempty and bounded by $L$.  Choosing
$v_j\in\pc H(ju)$ and taking a subsequence proves $A(u)\ne\varnothing$.
The graph identity
\[
 \operatorname{graph}A
 =\{(u,v):(u,0,v)\in
       \overline{\{(w,s,z):0<s\le1,\ z\in\pc H(w/s)\}}\}
\]
gives definability and closedness.  Boundedness then gives compact
values.  To check convexification, take $u_j\to u$ and $w_j\to w$
with $w_j\in \mathcal V(u_j)$. Carath\'eodory's theorem
\cite[Theorem~2.29]{RW} gives
\[
 \begin{gathered}
 w_j=\sum_{i=0}^{d}\theta_i^{(j)}v_i^{(j)},\qquad
 v_i^{(j)}\in A(u_j),\\
 \theta_i^{(j)}\ge0,\qquad\sum_{i=0}^{d}\theta_i^{(j)}=1.
 \end{gathered}
\]
After a subsequence, $\theta_i^{(j)}\to\theta_i$ and
$v_i^{(j)}\to v_i\in A(u)$ for every $i$.  Therefore
$w=\sum_i\theta_iv_i\in \mathcal V(u)$.  This proves closedness of the
graph of $\mathcal V$ and, with the same bound, compactness of its values.

Apply Lemma~\ref{g:fact:conservative} to $f_s$, $0<s\le1$, with
$f_0=J$.  This is a single jointly definable \mbox{parameter family, and}
\[
 \varnothing\ne\pc f_s(u)=\pc H(u/s)\subset\overline B(0,L).
\]
Its positive-parameter Clarke fields admit the required variational
stratifications.  The joint potential limit follows from
\[
 |f_s(u')-J(u)|\le L\norm{u'-u}+|f_s(u)-J(u)|\longrightarrow0
 \quad\text{as }(u',s)\to(u,0).
\]
The horizon term is $0$ because
\[
 \varepsilon_j\to0,\quad v_j\in\pc f_{s_j}(u_j)
 \quad\Longrightarrow\quad
 \norm{\varepsilon_jv_j}\le L|\varepsilon_j|\longrightarrow0.
\]
Schechtman's Theorem~4.3 and Corollary~4.5 thus give conservativity
of $A$ and $\mathcal V=\conv A$ for $J$~\cite{schechtman}.

For $t>0$, a sequence defining $v\in A(tu)$ also satisfies
\[
 v_j\in\pc H\bigl((tR_j)(u_j/t)\bigr),\qquad
 tR_j\to\infty,\quad u_j/t\to u.
\]
Hence $A(tu)\subset A(u)$.  The same argument with $1/t$ gives
$A(tu)=A(u)$ and consequently $\mathcal V(tu)=\mathcal V(u)$.  Fix $v\in \mathcal V(u)$.
Along the ray $t\mapsto tu$, conservativity and homogeneity give
\[
 J(u)=\frac{d}{dt}[tJ(u)]
     =\frac{d}{dt}J(tu)=\ip{v}{u}
 \quad\text{for almost every }t\in(1,2).
\]
The first and last expressions do not depend on $t$, proving
\eqref{g:fixed:radial} for this arbitrary $v$.
\end{proof}

For $\epsilon>0$, write
\[
 \mathcal V_{\epsilon}(u)=\{w:\ \norm{u-z}\le \epsilon,
                \ \norm{w-v}\le \epsilon\text{ for some }v\in \mathcal V(z)\}.
\]
On any compact annulus, \eqref{g:fixed:field} implies the uniform
inclusion
\begin{equation}\label{g:fixed:inclusion}
 \pc H(Ru)\subset \mathcal V_{\epsilon}(u)
 \quad\text{for all sufficiently large }R.
\end{equation}
Indeed, a sequence violating the inclusion has convergent state and
field subsequences, whose limit belongs to the graph of $A$.  This is
the approximation of the original drift used below.

Use the potential and field of Lemma~\ref{g:lem:recession}.
Write $\mathbb S^{d-1}=\{u\in\R^d:\norm{u}=1\}$ for the unit
sphere. For $u\in\mathbb S^{d-1}$, use $J(u)$ as the angular value.
Along the limiting differential
inclusion $\dot x\in-\mathcal V(x)$, a negative angular value forces outward
radial movement.  The next lemma also gives a strict angular margin
at the next radius.  That margin makes the conclusion stable under
discretization and small cumulative noise.

The polygon recursion below comes directly from the fixed-step chain.
At the start of a phase, let $r=\norm{Y_n}$ and use the space and
time scales
\[
 x_k=Y_{n+k}/r,\qquad t_k=k/r,\qquad v_k=V_{n+k},\qquad
 e_k=\frac1r\sum_{i=1}^k\eta_{n+i},\quad e_0=0.
\]
Dividing \eqref{g:fixed:chain} by $r$ gives exactly
\[
 x_{k+1}-x_k=-\frac{v_k}{r}+\frac{\eta_{n+k+1}}r
 =-(t_{k+1}-t_k)v_k+e_{k+1}-e_k.
\]
Thus $e_k$ is the cumulative rescaled noise; its difference is the
single-step perturbation. For large $r$, \eqref{g:fixed:inclusion}
places the actual drift in $\mathcal V_\epsilon(x_k)$ on the relevant
annulus. The next lemma allows a general small mesh and controls the
whole cumulative error, which is the quantity bounded by
Lemma~\ref{g:lem:doob} in the stochastic argument.

\begin{lemma}[A robust outward phase]\label{g:lem:outward}
There are $b<0$, $\delta>0$, $T>0$, and $\epsilon\in(0,1)$ with the
following property.  Consider an Euler polygon with vertices $x_k$
at times $0=t_0<t_1<\cdots$, with mesh at most $\epsilon$, and increments
\begin{equation}\label{g:fixed:polygon}
 x_{k+1}=x_k-(t_{k+1}-t_k)v_k+e_{k+1}-e_k,
 \qquad v_k\in \mathcal V_{\epsilon}(x_k).
\end{equation}
Assume $e_0=0$, $\sup_k\norm{e_k}\le \epsilon$, $\norm{x_0}=1$, and
$J(x_0)\le b$.  The mesh is defined through the first time at least
$T$.  Before a mesh time exceeding $T$ by at most $\epsilon$, the polygon
reaches a first vertex with radius at least two.  At every preceding
vertex its radius is greater than $1/2$ and its angular value is less than
$b+\delta/2$.  At that vertex,
\begin{equation}\label{g:fixed:block}
 2\le\norm{x_k}\le3,
 \qquad J(x_k/\norm{x_k})\le b-\delta/2.
\end{equation}
It is enough to impose the field inclusion in
\eqref{g:fixed:polygon} before this vertex or the first violation of
\mbox{either preceding bound.}
\end{lemma}
\begin{proof}
Let $x$ satisfy $\dot x=-v$, $v\in \mathcal V(x)$, on an interval where
$r=\norm{x}>0$, and put $u=x/r$.  Since $\mathcal V(x)=\mathcal V(u)$, the
conservative chain rule and \eqref{g:fixed:radial} give, almost
everywhere,
\begin{align*}
 \frac{d}{dt}J(x)&=\ip{v}{-v}=-\norm{v}^2,\\
 \dot r&=\ip{u}{-v}=-J(u),\\
 r\frac{d}{dt}J(u)
 &=\frac{d}{dt}J(x)-\dot rJ(u)
   =-\norm{v}^2+J(u)^2.
\end{align*}
Because $\ip{v}{u}=J(u)$ and $\norm{u}=1$, this is
\begin{equation}\label{g:fixed:flow}
 \dot r=-J(u),\qquad
 \frac{d}{dt}J(u)=-\frac{\norm{v-J(u)u}^2}{r}\le0.
\end{equation}

The set
\[
 \{u\in\mathbb S^{d-1}:J(u)u\in \mathcal V(u)\}
\]
is definable and has finitely many definably path connected
components \cite[Chapter~6, \S3]{vdd}. Along a piecewise $C^1$
path $u(t)$ in one component,
\[
 \frac{d}{dt}J(u(t))
 =\ip{J(u(t))u(t)}{\dot u(t)}
 =\frac{J(u(t))}{2}\frac{d}{dt}\norm{u(t)}^2=0
 \quad\text{almost everywhere}.
\]
Thus its image under $J$ is finite.  Choose
$b$ strictly between $\min_{\norm{u}=1}J(u)$ and $0$ and outside this finite set.
Compactness and the closed graph of $\mathcal V$ give $\delta_0>0$ and
$\mu>0$ such that $[b-2\delta_0,b+2\delta_0]\subset(-\infty,0)$
and
\[
 \norm{v-J(u)u}\ge\mu
 \quad\text{if }|J(u)-b|\le2\delta_0,\quad v\in \mathcal V(u).
\]
If this band has no points on the sphere, the assertion is vacuous
and any $\mu>0$ can be used.  Set
\[
 0<\delta\le\min\{\delta_0,\mu^2/(8L)\},
 \qquad T=2/(-b).
\]

For an exact solution with $r(0)=1$ and $J(u(0))\le b$,
\[
 J(u(t))\le b<0,
 \qquad -b\le\dot r(t)\le L,
 \qquad 1-bt\le r(t)\le1+Lt.
\]
Its first radius-two hit $\tau$ therefore satisfies
$1/L\le\tau\le1/(-b)<T$.  Once $J$ reaches $b-\delta$, it stays
below that value.  Otherwise $b-\delta<J(u(t))\le b$ for
$0\le t\le\tau$, and the band estimate gives
\[
 J(u(\tau))-J(u(0))
 \le-\int_0^\tau\frac{\mu^2}{r(t)}\,dt
 \le-\frac{\mu^2\tau}{2}
 \le-\frac{\mu^2}{2L}\le-4\delta.
\]
This contradicts $J(u(\tau))>b-\delta$.  Thus every exact solution
reaches radius two with $J\le b-\delta$.

Suppose no uniform tolerance exists. Take $\epsilon_j\downarrow0$ and
violating polygons. Stop them at the first hit, guard failure or
timeout, and extract $\sigma_j\to\sigma\le T$. Interpolate the
cumulative errors linearly and freeze both paths after stopping.
Ascoli gives a subsequence for which
\[
 z_j=x_j-e_j,\qquad \operatorname{Lip}(z_j)\le L+1,
 \qquad x_j,z_j\longrightarrow x\quad\text{uniformly}.
\]
The strict initial guard margins and uniform convergence imply
$\sigma>0$. For the preceding vertex $\bar x_j$ and every fixed
$t<\sigma$, for all large $j$,
\[
 \norm{z_j-\bar x_j}_{[0,t]}\le(L+4)\epsilon_j,
 \qquad \dot z_j\in-\mathcal V_{(L+5)\epsilon_j}(z_j).
\]
Faure--Roth~\cite[Corollary~4.10]{faure} applies to $-\mathcal V$:
it is bounded, with closed graph and nonempty compact convex values.
Enlarge the tolerance to $(L+6)\epsilon_j$ for their open graph neighborhoods.
The corollary gives $\dot x\in-\mathcal V(x)$ before $\sigma$.
The guards keep the limit away from $0$ there, so the exact-solution
bounds apply.

A violating subsequence and uniform convergence would give one of
\[
 \begin{gathered}
 r(\sigma)\le1/2,\qquad J(u(\sigma))\ge b+\delta/2,\\
 \sigma=T,\ r(T)\le2,\qquad
 r(\sigma)=2,\ J(u(\sigma))\ge b-\delta/2.
 \end{gathered}
\]
These contradict $r\ge1$, $J(u)\le b$, $r(T)\ge1-bT=3$, and the
strict angular estimate at the first radius-two hit, respectively.
The overshoot bound $(L+3)\epsilon_j<1$ also excludes \mbox{a terminal radius above three.}
\end{proof}

The deterministic part is now complete.  At physical radius $r$,
spatial rescaling by $r$ turns one step of
\eqref{g:fixed:chain} into an Euler step of duration $1/r$.
The cumulative centered noise over order $r$ steps has size
$o(r)$ with high probability.  We use this observation first to
continue outward phases, and then to rule out recurrent visits to
a bounded set.

\begin{proof}[Proof of Lemma~\ref{g:lem:fixed}]
Use the probability space $(\Omega,\mathcal F,\PP)$ and filtration
$(\mathcal F_n)$ carrying the auxiliary chain $Y_n$, its adapted
selections, and its noise sequence. All path expectations and
conditional expectations in this proof use $\PP$.
Let $M_2=\E\norm{\eta_1}^2=\tr\Sigma$. At a finite integer
stopping time $n$, future independence follows by conditioning on
$\{n=\ell\}$, $\ell=0,1,\ldots$.  In particular, for every bounded
Borel test $\phi$ and every fixed $k$,
\[
 \begin{aligned}
 &\E[\phi(\eta_{n+1},\ldots,\eta_{n+k})\mid\mathcal F_n]\\
 &\qquad=\int_{(\R^d)^k}\phi(z_1,\ldots,z_k)\,
                   \Gamma_\Sigma^{\otimes k}(d(z_1,\ldots,z_k)).
 \end{aligned}
\]
The right side is expectation on the $k$-noise product space and
is constant as a function of the observed history.
If $r=\norm{Y_n}\ge1$, put $m=\lceil Tr\rceil$.  Both are
$\mathcal F_n$-measurable.  \mbox{Centering and independence give}
\[
 \E\!\left[\left\|\sum_{i=1}^{m}\eta_{n+i}\right\|^2
                \middle|\mathcal F_n\right]=mM_2.
\]
The conditional maximal inequality in Lemma~\ref{g:lem:doob} therefore yields
\begin{equation}\label{g:fixed:noise}
 \begin{split}
 &\PP\left(
    \max_{0\le k\le m}\left\|r^{-1}\sum_{i=1}^k\eta_{n+i}\right\|>\epsilon
          \,\middle|\,\mathcal F_n\right)\\
 &\hspace{15mm}\le\frac{4mM_2}{\epsilon^2r^2}
 \le\frac{4M_2(Tr+1)}{\epsilon^2r^2}\le\frac Cr,
 \qquad C=\frac{4M_2(T+1)}{\epsilon^2}.
 \end{split}
\end{equation}
The inequality is applied conditionally for the known horizon $m$
and radius $r$.

Choose $R_0$ large enough that $1/R_0\le \epsilon$,
\eqref{g:fixed:inclusion} holds on $1/4\le\norm{u}\le4$ for every
scale $R\ge R_0$, and $2C/R_0\le1/2$.  Start a phase at any state
$y$ with radius $r\ge R_0$ and $J(y/r)\le b$.  The rescaled
vertices $x_k=Y_{n+k}/r$ obey \eqref{g:fixed:polygon}, with mesh
$1/r$ and cumulative perturbation
$e_k=r^{-1}\sum_{i=1}^k\eta_{n+i}$.  Until the stopping conditions
of Lemma~\ref{g:lem:outward}, the drift inclusion follows from
\eqref{g:fixed:inclusion}.

Let $B$ be the event that the actual path reaches its first vertex of
radius at least $2r$ within $m$ steps, before a guard violation, and that
this endpoint satisfies the scaled version of
\eqref{g:fixed:block}.  The end of the attempt is an integer stopping
time, and $B$ is measurable at that time.  Lemma~\ref{g:lem:outward}
and \eqref{g:fixed:noise} give
\[
 \begin{gathered}
 \left\{\max_{k\le m}\norm{e_k}\le \epsilon\right\}\subset B,\\
 \PP(B^c\mid\mathcal F_n)
 \le\PP\left(\max_{k\le m}\norm{e_k}>\epsilon\mid\mathcal F_n\right)
 \le\frac Cr.
 \end{gathered}
\]
Restart at the actual successful endpoint.  Its radius is in
$[2r,3r]$ and its angular value is at most $b$.  Conditioning uses
the actual endpoint sigma-field, rather than the sufficient
small-noise event over the full future horizon.

Write $\tau_j$ for the start of phase $j$, with $\tau_0=n$, and $B_j$
for its success event.  Put $G_0=\Omega$ and
$G_j=\bigcap_{i<j}B_i$.  After a failed phase, keep all subsequent
start times fixed at the failure time and declare their success
events empty.  Thus the $\tau_j$ are finite stopping times and
$G_j\in\mathcal F_{\tau_j}$.  Set $r_j=\norm{Y_{\tau_j}}$ on
$G_j$ and $r_j=1$ otherwise.  On $G_j$,
\[
 r_j\ge2^jr,\qquad
 \PP(B_j^c\mid\mathcal F_{\tau_j})\le\frac C{r_j}.
\]
The conditional probability of first failure in phase $j$ is bounded by
\begin{align*}
 \PP(G_j\cap B_j^c\mid\mathcal F_n)
 &=\E\!\left[
    \mathbf1_{G_j}\PP(B_j^c\mid\mathcal F_{\tau_j})
                 \middle|\mathcal F_n\right]\\
 &\le\E\!\left[\mathbf1_{G_j}\frac C{r_j}
                 \middle|\mathcal F_n\right]
 \le\frac C{2^jr}.
\end{align*}
The first-failure events are disjoint, so
\begin{equation}\label{g:fixed:phases}
 \begin{split}
 \PP\left(\bigcup_{j\ge0}(G_j\cap B_j^c)
                   \middle|\mathcal F_n\right)
 &\le\sum_{j=0}^{\infty}\frac C{2^jr}
 =\frac{2C}{r}\le\frac12,\\
 \PP\left(\bigcap_{j\ge0}B_j\middle|\mathcal F_n\right)
 &\ge1-\frac{2C}{r}\ge\frac12.
 \end{split}
\end{equation}
On $\bigcap_jB_j$, each phase has positive finite duration, and
\[
 \tau_j\ge n+j\longrightarrow\infty,\qquad
 \inf_{\tau_j\le k<\tau_{j+1}}\norm{Y_k}
 \ge\frac{r_j}{2}\ge2^{j-1}r\longrightarrow\infty.
\]
Thus all subsequent radii stay above $r/2$ and tend to infinity.
No independence of successive phases has been assumed.

It remains to reach those states from an arbitrary bounded region.
Fix an integer $M\ge1$.  Choose $u_*$ on the unit sphere with
$J(u_*)<b$.  A sufficiently small open ball $W$ around a sufficiently
large multiple of $u_*$ has
\[
 \norm{z}\ge R_0,\qquad \norm{z}>2M,
 \qquad J(z/\norm{z})<b\quad\text{for every }z\in W.
\]
Let $\varphi_\Sigma$ be the Gaussian density and let $|W|$ denote
the volume of $W$.  At any deterministic time $n$ with
$\norm{Y_n}\le M$,
\[
 \norm{-Y_n+V_n}\le M+L,\qquad
 W-Y_n+V_n\subset\overline W+\overline B(0,M+L).
\]
Continuity and strict positivity of the density on the latter
compact set give
\[
 q_M:=|W|\min_{z\in\overline W+\overline B(0,M+L)}
                  \varphi_\Sigma(z)>0.
\]
Consequently,
\begin{equation}\label{g:fixed:access}
 \begin{split}
 \PP(Y_{n+1}\in W\mid\mathcal F_n)
 &=\int_{W-Y_n+V_n}\varphi_\Sigma(z)\,dz\ge q_M\\
 &\hspace{20mm}\text{whenever }\norm{Y_n}\le M.
 \end{split}
\end{equation}
Let $\mathcal E_n$ be the event that all radii after time $n$ exceed
$M$ and tend to infinity.  On $\{Y_{n+1}\in W\}$, the starting
radius is greater than $2M$, so \eqref{g:fixed:phases} gives
$\PP(\mathcal E_n\mid\mathcal F_{n+1})\ge1/2$.
On $\{\norm{Y_n}\le M\}$, \mbox{the tower property yields}
\[
 \begin{split}
 \PP(\mathcal E_n\mid\mathcal F_n)
 &\ge\E\!\left[
     \mathbf1_{\{Y_{n+1}\in W\}}
     \PP(\mathcal E_n\mid\mathcal F_{n+1})
                      \middle|\mathcal F_n\right]\\
 &\ge\tfrac12\PP(Y_{n+1}\in W\mid\mathcal F_n)
 \ge q_M/2=:p_M.
 \end{split}
\]
Define
\[
 A_M=\{\norm{Y_n}\le M\text{ infinitely often}\},
 \qquad Z_n=\E[\mathbf1_{A_M}\mid\mathcal F_n].
\]
The event $A_M$ belongs to
$\mathcal F_\infty=\sigma(\bigcup_n\mathcal F_n)$.  Lemma~\ref{g:lem:upward} gives
\[
 Z_n\longrightarrow
 \E[\mathbf1_{A_M}\mid\mathcal F_\infty]
 =\mathbf1_{A_M}\quad\text{almost surely}.
\]
Since $A_M\cap\mathcal E_n=\varnothing$, the preceding conditional
bound gives
\[
 \norm{Y_n}\le M
 \quad\Longrightarrow\quad
 Z_n\le1-\PP(\mathcal E_n\mid\mathcal F_n)\le1-p_M.
\]
On $A_M$, this inequality holds infinitely often while $Z_n\to1$,
an impossibility.  Hence $\PP(A_M)=0$ for every integer $M\ge1$,
and
\[
 \PP(\norm{Y_n}\to\infty)
 =\PP\left(\bigcap_{M=1}^{\infty}A_M^c\right)=1.
\]

Finally, if a measurable time-homogeneous selection admitted an
invariant probability $\pi$, start the chain with law $\pi$ and
independent future noise vectors. For this final argument, take
\[
 \begin{aligned}
 \Omega_\pi&=\R^d\times(\R^d)^{\mathbb N},\qquad
 \mathcal F_\pi=\mathcal B(\R^d)\otimes
                         \mathcal B(\R^d)^{\otimes\mathbb N},\\
 \PP_\pi&=\pi\otimes\Gamma_\Sigma^{\otimes\mathbb N},
 \end{aligned}
\]
with $Y_0$ and the noises as coordinate projections. Invariance and bounded convergence
would imply
\[
 \begin{gathered}
 \pi(\overline B(0,M))
 =\PP_\pi(\norm{Y_n}\le M)\longrightarrow0
 \quad\text{for every }M,\\
 1=\lim_{M\to\infty}\pi(\overline B(0,M))=0.
 \end{gathered}
\]
This contradiction proves the last assertion.
\end{proof}

The absence of a stationary law is the conclusion used in
Lemma~\ref{g:prop:negative}: a stationary difference quotient function orbit cannot carry
a globally gapped normal component. The construction also identifies
the two probabilistic ingredients.
Centering and a second moment control accumulated perturbations within
an outward phase.  Gaussian full support supplies the one-step access
in \eqref{g:fixed:access}; a noise law supported in a small ball can
assign probability $0$ to that event.

\subsection{Proof of positive expected descent}\label{g:sec:energy}
On each stationary translation orbit, the inherited subgradient bound and Gaussian escape force a nonzero linear component. Stationarity cancels the normal component's expected increment, leaving mean descent equal to the squared norm of the linear coefficient. Averaging over orbits gives the positive lower bound in Lemma~\ref{g:prop:negative}.
\begin{proof}[Proof of Lemma~\ref{g:prop:negative}]
Proposition~\ref{g:prop:mixture} gives
$\mu=\int_\Theta\mu_\theta\,d\nu(\theta)$, where $\nu$ is a
probability on a parameter space $\Theta$ and each $\mu_\theta$ is
stationary on one translation orbit, the set of all recenterings of
a fixed difference quotient function. Fix $\theta\in\Theta$ and write
$H=H_\theta$. The measure $\mu_\theta$ is stationary and assigns
probability $1$ to translations of $H$, so Lemma~\ref{g:lem:normal}
applies. Explicitly, Lemma~\ref{g:lem:split} gives
\[
 S=\{p:T_pH=H\},\quad N=S^\perp,\qquad
 H(p+z)=\ip{c}{p}+K(z),\quad c\in S.
\]
For a linear subspace $A$, write $P_A$ for its orthogonal projection.
The notation $\nabla_NK$, $\pf^NK$, and $\pc^NK$ refers to the
gradient and subdifferentials computed within $N$.
The stationary orbit law induces a stationary normal position $Z$ for
\[
 Z^+=Z-\nabla_NK(Z)+P_N\eta.
\]
The word normal refers to the orthogonal complement of the translation
symmetries $S$. The construction and its measurability are proved in
Section~\ref{g:sec:orbits}.
Let $\pi_\theta$ denote this stationary distribution on $N$.
Throughout the calculation on this fixed orbit, $\E_\theta$ means
expectation on the probability space
\[
 \bigl(N\times\R^d,\,
       \mathcal B(N)\otimes\mathcal B(\R^d),\,
       \pi_\theta\otimes\Gamma_\Sigma\bigr),
 \qquad Z(z,e)=z,\quad \eta(z,e)=e.
\]
Thus $\theta$, $H$, $S$, $N$, $c$, and $K$ are fixed, while $Z$ and
$\eta$ are independent random variables. The updated $Z^+$ is
the function of these coordinates displayed above. Expectations
$\E_\eta$ involving only the noise use
$(\R^d,\mathcal B(\R^d),\Gamma_\Sigma)$.

We next show that the linear part satisfies $\norm{c}\ge a$.
Lemma~\ref{g:lem:profilegap} and \eqref{g:eq:subsplit} give
\begin{equation}\label{g:eq:normalgap}
 \norm{c}^2+\norm{w}^2\ge a^2
 \qquad(z\in N,\ w\in\pf^NK(z)).
\end{equation}
Suppose $\norm{c}<a$ and $\dim N>0$. Then
\[
 \inf\{\norm{w}:z\in N,\ w\in\pf^NK(z)\}
 \ge\sqrt{a^2-\norm{c}^2}>0.
\]
Let $D\subset N$ be the differentiability set of $K$. Rademacher's
theorem gives that $N\setminus D$ has Lebesgue measure $0$ in $N$,
and the ordinary gradient at every
point of $D$ belongs to the Clarke subdifferential. Definable choice
as used in Lemma~\ref{g:lem:orbits} supplies a Borel selection on the remaining
set, so there is a Borel map $w$ with
\[
 w(z)\in\pc^NK(z)\quad(z\in N),\qquad
 w(z)=\nabla_NK(z)\quad(z\in D).
\]
The stationary normal position in Lemma~\ref{g:lem:normal} satisfies
\[
 \pi_\theta(D)=1,\qquad
 \mathcal L\bigl(Z-w(Z)+P_N\eta\bigr)
 =\mathcal L(Z^+)=\mathcal L(Z).
\]
This is an invariant probability for a chain satisfying
Lemma~\ref{g:lem:fixed}, a contradiction. If $N=\{0\}$, then
$\pf^NK(0)=\{0\}$ and \eqref{g:eq:normalgap} directly gives
\begin{equation}\label{g:eq:linearbound}
 \norm{c}\ge a
\end{equation}
as well. Thus \eqref{g:eq:linearbound} holds on every stationary orbit
under consideration.

It remains to compute the expected descent.
By \eqref{g:eq:orbitincrement}, the decrease on this orbit is
\[
 (T_ZH)(0)-(T_ZH)(\Delta)
 =\norm{c}^2-\ip{c}{P_S\eta}+K(Z)-K(Z^+).
\]
In the stationary normal law,
\[
 K(Z^+)\overset{d}=K(Z),\qquad
 |K(Z)-K(Z^+)|
 \le L\norm{-\nabla_NK(Z)+P_N\eta}
 \le L(L+\norm{\eta})\in L^1.
\]
To take expectations without assuming $\E_\theta|K(Z)|<\infty$, define
$\chi_R(t)=\max\{-R,\min\{t,R\}\}$. Equal distributions and the
$1$-Lipschitz property of $\chi_R$ give
\begin{align*}
 \E_\theta\bigl[\chi_R(K(Z))-\chi_R(K(Z^+))\bigr]&=0,\\
 \bigl|\chi_R(K(Z))-\chi_R(K(Z^+))\bigr|
 &\le |K(Z)-K(Z^+)|.
\end{align*}
Dominated convergence therefore yields
\[
 \E_\theta[K(Z)-K(Z^+)]
 =\lim_{R\to\infty}\E_\theta\bigl[\chi_R(K(Z))-\chi_R(K(Z^+))\bigr]=0.
\]
The three terms in \eqref{g:eq:orbitincrement} now give
\begin{align*}
 \int_E d\,d\mu_\theta
 &=\norm{c}^2-\ip{c}{P_S\E_\eta\eta}
                 +\E_\theta[K(Z)-K(Z^+)]\\
 &=\norm{c}^2\ge a^2.
\end{align*}
Finally, $|d(H,v)|\le L(L+\E_\eta\norm{\eta})$ uniformly on $E$, so
\[
 \int_E d\,d\mu
 =\int_\Theta\left(\int_E d\,d\mu_\theta\right)d\nu(\theta)
 \ge a^2,
\]
which proves \eqref{g:eq:negative}.
\end{proof}

\subsection{Convergence windows and stopped averages}
\label{g:sec:windows}\label{g:sec:windowproofs}

Stopped occupation averages provide vanishing mean descent and approximate stationarity for the same sequence of probability measures. We construct them from finite windows of the original iteration, selecting the windows using past information and stopping at the first exit from a small ball. The two conclusions pass to one weak limit in the proof of Lemma~\ref{g:prop:zero}.

Throughout this section, the original random variables are defined on the
filtered probability space
\((\Omega,\mathcal F,(\mathcal F_n),\PP)\), and
\(\E Z=\int_\Omega Z(\omega)\,\PP(d\omega)\) for integrable \(Z\).
We display conditioning explicitly: \(\mathcal F_n\) records the available
history, while a past event \(B_j\in\mathcal F_{n_j}\) selects the windows
to be averaged.

The observation windows must contain many steps to make their endpoint terms
negligible, while remaining short relative to their starting times so that
their step lengths are nearly constant. Their spatial radius must also be
small relative to their total step length. The following choice meets all
three requirements.

\begin{lemma}\label{g:lem:window}
Assume \(\PP(X_n\to0)>0\). For each integer \(j\ge2\), there are increasing
deterministic times \(n_j\ge j^3\) and events \(B_j\in\mathcal F_{n_j}\)
with \(\PP(B_j)>0\) such that, on \(B_j\),
\begin{equation}\label{g:eq:stay}
 \norm{X_{n_j}}<j^{-2},\qquad
 \PP\bigl(\norm{X_n}<j^{-2}\text{ for all }n\ge n_j
              \mid\mathcal F_{n_j}\bigr)\ge1-\frac1j.
\end{equation}
For the window length and fixed scale
\begin{equation}\label{g:eq:windowparameters}
 m_j=\lceil n_j/j\rceil,\qquad s_j=\gamma_{n_j},
\end{equation}
one has
\begin{equation}\label{g:eq:windowlimits}
 m_j\to\infty,\qquad \frac{m_j}{n_j}\to0,\qquad
 \frac{j^{-2}}{m_js_j}\to0,\qquad
 \max_{n_j\le n\le n_j+m_j}
       \left|\frac{\gamma_n}{s_j}-1\right|\to0.
\end{equation}
\end{lemma}
\begin{proof}
Fix \(j\). The events
\(\bigcap_{n\ge N}\{\norm{X_n}<j^{-2}\}\) increase with \(N\), and their union
contains \(\{X_n\to0\}\). Hence, for some deterministic \(N\), the event
\[
 A=\bigcap_{n\ge N}\{\norm{X_n}<j^{-2}\}
\]
has positive probability. Since
\(A\in\sigma(\bigcup_n\mathcal F_n)\), upward convergence of conditional
expectations \cite[Theorems~4.6.8--4.6.9]{durrett} gives
\[
 \PP(A\mid\mathcal F_n)\longrightarrow\mathbf1_A
 \quad\text{almost surely}.
\]
On \(A\), both \(\norm{X_n}<j^{-2}\) and
\(\PP(A\mid\mathcal F_n)>1-1/j\) therefore hold for all sufficiently large \(n\).
Bounded convergence yields
\[
 \PP\bigl(A\cap\{\norm{X_n}<j^{-2},\
                 \PP(A\mid\mathcal F_n)>1-1/j\}\bigr)\longrightarrow\PP(A)>0.
\]
Choose an integer \(n_j\ge\max\{j^3,N,n_0\}\), larger than the preceding chosen
time, for which
\[
 B_j=\{\norm{X_{n_j}}<j^{-2},\
                \PP(A\mid\mathcal F_{n_j})>1-1/j\}
\]
has positive probability. This event belongs to \(\mathcal F_{n_j}\).
Moreover, \(A\) implies remaining in the ball from time \(n_j\) onward,
which proves \eqref{g:eq:stay}.

The deterministic estimates are
\begin{align*}
 m_j&\ge j^2\longrightarrow\infty,
 &\frac{m_j}{n_j}&\le\frac1j+\frac1{n_j}\longrightarrow0,\\
 n_js_j&=\kappa n_j(n_j+n_0)^{-\alpha}
       \ge\kappa2^{-\alpha}n_j^{1-\alpha}\ge\kappa2^{-\alpha},
 &\frac{j^{-2}}{m_js_j}
       &\le\frac1{jn_js_j}\le\frac{2^\alpha}{\kappa j}\longrightarrow0.
\end{align*}
Finally, since the steps decrease,
\[
 \max_{n_j\le n\le n_j+m_j}
       \left|\frac{\gamma_n}{s_j}-1\right|
 =1-\left(1+\frac{m_j}{n_j+n_0}\right)^{-\alpha}
 \le\frac{\alpha m_j}{n_j+n_0}\longrightarrow0.
\]
\end{proof}

The number of observed transitions grows, even though the original total step
length \(m_js_j\) need not grow. For \(\alpha=1\), it is of order \(1/j\).
The smaller radius \(j^{-2}\) ensures that endpoint value differences remain
negligible after division by \(m_js_j\).

For the selected event \(B_j\), we use the ordinary conditional average
\[
 \E[Z\mid B_j]=\frac{1}{\PP(B_j)}\int_{B_j}Z\,d\PP.
\]
This is defined whenever \(\int_{B_j}|Z|\,d\PP<\infty\).
Within these selected windows,
\(\E[Z\mid\mathcal F_n,B_j]\) further conditions on the history
\(\mathcal F_n\): it is the conditional expectation given \(\mathcal F_n\)
on \((\Omega,\mathcal F,\PP(\,\cdot\mid B_j))\).
Its identities are understood almost surely on \(B_j\).
For \(n\ge n_j\), the event \(B_j\) already belongs to \(\mathcal F_n\).
Independence of \(\eta_{n+1}\) from \(\mathcal F_n\) therefore gives, for
bounded Borel \(b\),
\[
 \E[b(\eta_{n+1})\mid\mathcal F_n,B_j]
 =\int_{\R^d}b(z)\,N(0,\Sigma)(dz).
\]
Thus future noise retains its Gaussian law after a window is selected,
even when the history up to the current step is given.

Stop recording at the first exit or at the end of the window:
\begin{equation}\label{g:eq:stopping}
 \sigma_j=
 \min\left(
   \{k\in\{1,\ldots,m_j\}:\norm{X_{n_j+k}}\ge j^{-2}\}
       \cup\{m_j\}\right).
\end{equation}
On \(B_j\), the recorded states have indices \(n_j\le n<n_j+\sigma_j\).
The terminal state \(X_{n_j+\sigma_j}\) may lie outside the ball.
It is excluded from the occupation average, but the transition reaching it
is included in both telescoping sums below; see Figure~\ref{g:fig:window}.
\begin{figure}[!htbp]
\centering
\begin{tikzpicture}[>=Stealth,font=\small,x=1cm,y=1cm]
 \draw[->,thick] (0,0)--(12.8,0) node[right] {$n$};
 \draw[dashed] (.2,0)--(3.2,0);
 \draw[blue!65!black,very thick] (3.4,0)--(11.8,0);
 \draw (3.4,.13)--(3.4,-.13) node[below] {$n_j$};
 \draw (11.8,.13)--(11.8,-.13) node[below] {$n_j+m_j$};
 \draw[orange!80!black,thick] (9.4,.13)--(9.4,-.13);
 \node[orange!70!black,above] at (9.4,.17) {$n_j+\sigma_j$};
 \node[above,align=center] at (3.4,.45) {condition on the past\\future stay is likely};
 \draw[<->] (3.4,1.25)--(11.8,1.25) node[midway,above] {$m_j\to\infty$,\quad $m_j/n_j\to0$};
 \node[below,align=center] at (6.5,-.65) {fixed scale $s_j=\gamma_{n_j}$:\quad $\gamma_{n_j+k}/s_j\to1$ uniformly\\
 endpoint contribution vanishes:\quad $j^{-2}/(m_js_j)\to0$};
\end{tikzpicture}
\caption{The stopped window used to obtain expected descent $0$. Average the states from $n_j$ to $n_j+\sigma_j-1$, where $\sigma_j$ is the first exit index or $m_j$. The final update may leave the ball. Horizontal lengths are schematic.}
\label{g:fig:window}
\end{figure}
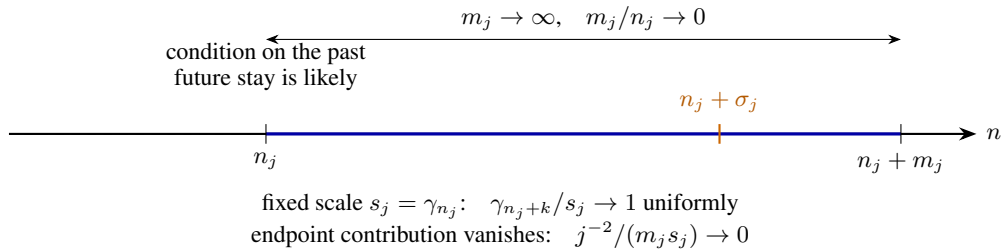

Equation~\eqref{g:eq:stay} implies
\[
 \E[\sigma_j\mid B_j]
 \ge m_j\,\PP(\sigma_j=m_j\mid B_j)
 \ge(1-1/j)m_j
 \ge\frac{m_j}{2}\longrightarrow\infty.
\]
The expected number of records within a selected window therefore diverges.
\mbox{The second-moment constant satisfies}
\[
 M_2=\E\norm{\eta_1}^2=\tr\Sigma,\qquad
 \E[\norm{\eta_{n_j+i}}^2\mid B_j]=M_2\quad(i\ge1).
\]
The second equality again uses independence of future noise from \(B_j\).

For all sufficiently large \(j\),
\(\overline B(0,j^{-2})\subset K_0\). Lemma~\ref{g:lem:extension} then gives
\(V_n=g(X_n)\) at every recorded state on \(B_j\), almost surely.
Throughout a fixed window, we write its actual state explicitly as
\((H_{X_n,s_j},g(X_n))\), with \(n\) denoting the original iteration index.
Define a probability \(\nu_j\) on \(E\) through the following identity
for every bounded Borel test function \(\psi:E\to\R\):
\begin{equation}\label{g:eq:occupation}
 \int_E\psi(H,v)\,\nu_j(dH,dv)
 =\frac{\displaystyle
      \E\!\left[\sum_{n=n_j}^{n_j+\sigma_j-1}
              \psi(H_{X_n,s_j},g(X_n))\,\middle|\,B_j\right]}
        {\E[\sigma_j\mid B_j]}.
\end{equation}
Taking \(\psi=\mathbf1_D\) for Borel \(D\subset E\) defines a nonnegative
measure; countable additivity follows
from monotone convergence. Its mass is
\[
 \nu_j(E)
 =\frac{\E[\sum_{n=n_j}^{n_j+\sigma_j-1}1\mid B_j]}
        {\E[\sigma_j\mid B_j]}=1.
\]
This is the ratio of the expected count in each set of states to the expected
total count, both conditional on \(B_j\). It is not an average of separately
normalized sample paths. The same normalizing factor will apply to both
telescoping identities.

On \(B_j\), before termination, dividing the original update by the fixed scale gives
\[
 \frac{X_{n+1}-X_n}{s_j}
   =\frac{\gamma_n}{s_j}\bigl(-g(X_n)+\eta_{n+1}\bigr).
\]
The difference quotient identities are exact:
\[
 T_{(X_{n+1}-X_n)/s_j}H_{X_n,s_j}=H_{X_{n+1},s_j},\qquad
 \nabla H_{X_n,s_j}\!\left(\frac{X_{n+1}-X_n}{s_j}\right)=g(X_{n+1})
 \quad\text{almost surely}.
\]
The gradient on the right is defined from the global extension \(h\), so the
second identity remains valid if the next point leaves \(K_0\).
Thus the actual transition uses \(P_{\gamma_n/s_j}\); only its first step
has ratio exactly one.

The next lemma combines the transition telescope with its limiting
consequence. The sum of successive changes of a test function reduces to
two endpoint values. Their contribution vanishes after division by the
expected number of recorded steps, and stability of the kernel replaces
the actual ratios by one.

\begin{lemma}\label{g:lem:occupation}
For every \(\psi\in C(E)\), the probabilities in \eqref{g:eq:occupation} satisfy
\[
 \int_E\bigl[(P\psi)(H,v)-\psi(H,v)\bigr]\,\nu_j(dH,dv)\longrightarrow0.
\]
\end{lemma}
\begin{proof}
Fix \(j\) sufficiently large. For \(n_j\le n<n_j+m_j\),
\[
 \{n<n_j+\sigma_j\}
   =\bigcap_{\ell=n_j+1}^{n}\{\norm{X_\ell}<j^{-2}\}
       \in\mathcal F_n,
\]
where the intersection is empty when \(n=n_j\).
The stopping indicator is therefore known at time \(n\).
Recall that \(P=P_1\), and that, for a fixed state \((H,v)\),
\[
 (P_\lambda\psi)(H,v)
 =\int_{\R^d}
   \psi\bigl(T_{\lambda(-v+z)}H,\nabla H(\lambda(-v+z))\bigr)
      \,N(0,\Sigma)(dz).
\]
Thus \(P_\lambda\psi\) averages the test function after one update of ratio
\(\lambda\). Given \(\mathcal F_n\) within \(B_j\), the current state
\((H_{X_n,s_j},g(X_n))\) and the stopping indicator are fixed, while
\(\eta_{n+1}\) still has law \(N(0,\Sigma)\).
Before termination, the update and recentering identities above identify
the next state with the two arguments of \(\psi\) in this integral, with
\(H=H_{X_n,s_j}\), \(v=g(X_n)\), and \(\lambda=\gamma_n/s_j\).
Integrating over the next noise and taking the known indicator outside
the conditional expectation therefore gives
\begin{align*}
 &\E\!\left[
   \mathbf1_{\{n<n_j+\sigma_j\}}
     \psi(H_{X_{n+1},s_j},g(X_{n+1}))
       \,\middle|\,\mathcal F_n,B_j\right]\\
 &\qquad=
 \mathbf1_{\{n<n_j+\sigma_j\}}
 \int_{\R^d}
   \psi\Bigl(
       T_{\frac{\gamma_n}{s_j}(-g(X_n)+z)}H_{X_n,s_j},\\
 &\hspace{9em}
       \nabla H_{X_n,s_j}\bigl(\tfrac{\gamma_n}{s_j}(-g(X_n)+z)\bigr)
       \Bigr)\,N(0,\Sigma)(dz)\\
 &\qquad=
 \mathbf1_{\{n<n_j+\sigma_j\}}
     (P_{\gamma_n/s_j}\psi)(H_{X_n,s_j},g(X_n)).
\end{align*}
The first equality integrates the next noise, replacing \(\eta_{n+1}\)
by the integration variable \(z\); the second is exactly the definition
of \(P_\lambda\psi\) with the three substitutions above.
On the event that recording has already stopped, both sides vanish.
Only the next noise is averaged in this conditional identity; the right-hand
side is still a function of the current random state.
The tower property used here is
\[
 \E\!\left[\E[Z\mid\mathcal F_n,B_j]\,\middle|\,B_j\right]
       =\E[Z\mid B_j],
\]
for every \(Z\) integrable under \(\PP(\,\cdot\mid B_j)\).
Both expectations use this same conditional probability measure: the inner
expectation averages given the current history, and the outer expectation
averages over that history within \(B_j\). Applying this identity to the
stopped next-state test value gives, \mbox{for each fixed \(n\),}
\begin{align*}
 &\E\!\left[\mathbf1_{\{n<n_j+\sigma_j\}}
     (P_{\gamma_n/s_j}\psi)(H_{X_n,s_j},g(X_n))\,\middle|\,B_j\right]\\
 &=\E\!\left[
     \E\!\left[\mathbf1_{\{n<n_j+\sigma_j\}}
       \psi(H_{X_{n+1},s_j},g(X_{n+1}))
         \,\middle|\,\mathcal F_n,B_j\right]\,\middle|\,B_j\right]\\
 &=\E\!\left[\mathbf1_{\{n<n_j+\sigma_j\}}
       \psi(H_{X_{n+1},s_j},g(X_{n+1}))\,\middle|\,B_j\right].
\end{align*}
Subtract the conditional average of the stopped current test value from both
sides and sum over the deterministic indices \(n_j\le n<n_j+m_j\).
The indicators restrict each sum to \(n<n_j+\sigma_j\), giving
\begin{align*}
 &\E\!\left[\sum_{n=n_j}^{n_j+\sigma_j-1}
   \bigl[(P_{\gamma_n/s_j}\psi)(H_{X_n,s_j},g(X_n))
                 -\psi(H_{X_n,s_j},g(X_n))\bigr]\,\middle|\,B_j\right]\\
 &=\E\!\left[\sum_{n=n_j}^{n_j+\sigma_j-1}
   \bigl[\psi(H_{X_{n+1},s_j},g(X_{n+1}))
                 -\psi(H_{X_n,s_j},g(X_n))\bigr]\,\middle|\,B_j\right]\\
 &=\E\!\left[
    \psi(H_{X_{n_j+\sigma_j},s_j},g(X_{n_j+\sigma_j}))
                 -\psi(H_{X_{n_j},s_j},g(X_{n_j}))\,\middle|\,B_j\right].
\end{align*}
All terms are integrable because \(\psi\) is bounded.
Hence
\begin{equation}\label{g:eq:stationarytelescope}
 \begin{split}
 &\frac{\displaystyle
  \E\!\left[\sum_{n=n_j}^{n_j+\sigma_j-1}
   \bigl[(P_{\gamma_n/s_j}\psi)(H_{X_n,s_j},g(X_n))
                 -\psi(H_{X_n,s_j},g(X_n))\bigr]\,\middle|\,B_j\right]}
 {\E[\sigma_j\mid B_j]}\\
 &\quad=
 \frac{\displaystyle
  \E\!\left[
    \psi(H_{X_{n_j+\sigma_j},s_j},g(X_{n_j+\sigma_j}))
                 -\psi(H_{X_{n_j},s_j},g(X_{n_j}))\,\middle|\,B_j\right]}
 {\E[\sigma_j\mid B_j]}.
 \end{split}
\end{equation}
This identity holds for every bounded Borel test function.
Its right-hand side is bounded \mbox{in absolute value by}
\[
 \frac{2\norm{\psi}_\infty}{\E[\sigma_j\mid B_j]}
       \le\frac{4\norm{\psi}_\infty}{m_j}\longrightarrow0.
\]
Here \(\norm{\psi}_\infty=\sup_{(H,v)\in E}|\psi(H,v)|\).
For the fixed continuous \(\psi\) in the statement,
\eqref{g:eq:occupation} and \eqref{g:eq:stationarytelescope} imply
\begin{align*}
 \left|\int_E(P\psi-\psi)\,d\nu_j\right|
 &\le\frac{4\norm{\psi}_\infty}{m_j}\\
 &\quad+\max_{n_j\le n<n_j+m_j}
       \sup_{(H,v)\in E}
       |(P_{\gamma_n/s_j}\psi)(H,v)-(P\psi)(H,v)|.
\end{align*}
The second term tends to zero by \eqref{g:eq:windowlimits} and the joint
continuity in Proposition~\ref{g:prop:kernel}, applied on a compact interval of
ratios containing one. This proves the conclusion.
\end{proof}

Lemma~\ref{g:lem:occupation} gives approximate stationarity of the occupation
measures. The value telescope must now be controlled for these same measures.
Unlike the bounded test function above, \(h\) need not be bounded globally.
We therefore use its Lipschitz bound and account explicitly for the last
step, whose endpoint may be outside the ball.

\begin{lemma}\label{g:lem:occupation-descent}
The same probabilities \(\nu_j\) satisfy
\[
 \int_E d(H,v)\,\nu_j(dH,dv)\longrightarrow0.
\]
\end{lemma}
\begin{proof}
Fix \(j\) sufficiently large. For every pre-termination update on \(B_j\), the definition
of the actual difference quotient gives
\[
 H_{X_n,s_j}(0)
  -H_{X_n,s_j}\!\left(\frac{X_{n+1}-X_n}{s_j}\right)
       =\frac{h(X_n)-h(X_{n+1})}{s_j}.
\]
Recall that \(d=d_1\), where
\[
 d_\lambda(H,v)
 =\int_{\R^d}
    [H(0)-H(\lambda(-v+z))]\,N(0,\Sigma)(dz).
\]
This is the expected decrease of \(H\) after one update of ratio
\(\lambda\) from the fixed state \((H,v)\).
Before termination,
\((X_{n+1}-X_n)/s_j=(\gamma_n/s_j)(-g(X_n)+\eta_{n+1})\).
Given \(\mathcal F_n\) within \(B_j\), the current state is fixed and the
next noise still has law \(N(0,\Sigma)\). Substituting this displacement
in the difference quotient identity above and integrating over the noise
therefore gives \(d_{\gamma_n/s_j}(H_{X_n,s_j},g(X_n))\).
The stopping indicator is \(\mathcal F_n\)-measurable, so it can be taken
outside the conditional expectation:
\begin{align*}
 &\E\!\left[
 \mathbf1_{\{n<n_j+\sigma_j\}}
       \frac{h(X_n)-h(X_{n+1})}{s_j}
       \,\middle|\,\mathcal F_n,B_j\right]\\
 &\qquad=
 \mathbf1_{\{n<n_j+\sigma_j\}}
 \int_{\R^d}\Bigl[
   H_{X_n,s_j}(0)\\
 &\hspace{9em}
   -H_{X_n,s_j}\bigl(\tfrac{\gamma_n}{s_j}(-g(X_n)+z)\bigr)
   \Bigr]\,N(0,\Sigma)(dz)\\
 &\qquad=
 \mathbf1_{\{n<n_j+\sigma_j\}}
       d_{\gamma_n/s_j}(H_{X_n,s_j},g(X_n)).
\end{align*}
The first equality uses the exact difference quotient identity and the
conditional Gaussian law. The second uses the definition of \(d_\lambda\)
with \(H=H_{X_n,s_j}\), \(v=g(X_n)\), and \(\lambda=\gamma_n/s_j\).
When recording has already stopped, both sides are zero.
The stopped terms are integrable under \(\PP(\,\cdot\mid B_j)\):
\mbox{on \(B_j\), before termination,}
\[
 \frac{|h(X_n)-h(X_{n+1})|}{s_j}
  \le L\frac{\gamma_n}{s_j}
         (L+\norm{\eta_{n+1}})
  \le L(L+\norm{\eta_{n+1}}).
\]
Apply the same tower property to the stopped value difference. The preceding
integrability bound justifies both conditional averages, and for each fixed
\(n\) it gives
\begin{align*}
 &\E\!\left[\mathbf1_{\{n<n_j+\sigma_j\}}
       d_{\gamma_n/s_j}(H_{X_n,s_j},g(X_n))\,\middle|\,B_j\right]\\
 &=\E\!\left[
     \E\!\left[\mathbf1_{\{n<n_j+\sigma_j\}}
       \frac{h(X_n)-h(X_{n+1})}{s_j}
         \,\middle|\,\mathcal F_n,B_j\right]\,\middle|\,B_j\right]\\
 &=\frac1{s_j}\E\!\left[\mathbf1_{\{n<n_j+\sigma_j\}}
       [h(X_n)-h(X_{n+1})]\,\middle|\,B_j\right].
\end{align*}
Summing over \(n_j\le n<n_j+m_j\), then using the stopping indicators
to write the \mbox{finite stopped sums, gives}
\begin{align*}
 &\E\!\left[\sum_{n=n_j}^{n_j+\sigma_j-1}
          d_{\gamma_n/s_j}(H_{X_n,s_j},g(X_n))\,\middle|\,B_j\right]\\
 &=\frac1{s_j}\E\!\left[
       \sum_{n=n_j}^{n_j+\sigma_j-1}[h(X_n)-h(X_{n+1})]
       \,\middle|\,B_j\right]\\
 &=\frac1{s_j}\E\!\left[
       h(X_{n_j})-h(X_{n_j+\sigma_j})\,\middle|\,B_j\right].
\end{align*}
Consequently,
\begin{equation}\label{g:eq:valuetelescope}
 \begin{split}
 &\frac{\displaystyle
 \E\!\left[\sum_{n=n_j}^{n_j+\sigma_j-1}
       d_{\gamma_n/s_j}(H_{X_n,s_j},g(X_n))\,\middle|\,B_j\right]}
 {\E[\sigma_j\mid B_j]}\\
 &\qquad=
 \frac{\E[h(X_{n_j})-h(X_{n_j+\sigma_j})\mid B_j]}
      {s_j\E[\sigma_j\mid B_j]}.
 \end{split}
\end{equation}

On \(B_j\), both the initial point and the point just before termination lie
in the ball. Including the last update, and using decreasing steps, gives
almost surely on this event
\begin{align*}
 \norm{X_{n_j+\sigma_j}-X_{n_j}}
 &\le\norm{X_{n_j+\sigma_j}-X_{n_j+\sigma_j-1}}
          +\norm{X_{n_j+\sigma_j-1}}+\norm{X_{n_j}}\\
 &\le\gamma_{n_j+\sigma_j-1}
             (L+\norm{\eta_{n_j+\sigma_j}})+2j^{-2}\\
 &\le s_j\left(L+\max_{1\le i\le m_j}
                         \norm{\eta_{n_j+i}}\right)+2j^{-2}.
\end{align*}
Cauchy--Schwarz under \(\PP(\,\cdot\mid B_j)\) and the unchanged second
moments give
\begin{align*}
 \E\!\left[
       \max_{1\le i\le m_j}\norm{\eta_{n_j+i}}\,\middle|\,B_j\right]
 &\le\left(
       \E\!\left[\sum_{i=1}^{m_j}
                       \norm{\eta_{n_j+i}}^2\,\middle|\,B_j\right]\right)^{1/2}\\
 &=\sqrt{m_jM_2}.
\end{align*}
Thus the Lipschitz bound and
\(\E[\sigma_j\mid B_j]\ge m_j/2\) yield
\begin{equation}\label{g:eq:boundarybound}
 \begin{split}
 \left|
 \frac{\E[h(X_{n_j})-h(X_{n_j+\sigma_j})\mid B_j]}
      {s_j\E[\sigma_j\mid B_j]}\right|
 &\le\frac{L}{\E[\sigma_j\mid B_j]}
       \left(\frac{2j^{-2}}{s_j}+L+\sqrt{m_jM_2}\right)\\
 &\le\frac{4L j^{-2}}{m_js_j}
      +\frac{2L^2}{m_j}
      +\frac{2L\sqrt{M_2}}{\sqrt{m_j}}
      \longrightarrow0.
 \end{split}
\end{equation}
The first term vanishes by the spatial-scale condition in
\eqref{g:eq:windowlimits}; the remaining terms control the possible exit step
using only the second moment of the noise.

Finally, uniformly over \((H,v)\in E\),
\begin{align*}
 |d_\lambda(H,v)-d(H,v)|
 &\le L|\lambda-1|
       \int_{\R^d}\norm{-v+z}\,N(0,\Sigma)(dz)\\
 &\le L(L+\sqrt{M_2})|\lambda-1|.
\end{align*}
Applying this estimate to the recorded states in \eqref{g:eq:occupation},
and then using \eqref{g:eq:valuetelescope}--\eqref{g:eq:boundarybound}, gives
\begin{align*}
 \left|\int_Ed(H,v)\,\nu_j(dH,dv)\right|
 &\le L(L+\sqrt{M_2})
       \max_{n_j\le n<n_j+m_j}\left|\frac{\gamma_n}{s_j}-1\right|\\
 &\quad+\frac{4L j^{-2}}{m_js_j}
       +\frac{2L^2}{m_j}
       +\frac{2L\sqrt{M_2}}{\sqrt{m_j}}\longrightarrow0.
\end{align*}
\end{proof}

The two lemmas concern the same occupation measures: one controls their
change under the unit-step kernel, and the other controls their mean descent
for that kernel. Compactness now provides a common limit with both properties.

\subsection{Proof of zero expected descent}\label{g:sec:convergence}

Compactness of the state space yields a common weak limit of the stopped occupation probabilities. Shrinking centers and scales place its support in \(\mathcal C\times\overline B(0,L)\), and continuity of \(d\) and \(P\psi\) transfers both conclusions of Lemmas~\ref{g:lem:occupation-descent} and~\ref{g:lem:occupation}: mean descent \(0\) and stationarity.

\begin{proof}[Proof of Lemma~\ref{g:prop:zero}]
Use the windows from Lemma~\ref{g:lem:window} and the probabilities
\(\nu_j\) in \eqref{g:eq:occupation}. The state space \(E\) is compact by
Lemma~\ref{g:lem:profilecompactness}. Prohorov's theorem
\cite[Theorem~5.1]{billingsley} therefore gives a subsequence, still indexed
by \(j\), for which \(\nu_j\Rightarrow\mu\).

Let \(\operatorname{dist}(H,\mathcal C)=\inf_{J\in\mathcal C}d_0(H,J)\),
where \(d_0\) is the metric for local uniform convergence defined earlier.
We first show
\[
 \sup_{\norm{x}<j^{-2}}
       \operatorname{dist}(H_{x,s_j},\mathcal C)\longrightarrow0.
\]
Otherwise, along a subsequence there would be centers
\(\norm{x_j}<j^{-2}\) whose corresponding distances exceed a fixed
\(\epsilon>0\). Compactness of \(\mathcal D\) gives a further locally uniformly
convergent subsequence of \(H_{x_j,s_j}\). Since \(x_j\to0\) and
\(s_j=\gamma_{n_j}\to0\), its limit belongs to \(\mathcal C\), contradicting
the distance bound.
All states entering \(\nu_j\) have such centers. Since the distance function
is continuous and bounded,
\[
 \int_E\operatorname{dist}(H,\mathcal C)\,\mu(dH,dv)
  =\lim_j\int_E\operatorname{dist}(H,\mathcal C)\,\nu_j(dH,dv)=0.
\]
The integrand is nonnegative and \(\mathcal C\) is closed, so
\(\mu(\mathcal C\times\overline B(0,L))=1\).

For \(\psi\in C(E)\), Proposition~\ref{g:prop:kernel} gives \(P\psi\in C(E)\).
Lemma~\ref{g:lem:occupation} and weak convergence imply
\[
 \int_E(P\psi-\psi)\,d\mu
       =\lim_j\int_E(P\psi-\psi)\,d\nu_j=0.
\]
Continuous tests determine Borel probabilities on the compact metric space
\(E\), hence \(\mu P=\mu\).

The function \(d\) is also bounded and continuous by
Proposition~\ref{g:prop:kernel}. For the same weak limit,
Lemma~\ref{g:lem:occupation-descent} gives
\[
 \int_E d(H,v)\,\mu(dH,dv)
       =\lim_j\int_E d(H,v)\,\nu_j(dH,dv)=0.
\]
This proves Lemma~\ref{g:prop:zero}.
\end{proof}

In the proof of the avoidance theorem, the positive-descent result applies
to every stationary probability on this same limiting state space.
It therefore applies to the law constructed here, whose mean descent is zero;
no uniqueness of the stationary law is needed.

\begin{proof}[Proof of Theorem~\ref{g:thm:main}]
Suppose $\PP(X_n\to0)>0$. Take the stationary law $\mu$ supplied by
Lemma~\ref{g:prop:zero}. Lemma~\ref{g:prop:negative} gives
\[
 0=\int d(H,v)\,d\mu(H,v)\ge a^2>0,
\]
a contradiction. Translating back to $x_*$ proves the theorem.
\end{proof}

\subsection{Finite sets and generic linear tilts}\label{g:sec:scope}
\begin{proof}[Proof of Corollary~\ref{g:cor:finite-generic}]
For finite $A$, apply Theorem~\ref{g:thm:main} to each point and take
a finite union. For the generic statement, observe that
$\{x:0\in\pc f_u(x)\}=\{x:u\in\pc f(x)\}$.
The definable graph of $\pc f$ has dimension $d$, so its nonempty
fibers over $u$ are $0$-dimensional, hence finite, outside a
definable set of dimension less than $d$
\cite[Theorem~3.5 and Corollary~4.4]{dl}; the extension to definable
functions is stated in that paper's introduction. Thus $R_u$ is
finite, and every $x\in R_u$ satisfies $0\notin\pl f_u(x)$.
The first assertion applied to $f_u$ completes the proof.
\end{proof}

Definability alone does not ensure finiteness of the Clarke critical
set: $f(x_1,x_2)=x_1^2$ has the critical line $\{0\}\times\R$.
The generic tilt supplies the finiteness used above. The definition
of sharply repulsive points is given in Appendix~\ref{p:sec:active}.
The corollary addresses sharply repulsive-point avoidance; general
active-strict-saddle avoidance and generic convergence to local minima
require further conclusions.

\section{Experimental details and additional results}\label{app:experiments}
We supplement the noise comparison in Section~\ref{sec:experiments} with individual trajectories and escape counts, using the same objective, initialization, step sizes, and covariance $\sigma^2I_2$, where $\sigma=0.15$. Uniform noise is sampled by area on the disk of radius $2\sigma$; Gaussian, logistic, and Laplace coordinates are independent, with standard deviation $\sigma$ and respective logistic and Laplace scales $\sigma\sqrt3/\pi$ and $\sigma/\sqrt2$.

Figure~\ref{fig:fixed-trajectories} shows the first eight trajectory indices from each of the four repeats, giving $32$ paths per law. These indices were fixed independently of the outcomes. The individual paths complement the aggregate statistics in Figure~\ref{fig:experiments} by showing the variation in escape times. The dashed line marks radius $0.05$.

\begin{figure}[htbp]
\centering
\includegraphics[width=\linewidth]{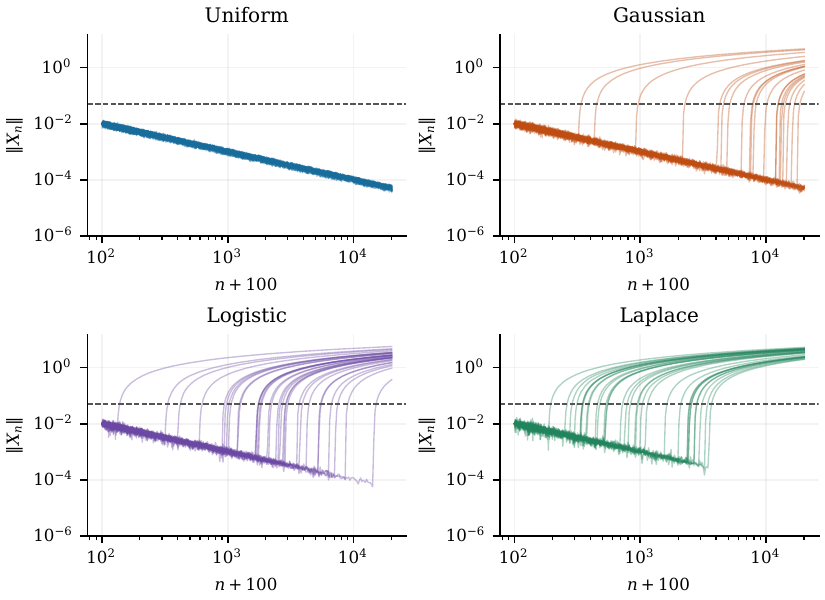}
\caption{Individual trajectories under the four noise laws. Each panel shows $32$ paths with indices fixed in advance, plotted at the recorded time indices. The dashed line marks radius $0.05$.}
\label{fig:fixed-trajectories}
\end{figure}

Table~\ref{tab:escape-counts} counts trajectories satisfying $\norm{X_n}>0.05$ at least once within $20{,}000$ updates, using all updates of all $1{,}024$ trajectories per law. These are finite-time escapes; an observed exit need not be permanent.

\begin{table}[htbp]
\centering
\caption{Trajectories escaping radius $0.05$ within $20{,}000$ updates.}
\label{tab:escape-counts}
\begin{tabular}{lc}
\toprule
Noise law & Escapes / total trajectories \\
\midrule
Uniform & $0/1{,}024$ \\
Gaussian & $766/1{,}024$ \\
Logistic & $1{,}023/1{,}024$ \\
Laplace & $1{,}024/1{,}024$ \\
\bottomrule
\end{tabular}
\end{table}

\end{document}